\documentclass[11pt,a4paper,oneside]{article}
\usepackage[utf8]{inputenc}
\usepackage[dvipsnames]{xcolor}
\usepackage[pdftitle={Scaling and local limits of Baxter permutations and bipolar orientations through coalescent-walk processes},
pdfauthor={Jacopo Borga, Mickaël Maazoun},
colorlinks=true,linkcolor=NavyBlue,urlcolor=RoyalBlue,citecolor=PineGreen,bookmarks=true,bookmarksopen=true,bookmarksopenlevel=2,unicode=true,linktocpage]{hyperref}
\usepackage{graphicx}
\usepackage{a4wide}
\usepackage{amsmath}
\usepackage{authblk}
\usepackage{amsthm}
\usepackage{amsfonts}
\usepackage{amssymb}
\usepackage{mathtools}
\usepackage{mathrsfs}
\usepackage{stmaryrd}
\usepackage{wasysym}
\usepackage[mathscr]{eucal}
\usepackage{dsfont}
\usepackage{graphicx}
\usepackage{todonotes}
\usepackage{enumitem}
\usepackage{subcaption}
\usepackage[capitalize]{cleveref}
\usepackage{amsmath,bm}
\setlist{nolistsep}
\usepackage{pst-node}
\usepackage{tikz-cd} 

\newcommand{\R}{\mathbb{R}}

\newcommand{\Z}{\mathbb{Z}}
\newcommand\indep{\perp\!\!\!\perp}
\newcommand{\eps}{\varepsilon}

\DeclareMathOperator{\E}{\mathbb{E}}
\DeclareMathOperator{\idf}{\mathds{1}}
\DeclareMathOperator{\Prob}{\mathbb{P}}
\renewcommand{\P}{\Prob}
\DeclareMathOperator{\Var}{Var}

\DeclareMathOperator{\pat}{pat}

\DeclareMathOperator{\Leb}{Leb}

\DeclareMathOperator{\Id}{Id}

\DeclareMathOperator{\sgn}{sgn}
\DeclareMathOperator{\std}{std}

\DeclareMathOperator{\wcp}{WC}
\DeclareMathOperator{\fortree}{LFor}

\DeclareMathOperator{\labtree}{LTr}
\DeclareMathOperator{\cpbp}{CP}
\DeclareMathOperator{\bow}{OW}
\DeclareMathOperator{\bobp}{OP}
\DeclareMathOperator{\Perm}{Perm}

\DeclareMathOperator{\wpc}{WPC}
\DeclareMathOperator{\pcw}{PCW}

\def\leqsi{\preccurlyeq_{\sigma,i}}
\def\coc{c\text{-}occ}

\newcommand{\DualForest}{\mathrm{DualF}}
\newcommand{\exc}{\mathrm{exc}}
\newcommand{\Walks}{\mathfrak W}
\newcommand{\Coals}{\mathfrak C}
\newcommand{\Perms}{\mathfrak S}
\newcommand{\Maps}{\mathfrak m}
\newcommand{\Steps}{A}
\newcommand{\Inverse}{\Theta}

\newcommand{\conti}[1]{{\bm{\mathscr #1}}}
\newcommand{\solution}{F}
\makeatletter
\DeclareRobustCommand{\cev}[1]{%
	\mathpalette\do@cev{#1}%
}
\newcommand{\do@cev}[2]{%
	\fix@cev{#1}{+}%
	\reflectbox{$\m@th#1\vec{\reflectbox{$\fix@cev{#1}{-}\m@th#1#2\fix@cev{#1}{+}$}}$}%
	\fix@cev{#1}{-}%
}
\newcommand{\fix@cev}[2]{%
	\ifx#1\displaystyle
	\mkern#23mu
	\else
	\ifx#1\textstyle
	\mkern#23mu
	\else
	\ifx#1\scriptstyle
	\mkern#22mu
	\else
	\mkern#22mu
	\fi
	\fi
	\fi
}

\makeatother
\newtheorem{theorem}{Theorem}[section]
\newtheorem{lemma}[theorem]{Lemma}
\newtheorem{proposition}[theorem]{Proposition}
\newtheorem{corollary}[theorem]{Corollary}

\theoremstyle{definition}
\newtheorem{definition}[theorem]{Definition}

\theoremstyle{remark}
\newtheorem{remark}[theorem]{Remark}
\newtheorem{observation}[theorem]{Observation}
\newtheorem{exmp}[theorem]{Example}
\newtheorem{conjecture}[theorem]{Conjecture}

\author[1]{Jacopo Borga\thanks{\href{mailto:jborga@stanford.edu}{jborga@stanford.edu}}}
\author[2]{Mickaël Maazoun\thanks{\href{mailto:mickael.maazoun@ens-lyon.fr}{mickael.maazoun@ens-lyon.fr}}}
\affil[1]{\small{Institut für Mathematik, Universität Zürich, Switzerland}}
\affil[2]{\small{Université de Lyon, ENS de Lyon, Unité de mathématiques pures et appliquées UMR 5669 CNRS, France}}

\title{\vspace{-2.5cm}Scaling and local limits of Baxter permutations and bipolar orientations through coalescent-walk processes}

\date{}

\makeatletter
\newcommand{\subjclass}[2][1991]{%
	\let\@oldtitle\@title%
	\gdef\@title{\@oldtitle\footnotetext{#1 \emph{Mathematics subject classification.} #2}}%
}
\newcommand{\keywords}[1]{%
	\let\@@oldtitle\@title%
	\gdef\@title{\@@oldtitle\footnotetext{\emph{Key words and phrases.} #1}}%
}
\makeatother


\subjclass[2010]{60C05, 60G50, 05A05, 05C10, 34K50. \\
\hphantom{|,} \emph{Additional notes.} An extended abstract of this preprint is available at \cite{borga_et_al:LIPIcs:2020:12037}.}

\begin{document}

	\maketitle
	
	\vspace*{-2.5em}
	
\begin{abstract}
	Baxter permutations, plane bipolar orientations, and a specific family of walks in the non-negative quadrant, called \emph{tandem walks}, are well-known to be related to each other through several bijections. 	
	We introduce a further new family of discrete objects, called \emph{coalescent-walk processes} and we relate it to the three families mentioned above.  
		
	We prove joint Benjamini--Schramm convergence (both in the annealed and quenched sense) for uniform objects in the four families. Furthermore, we explicitly construct a new random measure on the unit square,
	called the \emph{Baxter permuton} and we show that it is the scaling limit (in the permuton sense) of uniform Baxter permutations. 
	In addition, we relate the limiting objects of the four families to each other, both in the local and scaling limit case.
	
	The scaling limit result is based on the convergence of the trajectories of the coalescent-walk process to the \emph{coalescing flow} -- in the terminology of Le Jan and Raimond (2004) -- of a perturbed version of the Tanaka stochastic differential equation. 
	Our scaling result entails joint convergence of the tandem walks of a plane bipolar orientation and its dual, 
	%extending the main result of Gwynne, Holden, Sun (2016), and 
	giving an alternative answer to Conjecture 4.4 of Kenyon, Miller, Sheffield, Wilson (2019) compared to the one of Gwynne, Holden, Sun (2016).

		\end{abstract}
	
\begin{figure}[htbp]
	\centering
		\includegraphics[scale=0.45]{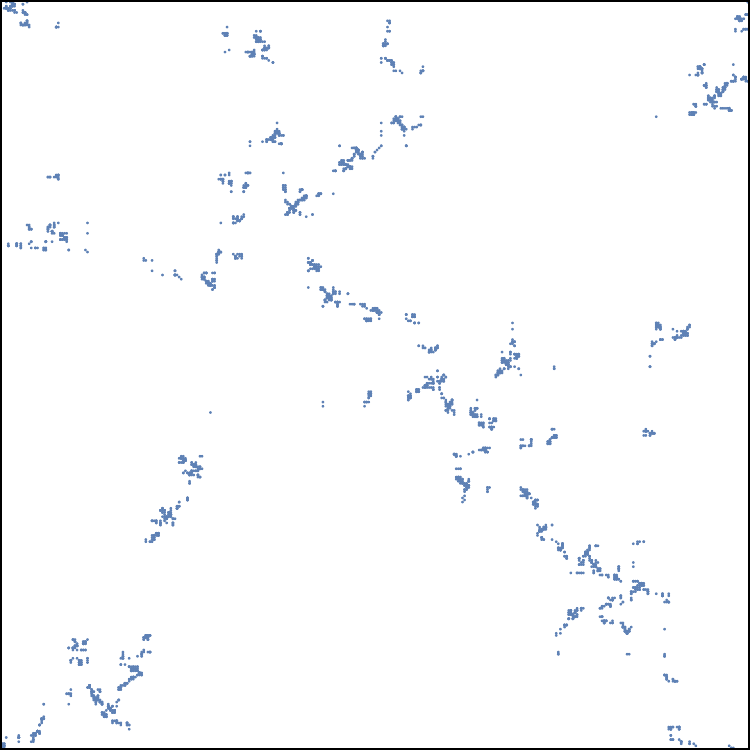}
		\hspace{1.9 cm}
		\includegraphics[scale=0.45]{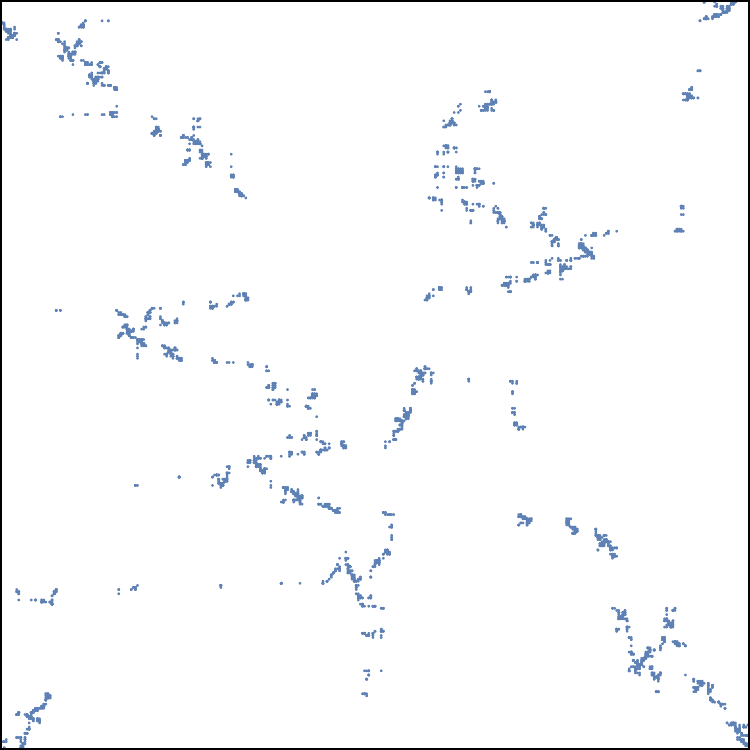}\\
		\caption{The diagrams of two uniform Baxter permutations of size 3253 and 4520. 	
		The underlying generating algorithm is discussed in \cref{sect:simulations}.
		 \label{fig:Baxter_perm_3253}}
\end{figure}

\tableofcontents
		
\section{Introduction and main results}

In the last 30 years, several bijections between Baxter permutations, plane bipolar orientations and certain walks in the plane have been discovered. These relations between discrete objects of different nature are a \emph{beautiful piece of combinatorics}\footnote{Quoting the abstract of \cite{MR2763051}.}, that we aim at investigating from a more probabilistic point of view.

These bijective results come from the enumerative works of Gire~\cite{gire1993arbres} and later Bousquet-M\'{e}lou~\cite{MR2028288}, where they explored the connection between Baxter permutations and generating trees with two-dimensional labels.
Bousquet-M\'{e}lou noticed that Baxter permutations were equinumerous to plane bipolar orientations. A remarkable bijection (denoted $\bobp$ in the present paper) between plane bipolar orientations with $n$ edges and Baxter permutations of size
$n$ was then given by Bonichon, Bousquet-Mélou and Fusy~\cite{MR2734180}. 

Later Felsner, Fusy, Noy and Orden~\cite{MR2763051} gave a unified presentation of some other (partially already-known) bijections between Baxter permutations, 2-orientations of planar quadrangulations, certain pairs of binary trees, and triples of non-intersecting lattice paths.

Kenyon, Miller, Sheffield and Wilson~\cite{MR3945746} introduced a bijection (denoted $\bow$ in the present paper) between plane bipolar orientations and a family of two-dimensional walks in the non-negative quadrant. 
The latter has been used in \cite{bousquet2019plane} to enumerate plane bipolar orientations together with the number of faces of degree $r$ for every $r\in\Z_{>0}$.

In this paper we explore local and scaling limits of some of these objects and we study the relations between their limits. Indeed, since these objects are related by several bijections at the discrete level, we expect that most of the relations among them also hold in the ``limiting discrete and continuous worlds''. 

\medskip

In the next three sections we introduce the precise definitions of the objects involved in our work and we describe some of the bijections mentioned above.

\subsection{Baxter permutations and permuton convergence}\label{sect:perm_sect}
Baxter permutations were introduced by Glen Baxter in 1964 \cite{MR0184217} to study fixed points of commuting functions. A permutation $\sigma$ is Baxter if it is not possible to find $i < j < k$ such that $\sigma(j+1) < \sigma(i) < \sigma(k) < \sigma(j)$ or $\sigma(j) < \sigma(k) < \sigma(i) < \sigma(j+1)$. Baxter permutations are well-studied from a combinatorial point of view by the \textit{permutation patterns} community (see for instance \cite{MR0250516,MR491652,MR555815,MR2028288,MR3882946}). They are a particular example of family of permutations avoiding \textit{vincular patterns} (see \cite{MR2901166} for more details).
We denote by $\mathcal P$ the set of Baxter permutations\footnote{We also denote by $\mathcal{P}_n$ the set of Baxter permutations of size $n$. This convention will be used for all combinatorial classes studied in the paper.}.

The study of random permutations, especially uniform permutations in \emph{permutation classes}, which are families of permutations avoiding classical patterns, is an emerging topic at the interface of combinatorics and
discrete probability theory. There are two main approaches to it: the first is the study
of statistics on permutations, and the second, more recent, looks for limits of permutations
themselves. For instance, one can study the shape of the rescaled diagram of a random permutation (i.e.\ the sets of points of
the Cartesian plane at coordinates $(i, \sigma(i))$)  using the formalism of \textit{permutons}, developed by \cite{hoppen2013limits}. This approach is a rapidly developing field in discrete probability theory, see for instance \cite{madras2010random, mp, madras_monotone, hoffman2017pattern, bassino2018brownian, bassino2019scaling, bassino2017universal, borga2019almost, hoffman2019scaling, maazoun, borga2019square, borga2018localsubclose, kenyon2015permutations}.

A permuton $\mu$ is a Borel probability measure on the unit square $[0,1]^2$ with uniform marginals, that is 
$\mu( [0,1] \times [a,b] ) = \mu( [a,b] \times [0,1] ) = b-a,$
for all $0 \le a \le b\le 1$. Any permutation $\sigma$ of size $n \ge 1$ may be interpreted as the permuton $\mu_\sigma$ given by the sum of Lebesgue area measures
$$\mu_\sigma(A)= n \sum_{i=1}^n \Leb\big([(i-1)/n, i/n]\times[(\sigma(i)-1)/n,\sigma(i)/n]\cap A\big),$$
for all Borel measurable sets $A$ of $[0,1]^2$.

Let $\mathcal M$ be the set of permutons equipped with the topology of weak convergence of measures: a sequence of (deterministic) permutons $(\mu_n)_n$ converges to $\mu$ if 
$
\int_{[0,1]^2} f d\mu_n \to \int_{[0,1]^2} f d\mu,
$
for every (bounded and) continuous function $f: [0,1]^2 \to \mathbb{R}$. With this topology, $\mathcal M$ is a compact metric space (we refer the reader to \cite[Section 2]{bassino2017universal} for a complete introduction to permutons).

A sequence of random permutations $\bm \sigma_n$ converges in distribution in the permuton sense, if the associated sequence of permutons $\mu_{\bm \sigma_n}$ converges.
Permuton convergence is a statement about the first-order geometry of the diagram of $\bm \sigma_n$. Nevertheless, permuton convergence is equivalent to joint convergence in distribution of all \textit{pattern density} statistics (see \cite[Theorem 2.5]{bassino2017universal}). 

Permuton convergence was investigated for some remarkable subclasses of Baxter permutations. Separable permutations, i.e.\ permutations avoiding the two classical patterns $2413$ and $3142$, converge to the \textit{Brownian separable permuton} \cite{bassino2018brownian}. This result provided the first example of a sequence of uniform permutations in a class converging to a non-deterministic permuton. Dokos and Pak \cite{MR3238333} explored the expected limit shape of the so-called \emph{doubly alternating Baxter
permutations}. In their article they claimed that \emph{``it would be interesting to compute the limit shape of random Baxter permutations''}. The present paper answers this open question (see \cref{fig:Baxter_perm_3253} for some simulations of the diagram of uniform Baxter permutations of large size).

\subsection{Plane bipolar orientations, tandem walks, and bijections with Baxter permutations}\label{sect:discr_obj}
\textit{Plane bipolar orientations}, or \emph{bipolar orientations} for short, are planar maps (i.e.\ connected graphs properly embedded in the plane up to continuous deformation) equipped with an acyclic orientation of the edges with exactly one source (i.e.\ a vertex with only outgoing edges) and one sink (i.e.\ a vertex with only incoming edges), both on the outer face. We denote by $\mathcal{O}$ the set of bipolar orientations. The size of a bipolar orientation $m$ is its number of edges and it is denoted by $|m|$.

Every bipolar orientation can be plotted in the plane with every edge oriented from bottom to top (this is a consequence for instance of \cite[Proposition 1]{MR2734180}; see the black map on the left-hand side of \cref{fig:bip_orient} for an example). We think of the outer face as split in two: the \textit{left outer face}, and the \textit{right outer face}. The orientation of the edges around each vertex/face is constrained: we sum up these local constraints, setting some vocabulary, on the right-hand side of \cref{fig:bip_orient}. We call indegree/outdegree of a vertex the number of incoming/outgoing edges around this vertex. We call left degree (resp.\ right degree) of an inner face the number of left (resp.\ right) edges around that face.

\begin{figure}[htbp]
	\centering
	\includegraphics[scale=0.8]{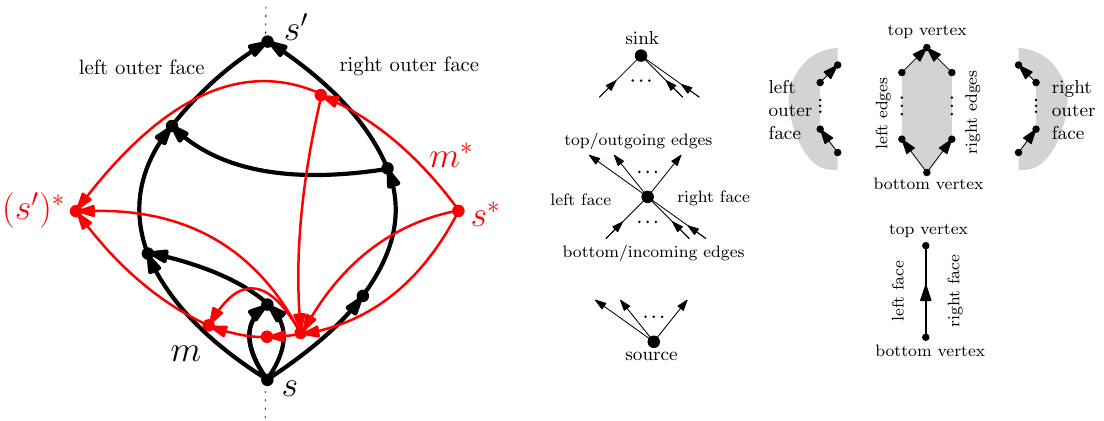}\\
	\caption{On the left-hand side, in black, a bipolar orientation $m$ of size 10 drawn with every edge oriented from bottom to top. In red, its dual map $m^*$ (defined below), drawn with every edge oriented from right to left. On the right-hand side, the behavior of the orientation around each vertex/face/edge. Note for instance that in the clockwise ordering around each vertex different from the source and the sink there are top/outgoing edges, a right face, bottom/incoming edges, and a left face.\label{fig:bip_orient}}
\end{figure}

The \emph{dual map} $m^*$ of a bipolar orientation $m$ (called the \textit{primal}) is obtained by putting a vertex in each face of $m$, and an edge between two faces separated by an edge in $m$, oriented from the right face to the left face. The primal right outer face becomes the dual source, and the primal left outer face becomes the dual sink. Then $m^*$ is also a bipolar orientation (see the left-hand side of \cref{fig:bip_orient}).
The map $m^{**}$ is just $m$ with the orientation reversed, and $m^{****}=m$.

We now define a notion at the heart of the two bijections $\bow$ and $\bobp$ mentioned before.
Let $m$ be a bipolar orientation. Disconnecting every incoming edge but the rightmost one at every vertex turns the map $m$ into a plane tree $T(m)$ rooted at the source, which we call the \textit{down-right tree} of the map (see the left-hand side of \cref{fig:bip_orient _with_trees} for an example). The tree $T(m)$ contains every edge of $m$, and the clockwise contour exploration of $T(m)$ identifies an ordering of the edges of $m$. We denote by $e_1,\ldots,e_{|m|}$ the edges of $m$ in this order (see again \cref{fig:bip_orient _with_trees}). The tree $T(m^{**})$ can be obtained similarly from $m$ by disconnecting every outgoing edge but the leftmost, and is rooted at the sink. The following remarkable facts hold\footnote{This fact is also discussed in a slightly different way in \cite[Section 2.1]{MR3945746}.}:
\textit{The contour exploration of $T(m^{**})$ visits edges of $m$ in the order $e_{|m|},\ldots,e_1$. Moreover, one can draw $T(m)$ and $T(m^{**})$ in the plane, one next to the other, in such a way that the interface between the two trees traces a path, called \emph{interface path}\footnote{Note that the interface path coincides with the clockwise contour exploration of $T(m)$, an example is given in the first two pictures of \cref{fig:bip_orient _with_trees}.}, from the source to the sink visiting edges $e_1,\ldots, e_{|m|}$ in this order (see the middle picture of \cref{fig:bip_orient _with_trees} for an example).}

The following bijection between bipolar orientations and a specific family of two-dimensional walks in the non-negative quadrant was discovered by Kenyon, Miller, Sheffield and Wilson~\cite{MR3945746}.

\begin{definition}\label{defn:KMSW}
Let $n\geq 1$, $m \in \mathcal O_n$. We define $\bow(m)=(X_t,Y_t)_{1\leq t\leq n} \in (\Z_{\geq 0}^2)^n$ as follows: for $1\leq t\leq n$, $X_t$ is the height in the tree $T(m)$ of the bottom vertex of $e_t$ (i.e.\ its distance in $T(m)$ from the source $s$), and $Y_t$ is the height in the tree $T(m^{**})$ of the top vertex of $e_t$ (i.e.\ its distance in $T(m^{**})$ from the sink $s'$).
\end{definition}

An example is given in the right-hand side of \cref{fig:bip_orient _with_trees}.

\begin{figure}[htbp]
		\centering
		\includegraphics[scale=1]{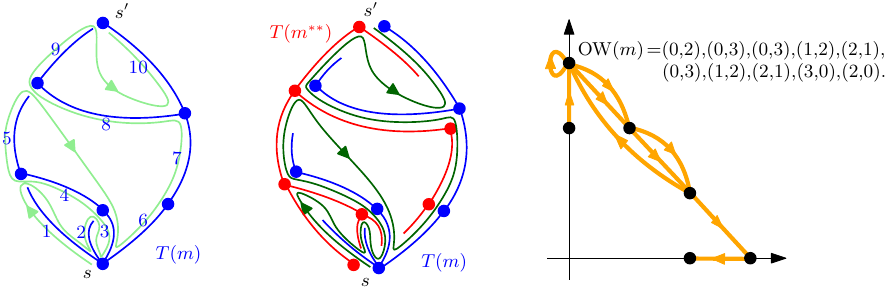}\\
		\caption{On the left-hand side the tree $T(m)$ built by disconnecting the bipolar orientation $m$ from \cref{fig:bip_orient} with the edges ordered according to the exploration process (in light green). In the middle, the two trees $T(m)$ and $T(m^{**})$ with the \emph{interface path} tracking the interface between the two trees (in dark green). On the right-hand side, the two-dimensional walk $\bow(m)$ defined in  \cref{defn:KMSW}. \label{fig:bip_orient _with_trees}}
\end{figure}

\begin{theorem}[Theorem 1 of \cite{MR3945746}] \label{thm:KMSW}
	The mapping $\bow$ is a size-preserving bijection between $\mathcal O$ and the set $\mathcal W$ of walks in the non-negative quadrant $\Z_{\geq 0}^2$ starting on the $y$-axis, ending on the $x$-axis, and with increments in 
	\begin{equation}
	\label{eq:admis_steps}
	\Steps = \{(+1,-1)\} \cup \{(-i,j), i\in \Z_{\geq 0}, j\in \Z_{\geq 0}\}.
	\end{equation}
\end{theorem}

We call $\mathcal W$ the set of \emph{tandem walks}, as done in \cite{bousquet2019plane}. For more explanations on the bijection $\bow$ and the set $\mathcal W$ we refer to \cref{sect:KMSW}.

\medskip

We now introduce a second bijection, fundamental for our results, between bipolar orientations and Baxter permutations, discovered by Bonichon, Bousquet-Mélou and Fusy~\cite{MR2734180}.

\begin{definition}\label{defn:bobp}
Let $n\geq 1, m\in \mathcal O_n$. Recall that every edge of the map $m$ corresponds to its dual edge in the dual map $m^*$. Let $\bobp(m)$ be the only permutation $\pi$ such that for every $1\leq i \leq n$, the $i$-th edge to be visited in the exploration of $T(m)$ corresponds to the $\pi(i)$-th edge to be visited in the exploration of $T(m^*)$.
\end{definition}

An example is given in \cref{fig:bip_orient_and_perm}.

\begin{theorem}[Theorem 2 of \cite{MR2734180}]\label{thm:bobp}
	The mapping $\bobp$ is a size-preserving bijection between the set $\mathcal O$ of bipolar orientations and the set $\mathcal P$ of Baxter permutations.
\end{theorem}

The definition given in \cref{defn:bobp} is a simple reformulation of the bijection presented in \cite{MR2734180}, for more details see \cref{sect:Baxt_bipol}.

\begin{figure}[hbtp]
		\centering
		\includegraphics[scale=0.6]{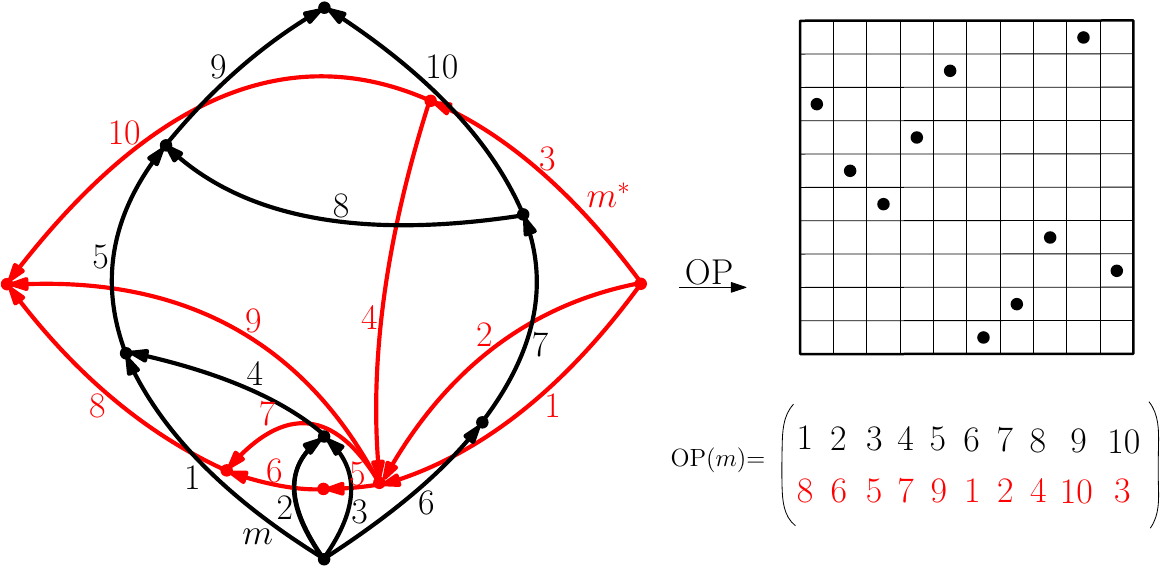}\\
		\caption{A schema explaining the mapping $\bobp$. On the left-hand side, the bipolar orientation $m$ and its dual $m^*$, from \cref{fig:bip_orient}. We plot in black the labeling of the edges of $m$ obtained in \cref{fig:bip_orient _with_trees} and in red the labeling of the edges of $m^*$ obtained by the same procedure. On the right-hand side, the permutation $\bobp(m)$ (together with its diagram) obtained by pairing the labels of the corresponding primal and dual edges between $m$ and $m^*$.  \label{fig:bip_orient_and_perm}}
\end{figure}

\subsection{Coalescent-walk processes}\label{sect:coal_walk_intro}

So far we have considered three families of objects: Baxter permutations ($\mathcal{P}$), tandem walks ($\mathcal W$), and bipolar orientations ($\mathcal{O}$). We saw that they are linked by the mappings $\bow$ and $\bobp$. 

To investigate local and scaling limits of Baxter permutations, it is natural to first prove local and scaling limits results for tandem walks and then try to transfer these convergences to permutations through the mapping $\bobp \circ \bow^{-1}$. However, the definition of this composite mapping makes it not very tractable, and our first combinatorial result is a rewriting of it.

Consider a tandem walk $W = (X,Y) \in \mathcal W_n$ and the corresponding Baxter permutation $\sigma = \bobp \circ \bow^{-1} (W)$. 
We introduce the \textit{coalescent-walk process} driven by $W$. It is a family of discrete walks $Z = \{Z^{(i)}\}_{1\leq i \leq n}$, where $Z^{(i)}=Z^{(i)}_t$ has time indices  $t\in\{i,\ldots,n\}$ and it is informally defined as follows: $Z^{(i)}$ starts at $0$ at time $i$, takes the same steps as $Y$ when it is non-negative, takes the same steps as $-X$ when it is negative unless such a step would force $Z^{(i)}$ to become non-negative. If the latter case happens at time $j$, then $Z^{(i)}$ is forced to coalesce with $Z^{(j)}$ at time $j+1$. For a precise definition we refer the reader to \cref{sec:discrete_coal}. An illustration of a coalescent-walk process is given on the left-hand side of \cref{fig:Coal_process_exemp}.
 \begin{figure}[hbtp]
	\centering
	\includegraphics[scale=0.85]{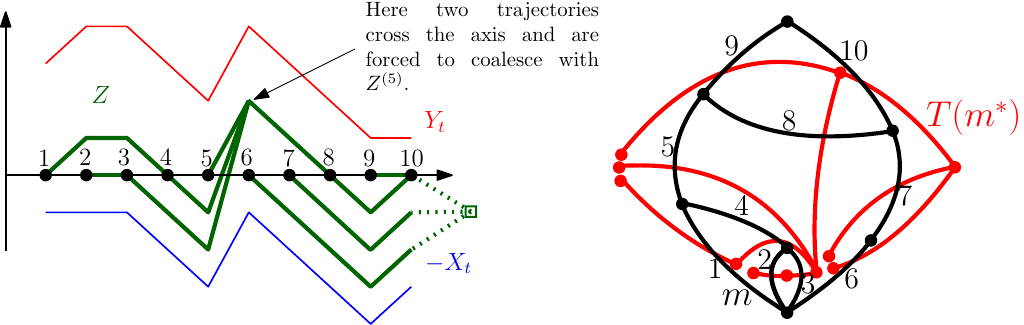}\\
	\caption{The coalescent-walk process $Z = \wcp(W)$ associated with the walk $W=(X,Y) = \bow(m)$. The walk $Y$ is plotted in red and $-X$ is plotted in blue. On the right-hand side, the map $m$ together with the tree $T(m^*)$ drawn in red. \label{fig:Coal_process_exemp}}
\end{figure}
We denote by $\mathcal C_n$ the set of coalescent-walk processes obtained in this way from tandem walks in $\mathcal W_n$, and we define $\wcp:\mathcal W_n \to \mathcal C_n$ to be the mapping that associates a tandem walk $W$ with the corresponding coalescent-walk process $Z$.
 
In a coalescent-walk process, trajectories do not cross but coalesce, hence the name.

\begin{remark}
	To the best of our knowledge this is the first article where \emph{coalescent-walk processes} have been formally considered. We remark that in \cref{sec:discrete_coal} we will present a general definition of \emph{coalescent-walk process}, but then we will mainly focus on coalescent-walk processes associated with tandem walks. Nevertheless, in a  work-in-progress we plan to use similar constructions associated with other families of walks in cones.
	
	We finally mention for completeness that an implicit and very specific version of coalescent-walk process was considered in \cite{GHS} (for more details see \cref{sec:lqg_sde}).
\end{remark}

From a coalescent-walk process one can construct a permutation of the integers. We denote by $\Perms_n$ the set of permutations of size $n$. If $Z\in \mathcal C_n$, we denote $\cpbp(Z)$ the only permutation $\pi \in \Perms_n$ such that for $i,j\in [n]$ with $i<j$, $\pi(i)<\pi(j)$ if and only if $Z^{(i)}_j<0$. We will see that $\cpbp:\mathcal C_n \to \mathcal P_n$, hence $\cpbp(Z)$ is a Baxter permutation. The reader can check that in the case of \cref{fig:Coal_process_exemp} we have $\cpbp(Z) = 8\,6\,5\,7\,9\,1\,2\,4\,10\,3$, which corresponds to $\bobp \circ \bow^{-1}(W)$ (see \cref{fig:bip_orient _with_trees,fig:bip_orient_and_perm}) witnessing an instance of our main combinatorial result.

\begin{theorem}\label{thm:diagram_commutes}
For all $n\in \Z_{>0}$, the following diagram of bijections commutes
 \begin{equation}
 \label{eq:comm_diagram}
 \begin{tikzcd}
 \mathcal{W}_n \arrow{r}{\wcp}  & \mathcal{C}_n \arrow{d}{\cpbp} \\
 \mathcal{O}_n \arrow{u}{\bow} \arrow{r}{\bobp}& \mathcal{P}_n
 \end{tikzcd} \;.
 \end{equation}
\end{theorem}

Note that the mappings involved in the diagram are denoted using two letters that refer to the domain and co-domain.
The proof of \cref{thm:diagram_commutes} is given in \cref{sect:equiv_bij}. 
The key-step is the following fact, proved in \cref{prop:eq_trees}: \emph{Given a bipolar orientation $m$, then the ``branching structure" of the trajectories of the coalescent-walk process $\wcp\circ\bow(m)$ is equal to the tree $T(m^*)$}. The reader is invited to verify it in \cref{fig:Coal_process_exemp}.

\subsection{Local limit results}

We can now consider local limits, more precisely Benjamini--Schramm limits, of the four families in \cref{eq:comm_diagram}. Informally, Benjamini--Schramm convergence for discrete objects looks at the convergence of the neighborhoods of any fixed size of a uniformly distinguished point, called the root of the object.
In order to properly define the Benjamini--Schramm convergence for the four families, we need to present the spaces of infinite objects and the respective local topologies. This is done in \cref{sect:inf_loc_obj,sect:local topologies}, but we give a quick summary here.
\begin{itemize}
	\item $\widetilde \Walks_\bullet$ is the space of two-dimensional walks indexed by a finite or infinite interval of $\Z$ containing zero, with value $(0,0)$ at time $0$, local convergence being finite-dimensional convergence.
	\item $\widetilde \Coals_\bullet$ is the space of coalescent-walk processes indexed by a finite or infinite interval of $\Z$ containing zero, local convergence being finite-dimensional convergence.
	\item The space $\widetilde \Perms_\bullet$ of infinite permutations and its local topology were defined in \cite{borga2018local}. In this context, an infinite permutation is a total ordering on a finite or infinite interval of $\Z$ containing zero.
	\item The space $\widetilde \Maps_\bullet$ of infinite rooted maps is equipped with the local topology derived from the local convergence for graphs of Benjamini and Schramm. See for instance \cite{curienrandom} for an introduction.
\end{itemize}
In the first three items, the index $0$ has to be understood as the root of the infinite object, and comparison between a rooted finite object and an infinite one is done after applying the appropriate shift.

For every $n\in \Z_{>0}$, let $\bm W_n$, $\bm Z_n$, $\bm \sigma_n$, and $\bm m_n$ denote uniform objects of size $n$ in $\mathcal W_n$, $\mathcal C_n$, $\mathcal P_n$, and $\mathcal O_n$, respectively, related by the four bijections in the commutative diagram in \cref{eq:comm_diagram}.
We define below the candidate local random limiting objects. Let $\nu$ denote the probability distribution on $\Z^2$ given by:
\begin{equation}\label{eq:step_distribution_walk}
\nu = \frac 12 \delta_{(+1,-1)} + \sum_{i,j\geq 0} 2^{-i-j-3}\delta_{(-i,j)},\quad \text{where $\delta$ denotes the Dirac measure},
\end{equation}
and let\footnote{Here and throughout the paper we denote random quantities using \textbf{bold} characters.} $\overline{\bm W} = (\overline{\bm X},\overline{\bm Y}) = (\overline{\bm W}_t)_{t\in \Z}$ be a two-sided two-dimensional random walk with step distribution $\nu$, having value $(0,0)$ at time 0. Remark that $\overline {\bm W}$ is not confined to the non-negative quadrant.

A formal definition of the other limiting objects requires an extension of the mappings in \cref{eq:comm_diagram} to infinite-volume objects\footnote{The terminology \emph{finite/infinite-volume} refers to the fact that the objects are defined in a compact/non-compact set.
For instance a Brownian motion with time space $\mathbb R$ is an infinite-volume object and a Brownian excursion with time space $[0,1]$ is a finite-volume object.}
which is done in \cref{sec:discrete_coal}, \cref{sect:infinite_bij} and \cref{sect:inf_loc_obj}.
Admitting we know such extensions, let $\overline{\bm Z} = \wcp(\overline{\bm W})$ be the corresponding infinite coalescent-walk process, $\overline{\bm \sigma} = \cpbp(\overline{\bm Z})$ the corresponding infinite permutation on $\Z$, and $\overline{\bm m} = \bow^{-1}(\overline{\bm W})$ the corresponding infinite map.

\begin{theorem}
	[Quenched Benjamini--Schramm convergence]
	\label{thm:local}
Consider the sigma-algebra $\mathfrak{B}_n\coloneqq\sigma(\bm W_n)=\sigma(\bm Z_n)=\sigma(\bm \sigma_n)=\sigma(\bm m_n)$. Let $\bm i_n$ be an independently chosen uniform index of $[n]$.
We have the following convergence in probability in the space of probability measures on $\widetilde \Walks_\bullet\times\widetilde \Coals_\bullet\times\widetilde \Perms_\bullet\times\widetilde \Maps_\bullet$,
		\begin{equation}\mathcal{L}\Big(\big((\bm W_n, \bm i_n), (\bm Z_n, \bm i_n), (\bm \sigma_n, \bm i_n), (\bm m_n, \bm i_n)\big)\Big| \mathfrak{B}_n\Big)\xrightarrow[n\to\infty]{P}\mathcal{L}\left(\overline{\bm W},\overline{\bm Z},\overline{\bm \sigma},\overline{\bm m}\right),\label{eq:quenched_conv_walks}
		\end{equation}
	where $\mathcal{L}(\cdot)$ denotes the law of a random variable.
\end{theorem}
We note that the mapping $\bow^{-1}$ naturally endows the map $\bm m_n$ with an edge labeling and the root $\bm i_n$ of $\bm m_n$ is chosen according to this labeling. An immediate corollary, which follows by averaging, is the simpler \textit{annealed} statement.
\begin{corollary}[Annealed Benjamini--Schramm convergence]
	We have the following convergence in distribution in the space $\widetilde \Walks_\bullet\times\widetilde \Coals_\bullet\times\widetilde \Perms_\bullet\times\widetilde \Maps_\bullet$,	\begin{equation}\label{eq:annealed_conv_walks}
	((\bm W_n, \bm i_n), (\bm Z_n, \bm i_n), (\bm \sigma_n, \bm i_n), (\bm m_n, \bm i_n)) \xrightarrow[n\to\infty]{d} (\overline{\bm W},\overline{\bm Z},\overline{\bm \sigma},\overline{\bm m}).
\end{equation}
\end{corollary}

Intuitively, in the quenched case, the random discrete objects are frozen, whereas in the annealed case, the random discrete objects and the choice of the roots are treated on the same level.

Using the theory developed in \cite{borga2018local}, quenched convergence of permutations is equivalent to a statement on \textit{consecutive patterns densities}. For $\pi \in \Perms_k$ and $\sigma \in \Perms_n$, denote $\widetilde\coc(\pi,\sigma)$ the proportion of the $n-k+1$ sets of $k$ consecutive indices of $[n]$ that induce the pattern $\pi$ in $\sigma$. Denote $\widetilde{\coc}(\pi,\overline{\bm \sigma})$ the probability that the restriction of the total order $\overline{\bm \sigma}$ to an interval of size $|\pi|$ induces the pattern $\pi$ (the choice of the interval is not relevant since $\overline{\bm \sigma}$ is \emph{shift-invariant}, for more details see \cite[Section 2.6]{borga2018local}). By \cite[Corollary 2.38]{borga2018local}, quenched convergence of Baxter permutations implies the following.
\begin{corollary}
	We have the following convergence in probability w.r.t.\ the product topology on $[0,1]^\Perms$
	\begin{equation}\label{eq:cocc_conv}
	\left(\widetilde{\coc}(\pi,\bm{\sigma}_n)\right)_{\pi\in\Perms}\stackrel{P}{\to}\left(\widetilde{\coc}(\pi,\overline{\bm \sigma})\right)_{\pi\in\Perms}.
	\end{equation}
\end{corollary}

We collect a few comments on these results.
	\begin{enumerate}
	\item \cref{eq:quenched_conv_walks,eq:cocc_conv} witness a concentration phenomenon. Indeed, in both instances, the left-hand side is random, and the right-hand side is deterministic.
		\item The fact that the four convergences are joint follows from the fact that the extensions of the mappings in \cref{eq:comm_diagram} to infinite-volume objects are a.s.\ continuous.
		\item The annealed Benjamini-Schramm convergence for  bipolar orientations to the so-called \textit{Uniform Infinite Bipolar Map} was already proven in \cite[Prop. 3.10]{gwynne2017mating} (see  \cref{sect:planarmaps} for the relations between our work and existing works in the theory of planar maps); the quenched version is however a new result.  
	\end{enumerate}
\bigskip

\subsection{Scaling limit results}
\label{sect:intro_scaling_limit_results}

We now turn to our main result.
For $n\geq 1$, let $\bm \sigma_n$ be a uniform Baxter permutation of size $n$ and $\bm m_n=\bobp^{-1}(\bm \sigma_n)$ the corresponding uniform bipolar orientation with $n$ edges. Let $\bm W_n = \bow(\bm m_n)$ and $\bm W_n^* = \bow(\bm m_n^*)$ be the two tandem walks associated with $\bm m_n$ and its dual $\bm m_n^*$.	
Let ${\conti W}_n$ and ${\conti W}_n^*$ be the two continuous functions from $[0,1]$ to $\R_{\geq 0}^2$ that linearly interpolate between the points ${\conti W}_n^\theta\left(\frac kn\right) = \frac 1 {\sqrt {2n}} {\bm W}_{n}^\theta(k)$ for $1\leq k \leq n$ and $\theta\in\{\emptyset,*\}$.

Let $\conti W = (\conti X(t),\conti Y(t))_{t\geq 0}$ be a \textit{standard two-dimensional Brownian motion of correlation -1/2}, that is a continuous two-dimensional Gaussian process such that the components $\conti X$ and $\conti Y $ are standard one-dimensional Brownian motions, and $\mathrm{Cov}(\conti X(t),\conti Y(s)) = -1/2 \cdot (t\wedge s)$. Let $\conti W_e$ be a \textit{two-dimensional Brownian excursion of correlation -1/2 in the non-negative quadrant}, that is the process $(\conti W(t))_{0\leq t\leq 1}$ conditioned on $\conti W(1) = (0,0)$ and on staying in the non-negative quadrant $\R_{\geq 0}^2$. A rigorous definition is given in \cref{sec:appendix}.

Consider the time-reversal and coordinate-swapping mapping
$s:\mathcal C([0,1],\R^2) \to \mathcal C([0,1],\R^2)$ defined by $s(f,g) = (g(1-\cdot), f(1-\cdot))$. Consider also the mapping $R: \mathcal M \to \mathcal M$ that rotates a permuton by an angle $-\pi/2$. 
%that is $R(\mu)(A) = \mu\left(\begin{psmallmatrix}
%0&-1\\ 1 & 0
%\end{psmallmatrix}
%\cdot A\right)$ for every Borel set $A\subseteq [0,1]^2$.

\begin{theorem}	\label{thm:joint_intro}
	There exist two measurable mappings $r:\mathcal C([0,1], \R_{\geq 0}^2) \to \mathcal C([0,1],\R_{\geq 0}^2)$ and $\phi:\mathcal C([0,1], \R_{\geq 0}^2) \to \mathcal M$ (both explicitly described in \cref{sec:final}) such that we have the convergence in distribution
	\begin{equation}
	\label{eq:scal_lim_comp}
	({\conti W}_n,{\conti W}_n^*, \mu_{\bm \sigma_n}) \to (\conti W_e, \conti W_e^*, \bm \mu_B),
	\end{equation}	
	where $\conti W_e^* = r(\conti W_e)$, and $\bm \mu_B  = \phi(\conti W_e)$.
	In particular, we have $r(\conti W_e) \stackrel d= \conti W_e$. Moreover, we have the following equalities that hold at $\Prob_{\conti W_e}$-almost every point of $\mathcal C([0,1], \R_{\geq 0}^2)$,
	\begin{gather*}
	r^2 = s, \quad
	r^4 = \Id,\quad
	\phi \circ r = R\circ \phi.
	\end{gather*} 
\end{theorem}
We give a few remarks on this result:
\begin{enumerate}
	\item The convergence of the first or second marginal was obtained in \cite{MR3945746} as an immediate application of the results of \cite{duraj2015invariance} on walks in cones.
	\item Our strategy of proof is based on  coalescent-walk processes, which describe the relation between $\conti W_n$, $\conti W_n^*$ and $\bm \sigma_n$ in a way that allows itself to take limits. In the remainder of this section we explain what the scaling limit of coalescent-walk processes is, providing the reader with some insights on how the coupling of the right-hand side of \cref{eq:scal_lim_comp} is constructed. Precise statements, including explicit constructions of the mappings $r$ and $\phi$, are given in \cref{sec:final} (see in particular \cref{thm:permuton,thm:joint_scaling_limits}).
	\item The limiting permuton $\bm \mu_B$, called the \emph{Baxter permuton}, is a new fractal random measure on the unit square (see \cref{defn:Baxter_perm} for a precise definition and an explicit construction of the mapping $\phi$). 
	\item The map $r:\mathcal C([0,1], \R_{\geq 0}^2) \to \mathcal C([0,1],\R_{\geq 0}^2)$ is almost surely determined by the relations in \cref{eq:definition_map_r_1} below.
	\item Recall that each coordinate of $\bm W_n$ or $\bm W_n^*$ records the height function of a tree which can be drawn on $\bm m_n$ or its dual. So this statement can be interpreted as joint convergence of four trees to a coupling of four Brownian CRTs. We discuss the relation with Conjecture 4.4 of \cite{MR3945746}, the main result of \cite{GHS}, and other related works, in \cref{sect:planarmaps}.
\end{enumerate}

\medskip

The proof of \cref{thm:joint_intro} is based on a result on scaling limits of the coalescent-walk processes $\bm Z_n=\wcp (\bm W_n)$, which appears to be of independent interest.
We give here some brief explanations and
we refer the reader to \cref{sec:coalescent} for more precise results.

The definition of the coalescent-walk process $Z=\wcp(W)$ associated with a two-dimensional walk $W$ is a sort of ``discretized" version of the following family of stochastic differential equations driven by \textit{the same} two-dimensional process $\conti W = (\conti X,\conti Y)$ and defined for $u\in\mathbb R$ by
\begin{align}
\begin{cases}d\conti Z^{(u)}(t) = \idf_{\{\conti Z^{(u)}(t)> 0\}} d\conti Y(t) - \idf_{\{\conti Z^{(u)}(t)\leq 0\}} d\conti X(t), &t\geq u,\\
\conti Z^{(u)}(t)=0, &t\leq u.
\end{cases} \label{eq:flow_SDE_intro}
\end{align}
We emphasize that we have a different equation for every $u\in\R$, but all depend on \textit{the same} process $\conti W = (\conti X,\conti Y)$.
For a fixed $u\in\R$, this equation, that goes under the name of \textit{perturbed Tanaka's SDE}, has already been studied in the literature \cite{MR3098074,MR3882190} in the case where $\conti W$ is a two-dimensional Brownian motion of correlation $\rho$ with $\rho\in(-1,1)$, and more generally when the correlation coefficient varies with time. In particular, pathwise uniqueness and existence of a strong solution are known. Since the scaling limit of $\bm W_n$ (that is conditioned to start on the $x$-axis and end on the $y$-axis) is a two-dimensional Brownian \emph{excursion} $\conti W_e$ of correlation $-1/2$, one can expect that the scaling limit for the coalescent-walk process $\bm Z_n = \wcp(\bm W_n)$ is a sort of flow of solutions $\{\conti Z_e^{(u)}(t)\}_{u\in[0,1]}$ of the SDEs in \cref{eq:flow_SDE_intro} driven by $\conti W_e$ (instead of $\conti W$). This intuition is made precise in \cref{thm:discret_coal_conv_to_continuous} and it is the key-step for proving \cref{thm:joint_intro}.

\medskip

The study of flows of solutions driven by the \emph{same noise} is the subject of the theory of coalescing flows of Le Jan and Raimond, specifically that of \textit{flows of mappings}. See \cite{MR2060298} and the references therein. We point out that we do not need to make use of this theory. Indeed, in our proof of \cref{thm:joint_intro} we consider solutions of \cref{eq:flow_SDE_intro} for only a countable number of distinct $u$ at a time for a specific equation which admits strong solutions.
In particular, \cref{thm:discret_coal_conv_to_continuous} gives convergence of a countable number of trajectories in the product topology. Stronger convergence results, such as the ones obtained for the \textit{Brownian web} (see \cite{MR3644280} for a comprehensive survey) would be desirable.

We now turn on discussing the implications of \cref{thm:joint_intro} on the scaling limit of bipolar orientations and their trees.

\subsection{Scaling limits of bipolar orientations and relations with other works}\label{sect:planarmaps}

Scaling limits of random planar maps have been thoroughly studied with motivations from string theory and conformal field theory. Convergence results for many models of random planar maps were obtained, both as random metric spaces (with the celebrated theorems of Le Gall and Miermont \cite{le2013uniqueness,miermont2013brownian} showing convergence to the \textit{Brownian map}) and, more recently, as random Riemannian surfaces (with the remarkable achievements of Holden and Sun~\cite{HoldenSun} showing convergence to the $\sqrt{8/3}$-Liouville quantum gravity).

The number of bipolar orientations of a given planar map is computed by coefficient extraction in the Tutte polynomial. This makes bipolar orientations one of the various combinatorially tractable models of \textit{planar maps with additional structure}, like spanning-tree decorated maps, loop-decorated maps, and so on. Uniform objects in such combinatorial classes are not uniform planar maps anymore.  Their scaling limit is expected to differ, and be connected to $\gamma$\textit{-Liouville quantum gravity} ($\gamma$-LQG for short) for some $\gamma\neq \sqrt{8/3}$.
While convergence results as random surfaces remain open,  weaker topologies such as \textit{Peanosphere convergence} have been investigated with success. We refer to \cite{gwynne2019mating} for a comprehensive survey.

The specific case of bipolar orientations was first studied by Kenyon, Miller, Sheffield and Wilson~\cite{MR3945746}, using their bijection $\bow$ interpreted as a \textit{mating-of-trees} encoding of bipolar orientations.
Using the remarkable works of Denisov and Wachtel~\cite{MR3342657} and Duraj and Wachtel~\cite{duraj2015invariance} on scaling limits of random walks conditioned to stay in a cone, they showed that the random tandem walk associated with a random bipolar orientation  converges to a two-dimensional Brownian excursion of correlation $-1/2$ in the non-negative quadrant (this corresponds to the convergence of the first marginal in \cref{eq:scal_lim_comp}). This is called
\textit{Peanosphere convergence} of the maps decorated by their interface path to a $\sqrt{4/3}$-LQG sphere together with a SLE${}_{12}$ curve.

A stronger result was given by Gwynne, Holden and Sun~\cite{GHS}, in the specific case of the \textit{uniform infinite-volume bipolar triangulation (UIBT)}. They show a joint scaling limit result for the coding two-dimensional walks of this map and of its dual. The limit is a coupling of two plane Brownian motions defined by LQG and imaginary geometry theory.

The convergence of the first two marginals in \cref{thm:joint_intro} gives a parallel result for \emph{finite-volume bipolar orientations} instead of \emph{infinite-volume bipolar triangulations}.  This result gives a new answer to Conjecture 4.4 of~\cite{MR3945746} as follows.

\medskip

Conjecture 4.4 of \cite{MR3945746} can be split in two parts:
	\begin{itemize}
		\item Using the same notation as in \cref{thm:joint_intro}, Conjecture 4.4 states that $({\conti W}_n,{\conti W}_n^*)$ jointly converge in distribution.
		
		\item In addition, Conjecture 4.4 states that the distribution of the joint limit $(\conti W_e, \conti W_e^*)$ can be extracted from a $\sqrt{4/3}$-LQG surface decorated with an independent SLE$_{12}$ curve and the "dual" SLE$_{12}$ curve.	 
	\end{itemize}
	We now explain how our results in \cref{thm:joint_intro} (combined with some known results in the literature) answers also the second part of this conjecture.
	
	\medskip 
	
	We start by explaining how $\conti W_e$ encodes a SLE$_{12}$ curve decorating a $\sqrt{4/3}$-LQG surface. To do that, we need to use various classical objects related to the LQG theory. We do not define here  all these objects since they are only needed for this section, but we provide precise references for each of the mentioned objects\footnote{Most of the references in this section point to the self-contained survey paper \cite{gwynne2019mating}, where the authors chose to focus (in the finite-volume case) on \emph{quantum disks with boundary length} $\ell$. In the present paper we need to consider \emph{quantum spheres} instead of \emph{quantum disks}. We remark that a quantum sphere can be viewed as a quantum disk with boundary length 0. If the reader is still not satisfied, we point out that all the definitions used in this section can be also found in \cite{MR4010949}.}, hoping this helps the reader.
	
	Let $\mathcal C=\R \times (0, 2\pi)$. In order to be precise, we need to consider a $\sqrt{4/3}$-Liouville quantum sphere $({\mathcal C}, \bm h^{\mathcal C}, -\infty)$ with quantum area one (see \cite[Definition 3.20]{gwynne2019mating}) and one marked point at $-\infty$. Let also $\bm \eta^{\mathcal C}$ be a space-filling SLE$_{12}$ loop on
	${\mathcal C}$ from $-\infty$ to $-\infty$ (see \cite[Definition 3.27]{gwynne2019mating}), sampled independently of $\bm h^{\mathcal C}$ and then parametrized by the $\mu_{\bm h^{\mathcal C}}$-LQG area measure (see \cite[Section 3.3]{gwynne2019mating}), 
	so that $\mu_{\bm h^{\mathcal C}}(\bm \eta^{\mathcal C}([0, t])) = t$ for each $t\in[0, 1]$. 
	
	For $t\in[0,1]$, let $\conti L_t$ be
	the $\nu_{\bm h^{\mathcal C}}$-LQG length measure (see \cite[Section 3.3]{gwynne2019mating}) of the left outer boundary of $\bm \eta^{\mathcal C}([0, t])$ and $\conti R_t$ be
	the $\nu_{\bm h^{\mathcal C}}$-LQG length measure
	of the right outer boundary of $\bm \eta^{\mathcal C}([0, t])$.
	From \cite[Theorem 1.1]{MR4010949} (see also \cite[Theorem 4.10]{gwynne2019mating}) the process $\conti Z_t=(\conti L_t,\conti R_t)$ defined above has the law of a two-dimensional Brownian excursion of correlation $-1/2$ in the non-negative quadrant.
	In addition, the process $\conti Z_t$ a.s.\ determines $({\mathcal C}, \bm h^{\mathcal C},\bm \eta^{\mathcal C}, -\infty)$ as a curve-decorated quantum surface.  
	
	This explains how the process $\conti W_e$ in \cref{thm:joint_intro} encodes a SLE$_{12}$ curve decorating a $\sqrt{4/3}$-Liouville quantum sphere.
	
	\medskip
	
	We remark that \cref{thm:joint_intro} (and its more complete formulation in \cref{thm:joint_scaling_limits}) explicitly determines the joint distribution of the pair $(\conti W_e, \conti W_e^*)$. Indeed, we prove in \cref{thm:joint_scaling_limits} that $ \conti W_e^*=r( \conti W_e)$ and the map $r:\mathcal C([0,1], \R_{\geq 0}^2) \to \mathcal C([0,1],\R_{\geq 0}^2)$ is almost surely determined by the relations in \cref{eq:definition_map_r_1}.
	
	\medskip
	
	We finally comment on the "dual" SLE$_{12}$ curve mentioned in Conjecture 4.4. To do that we need to use some results from \cite{GHS}. Unfortunately, as already mentioned before, our results are stated for uniform finite-volume bipolar orientations, while the results in \cite{GHS} are stated for uniform infinite-volume triangulations. It should not come as a surprise the fact that both results can be adapted to the other case. In what follows we explain one possible strategy to build a connection between our results and the ones in \cite{GHS}.
	
	Using the methods developed in the present paper, it is straightforward to reproduce the scaling limit result obtained in \cite[Theorem 1.6]{GHS}, that is the scaling limit in the Peanosphere topology for uniform infinite-volume triangulations\footnote{It simply ammounts in considering a family of tandem walks with different admissible steps, but with exactly the same correlation $-1/2$, see \cite[Remark 2]{MR3945746}}. To fix the notation, we state this result.
	
	\begin{theorem}
		Let $({\widetilde{\conti W}}_n,{\widetilde{\conti W}}_n^*)$ be the pair of tandem walks obtained from a uniform infinite-volume triangulation of size $n$.
		There exists an (explicit) measurable mapping $r:\mathcal C([0,1], \R_{\geq 0}^2) \to \mathcal C([0,1],\R_{\geq 0}^2)$ such that we have the convergence in distribution
		\begin{equation}
			({\widetilde{\conti W}}_n,{\widetilde{\conti W}}_n^*) \to (\conti W, \conti W^*),
		\end{equation}	
		where $\conti W$ is a two-dimensional Browian motion of correlation $-1/2$ and $\conti W^* = r(\conti W)$. 
	\end{theorem}

	We recall that the pair of Brownian motions $(\conti W, \conti W^*)$ describes a $\sqrt{4/3}$-LQG surface decorated with an independent SLE$_{12}$ curve and the "dual" SLE$_{12}$ curve as explained in \cite[Section 1.5]{GHS}. It is straightforward to transfer this description in the finite-volume case where the pair $(\conti W, \conti W^*)$ is replaced by a pair of two-dimensional Browian excursions of correlation $-1/2$ (for instance this procedure is implemented in \cite[Section 5.4]{LSW} in a slightly different case). 
	Therefore the distribution of our joint limit $(\conti W_e, \conti W_e^*)$ in \cref{thm:joint_intro} -- that is the same both for uniform finite-volume triangulations and bipolar orientations -- describes a $\sqrt{4/3}$-Liouville quantum sphere decorated with an independent SLE$_{12}$ curve and the "dual" SLE$_{12}$ curve.

	Our contribution is the new description of the map $r:\mathcal C([0,1], \R_{\geq 0}^2) \to \mathcal C([0,1],\R_{\geq 0}^2)$ above: it gives a new perspective on the description of the coupling between $\conti W_e$ and $\conti W_e^*$. Indeed, in our paper this coupling is explicitly described in terms of the SDEs in \cref{eq:flow_SDE_intro} and the Baxter permuton (see again \cref{eq:definition_map_r_1}), while in \cite{GHS} this coupling is only implicitly described building on some constructions derived from imaginary geometry (see for instance the comment before \cite[Proposition 3.2]{GHS}). We plan to investigate this new description in a much more general setting in some future projects, see \cref{sec:lqg_sde}.

\subsection{Generality of our techniques and open problems}
\label{sec:perspectives}

We discuss here some problems that we would like to address in future projects and some possible further applications of the techniques that we developed in the current paper.

\paragraph{Properties of the Baxter permuton} The Baxter
permuton $\bm \mu_B$ is a new fractal random measure on the unit square. We think it might be worth it to investigate its properties. The first two natural questions that we would like to answer are the following:
\begin{itemize}
\item What is the density of the intensity measure $\E[\bm \mu_B]$?
\item What is the Hausdorff dimension of the support of $\bm \mu_B$?
\end{itemize}
We point out that similar questions were solved for the Brownian separable permuton (i.e.\ the permuton limit for separable permutations) in a recent work of the second author~\cite{maazoun}. For instance, it was shown that almost surely, the support of the Brownian separable permuton
is totally disconnected, and its Hausdorff
dimension is 1 (with one-dimensional Hausdorff measure bounded above by $\sqrt 2$). For an expression for the intensity measure see \cite[Theorem 1.7]{maazoun}.

\paragraph{Strong convergence for coalescent-walk processes} As already mentioned, it would be desirable to improve the convergence of discrete coalescent-walk processes to continuous coalescent-walk processes in a stronger topology. As in the case of the Brownian web, this would allow to study coalescence points, non-uniqueness points, and the interaction between the coalescent-walk process and its backwards version, features that are not captured in our results.
\paragraph{Generality of our techniques} We strongly believe that our techniques used for proving scaling limit results for uniform bipolar orientations would still apply (with minor modifications) to the weighted models of bipolar orientations considered in \cite[Theorem 2.6]{MR3945746}, including in particular uniform bipolar $k$-angulations for every $k\geq 3$. We point out that this weighting is not very natural in terms of the corresponding Baxter permutations.
 
\paragraph{Universality of the Baxter permuton and possible generalizations} We believe that the robustness of our techniques goes further, and hope to apply them to many other families of permutations, showing that the Baxter permuton $\bm \mu_B$ is a \emph{universal limiting object}.
\begin{itemize}
	\item We will observe in \cref{sect:sep_perm} that the case of separable permutations can be treated with coalescent-walk processes too. In particular, we will explain that their limiting permuton, i.e.\ the Brownian separable permuton mentioned in \cref{sect:perm_sect}, is related to the Tanaka's SDE, which is \cref{eq:flow_SDE_intro} when the driving process $\conti W$ is a two-dimensional Brownian excursion of correlation $\rho = 1$.
	\item Recall that in the case of Baxter permutations, we had $\rho = -1/2$. It would be interesting to find families of permutations that correspond to yet other values of $\rho$.
	Baxter permutations avoiding the pattern $2413$, which form a subset of Baxter permutations, and a superset of separable permutations, could be a good candidate for a first answer to the question above. In \cite[Proposition 5]{MR2734180}, they are shown to be in bijection, through $\bobp$, with rooted non-separable planar maps.
	
	\item We would like to explore several other families of permutations where a bijection with two-dimensional walks is available. A technique is presented in a recent work of the first author~\cite{BorgaCLT2020} to sample uniform permutations in families enumerated through
	generating trees with $d$-dimensional labels as conditioned
	random colored walks in $\Z^d$.  
	Examples of families of permutations with two-dimensional labels can be found for instance in \cite{MR2028288,MR2376115, MR3882946, MR3961884}.
	\item In \cref{sec:lqg_sde},  we show how the LQG literature \cite{GHS,LSW} suggests a more general version of the SDE \eqref{eq:flow_SDE_intro} with an additional parameter $p$ and a local time term (for more details see the SDE \eqref{eq:generalized}). The SDE \eqref{eq:flow_SDE_intro} corresponds to the special case $p=1/2$. Such a generality might be needed to treat some of the models cited above.

	The study of Schnyder woods by \cite{LSW} correspond to $\rho = -\frac {\sqrt 2}2$ and $p = \frac {\sqrt 2}{1+\sqrt 2}$. However, the corresponding model of permutations is not natural (it is a weighted model of Baxter permutations). It would also be interesting to study the generalizations of Schnyder woods described in \cite{MR2871142}.
\end{itemize}

\subsection{Outline of the paper}
\noindent \textbf{\cref{sec:discrete}.}
After setting some definitions and recalling some properties of the bijection $\bow$, we properly define  coalescent-walk processes and prove \cref{thm:diagram_commutes}. This section contains all discrete arguments used in the rest of the paper.

\medskip 
\noindent \textbf{\cref{sec:local}.} This section is devoted to the proof of \cref{thm:local}. We first define the local topologies and the infinite-volume objects. Then the argument follows readily from local convergence of uniform tandem walks $\bm W_n$ to the random walk $\overline{\bm W}$, and local continuity of the mappings $\bow^{-1}, \wcp,\cpbp$. In particular, the local limit of Baxter permutations is defined from the infinite-volume coalescent-walk process $\overline{\bm Z} = \wcp(\overline{\bm W})$, which enjoys the nice property that its trajectories are random walks (\cref{prop:trajectories_are_rw}). This turns out to be useful also in the following sections.

 \medskip
 \noindent \textbf{\cref{sec:coalescent}.}
To proceed with the proof of \cref{thm:joint_intro}, we need to show that the trajectories of the coalescent-walk process $\bm Z_n = \wcp(\bm W_n)$ converge in distribution, jointly with $\bm W_n$. We prove this for coalescent-walk processes driven by unconditioned random walks (\cref{thm:coal_con_uncond}). The proof relies on the pathwise uniqueness property of the SDE \eqref{eq:flow_SDE_intro}. We then transfer this result to two-dimensional excursions in the non-negative quadrant, culminating in \cref{thm:discret_coal_conv_to_continuous}, which is the basis for the next section.

 \medskip
 \noindent \textbf{\cref{sec:final}.}
 We finally state and prove \cref{thm:permuton,thm:joint_scaling_limits}, which are more precise versions of \cref{thm:joint_intro}.
 
  \medskip
  \noindent \textbf{\cref{sec:appendix}.}
  This section contains absolute continuity results and local limits theorem used in the study of conditioned walks in the non-negative quadrant, extracted from \cite{MR3342657,duraj2015invariance, bousquet2019plane} or stated in a different form when needed.

  \medskip
\noindent \textbf{\cref{sec:general}.}
This section contains comments on possible generalizations.

  \medskip
\noindent \textbf{\cref{sect:simulations}.}
In this final section we explain how the simulations for Baxter permutations presented in the first page of this paper are obtained.

\paragraph*{Acknowledgments} Thanks to Mathilde Bouvel, Valentin Féray and Grégory Miermont for their dedicated supervision and enlightening discussions. Thanks to Nicolas Bonichon, Nina Holden, Emmanuel Jacob, Jason Miller, Kilian Raschel, Xin Sun, Olivier Raimond and Vitali Wachtel for enriching discussions and pointers.
We finally thank the anonymous referees for all their precious and useful comments as well as Emmanuel Kammerer for some remarks.

\section{Bipolar orientations, walks in cones, Baxter permutations and coalescent-walk processes}
\label{sec:discrete}

This section contains the combinatorial material relevant to our arguments. We first settle in \cref{sec:discrete_maps} some definitions and terminology related to planar maps and rooted trees. Then in \cref{sect:KMSW} we describe the reverse $\bow$ bijection and in \cref{sect:Baxt_bipol} we show that the definition of $\bobp$ given in \cref{defn:bobp} is equivalent to the one presented in \cite{MR2734180}. Finally, \cref{sec:discrete_coal} is the combinatorial heart of the paper: we properly introduce coalescent-walk processes and the mappings $\cpbp$ and $\wcp$, proving \cref{thm:diagram_commutes}.

\subsection{Planar maps and rooted trees} 
\label{sec:discrete_maps}
A \emph{planar map} is a finite connected graph embedded in the plane with no edge-crossings, considered up to orientation-preserving homeomorphisms of the plane. A map has vertices, edges, and faces, the latter being the connected components of the plane remaining after deleting the edges. The outer face is unbounded, the inner faces are bounded\footnote{The outer face plays a special role in the maps we consider. In the usual terminology of the literature, they are \textit{planar maps with one boundary}.}.

Alternatively, a planar map is a finite collection of finite polygons (the inner faces), glued along some pairs of edges, so that the resulting surface has the topology of the disc, i.e.\ is simply connected and has one boundary. We call finite map an arbitrary gluing of a finite collection of finite polygons.
A submap  $m'$ of a planar map $m$ is a subset of the inner faces of $m$, where gluing of faces in $m'$ inherits from the gluing in $m$. The submap $m'$ is in general a finite map, and it is a planar map if and only if it is simply connected.

\medskip

For our purposes, we view rooted plane trees with the root at the bottom. A rooted plane tree may be seen as a set of edges equipped with a parent-child relation, where each edge has at most one parent. The children of each edge are ordered as well as the parentless edges, that are the edges on top of the root. Other edges sit on top of their respective parent in the prescribed ordering.

The \emph{down-right tree} $T(m)$ of a bipolar orientation $m$ was defined informally in the introduction. In this context, we have the following more rigorous definition:
\begin{itemize}
	\item The edges of $T(m)$ are the edges of $m$.
	\item Let $e\in m$ and $v$ its bottom vertex. The parent of $e$ in $T(m)$ is the right-most incoming edge of $v$, if it exists.
	\item The ordering of edges on top of $e$ in $T(m)$ is inherited from their ordering on top of their common bottom vertex in $m$.
\end{itemize}

We conclude this section recalling that the \emph{exploration} of a tree $T$ is the visit of its vertices (or its edges) starting from the root and following the contour of the tree in the clockwise order. Moreover, the \emph{height process} of a tree $T$ is the sequence of integers obtained by recording for each visited vertex (following the exploration of $T$) its distance to the root.

\subsection{The Kenyon-Miller-Sheffield-Wilson bijection}
\label{sect:KMSW}

We recall the definition of the mapping $\bow : \mathcal O \to \mathcal W$ given in the introduction since it is fundamental for what follows. Recall that for a bipolar orientation $m$, we denote by $e_1,\ldots,e_{|m|}$ the edges of $m$ in the order determined by the interface path.
\begin{definition}\label{defn:KMSW2}
	Let $n\geq 1$, $m \in \mathcal O_n$. We define $\bow(m)=(X_t,Y_t)_{1\leq t\leq n} \in (\Z_{\geq 0}^2)^n$ as follows: for $1\leq t\leq n$, $X_t$ is the height in the tree $T(m)$ of the bottom vertex of $e_t$
	(i.e.\ its distance in $T(m)$ from the source $s$),
	and $Y_t$ is the height in the tree $T(m^{**})$ of the top vertex of $e_t$
	(i.e.\ its distance in $T(m^{**})$ from the sink $s'$).
\end{definition}

Recall also that $\mathcal{W}$ is the set of tandem walks, i.e.\ two-dimensional walks in the non-negative quadrant, starting at $(0,h)$ and ending at $(k,0)$ for some $h\geq 0, k\geq 0$, with increments in $\Steps = \{(+1,-1)\} \cup \{(-i,j), i\in \Z_{\geq 0}, j\in \Z_{\geq 0}\}$. 

An equivalent way of understanding $\bow$ is as follows.

\begin{remark}\label{rem:height_process}
	Let $m\in \mathcal O$ and $\bow(m) = ((X_t,Y_t))_{1\leq t \leq n}$. The walk $(0,X_1+1,\ldots,X_{|m|}+1)$ is the height process of the tree $T(m)$. The walk $(0,Y_{|m|}+1, Y_{|m|-1}+1,\ldots ,Y_1+1)$ is the height process of the tree $T(m^{**})$.
\end{remark}

We now explain some properties of the mapping $\bow$. Let $m\in \mathcal O$ and $\bow(m) = ((X_t,Y_t))_{1\leq t \leq n}$. Suppose that the left outer face of $m$ has $h+1$ edges and the right outer face of $m$ has $k+1$ edges, for some $h,k\geq 0$. Then the walk $(X_t,Y_t)_{t\in[|m|]}$ starts at $(0,h)$, ends at $(k,0)$, and stays by definition in the non-negative quadrant $\mathbb{Z}_{\geq 0}^2$.

We give an interpretation to the increments of the walk, i.e.\ the values of $(X_{t+1},Y_{t+1})-(X_t,Y_t)$. 
We say that two edges of a tree are consecutive if one is the parent of the other.
The interface path of the map $m$, defined in \cref{sect:discr_obj}, has two different behaviors: 
\begin{itemize}
	\item either it is following two edges $e_t$ and $e_{t+1}$ that are consecutive, both in $T(m)$ and $T(m^{**})$, in which case the increment is $(+1,-1)$;
	\item or it is first following $e_t$, then it is traversing a face of $m$, and finally is following $e_{t+1}$, in which case the increment is $(-i,+j)$ with $i,j\in\Z_{\geq 0}$, and the traversed face has left degree $i+1$ and right degree $j+1$.
\end{itemize}

\begin{exmp}\label{exemp:walk}
	Consider the map $m$ in Fig.~\ref{fig:bip_orient _with_trees}. 
The corresponding walk $\bow(m)$ plotted on the right-hand side of \cref{fig:bip_orient _with_trees} is:
	\begin{equation*}
	\begin{split}
	&W_{1}=(0,2),W_{2}=(0,3),W_{3}=(0,3),W_{4}=(1,2),W_{5}=(2,1),\\
	&W_{6}=(0,3),W_{7}=(1,2),W_{8}=(2,1),W_{9}=(3,0),W_{10}=(2,0).
	\end{split}
	\end{equation*}
	Note, for instance, that $W_6 - W_5=(-2,2)$, indeed between the edges 5 and 6 the interface path is traversing a face with $3$ edges on the left boundary and $3$ edges on the right boundary.
	On the other hand $W_9-W_8 = W_8 - W_7 = W_7 - W_6 = (+1,-1)$. Indeed, in these cases, the interface path is following consecutive edges.
\end{exmp}
 
We finish this section by describing the inverse bijection  $\bow^{-1}$. We actually construct a mapping $\Inverse$ on a larger space of walks, whose restriction to $\mathcal W$ is the inverse of $\bow$.

Let $I$ be an interval (finite or infinite) of $\Z$. Let $\Walks(I)$ be the set of two-dimensional walks with time space $I$ considered up to an additive constant. More precisely $\Walks(I)$ is the quotient $(\Z^2)^I / \sim$, where $w\sim w'$ if and only if there exists $x\in \Z^2$ such that $w(i) = w'(i)+x$ for all $i\in I$. We usually take an explicit representative of elements of $\Walks(I)$, chosen according to the context. For instance, if $0\in I$, we often take the representative that satisfies $w(0)=(0,0)$, called ``pinned at zero".
Let $\Walks_\Steps(I)\subset \Walks(I)$ be the restriction to two-dimensional walks with increments in $\Steps$. For every $n\geq 1$, $\mathcal W_n$ is naturally embedded in $\Walks_\Steps([n])$, with an explicit representative

Let $I=[j,k]$ be a finite integer interval. We shall define  $\Inverse$ on every walk in $ \Walks_\Steps(I)$ by induction on the size of $I$, and denote by $\Maps(I)$ the image of $\Walks_A(I)$ by $\Inverse$. An element $m\in\Maps(I)$ is a bipolar orientation, together with a subset of edges  labeled by $j,\ldots,k$ (which identifies a subinterval of the interface path started on the left boundary and ended on the right boundary of $m$). The edges labeled $j,\ldots,k$ are called \textit{explored edges}, and the edge labeled $k$ is called \textit{active}. The other edges, called \textit{unexplored}, are either below $j$ on the left boundary, or above $k$ on the right boundary. Bipolar orientations of size $n\geq 1$ are the elements of $\Maps([n])$ with no unexplored edges. The elements of $\Maps(I)$ are called \textit{marked bipolar orientations} by \cite{bousquet2019plane}.

The base case for our induction is $I=\{j\}$. In this case, an element $W$ of $\Walks_A(\{j\})$ is mapped to a single edge with label $j$. If $W \in \Walks_A([j,k+1])$ with $k\geq j$, then denote by $m' = \Inverse(W|_{[j,k]})$ and

\begin{enumerate}
	\item if $W_{k+1} - W_{k} = (1,-1)$, then $\Inverse(W)$ is obtained from $m'$, by giving label $k+1$ to the edge immediately above the edge of label $k$. If no such edge exists, a new edge is added on top of the sink with label $k+1$.
	\item If $W_{k+1} - W_{k} = (-i,j)$, then $\Inverse(W)$ is obtained from $m'$ by adding a face of left-degree $i+1$ and right-degree $j+1$. Its left boundary is glued to the right boundary of $m'$, starting with identifying the top-left edge of the new face with $e_k$, and continuing with edges below.
	The bottom-right edge of the new face is given label $k+1$, hence is now active. All other edges that were not present in $m'$ are unexplored.
\end{enumerate}

An example of this construction (inductively building $\Inverse(W)$ with $W\in \mathcal W$) is given in \cref{fig:bip_orient_from_walk}. For an example of application of the mapping $\Inverse$ to a walk that is not a tandem walk see \cref{fig:bip_orient_from_not_tandem_walk}. In this case, the final map still has unexplored edges.

\begin{figure}[htbp]
	\centering
	\includegraphics[scale=0.72]{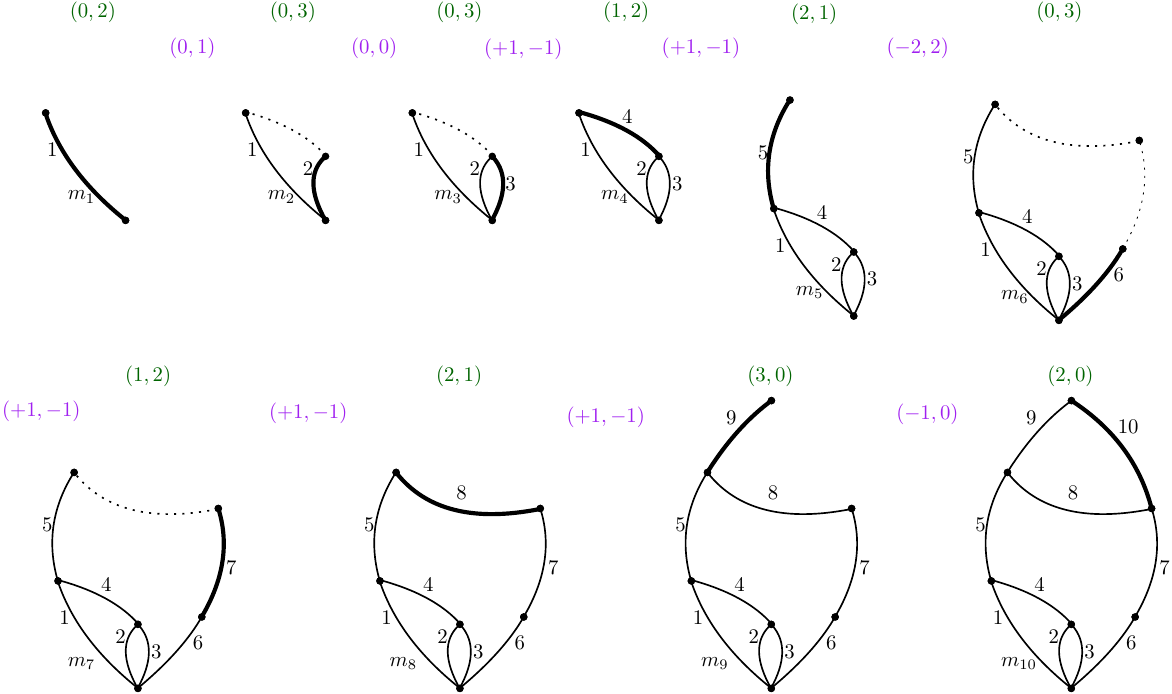}\\
	\caption{The sequence of bipolar orientations $m_k=\Inverse(W|_{[1,k]})$ determined by the walk $W$ considered in \cref{exemp:walk}. Note that $m_{10}$ is exactly the map $m$ in \cref{fig:bip_orient} and \cref{fig:bip_orient _with_trees}. For each map $m_k$, we indicate on top of it (in green) the value $W_k$. Between two maps $m_k$ and $m_{k+1}$, we report (in purple) the corresponding increment $W_{k+1}-W_k$. For every map $m_k$, we draw the explored edges with full lines, the unexplored edges with dotted lines, and we additionally highlight the active edge $e_k$ in bold.   \label{fig:bip_orient_from_walk}}
\end{figure}

\begin{figure}[htbp]
	\centering
	\includegraphics[scale=0.85]{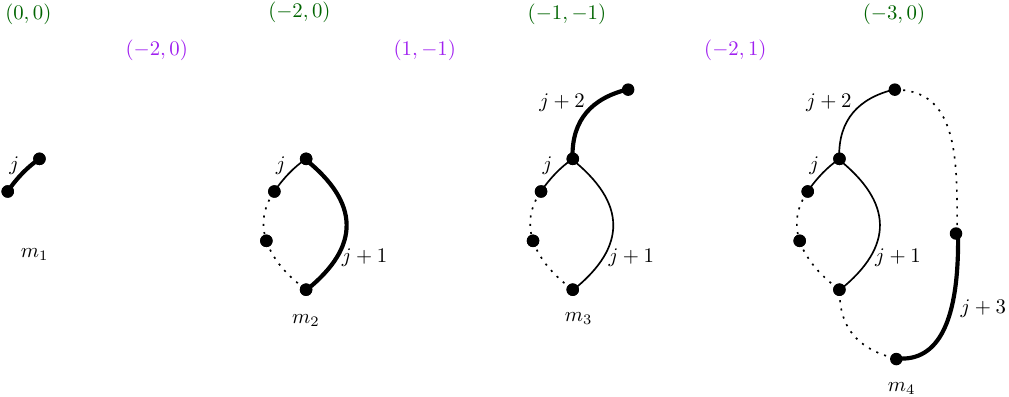}\\
	\caption{The sequence of bipolar orientations $m_k=\Inverse(W|_{[j,j+k-1]})$ determined by the walk $W_j=(0,0),W_{j+1}=(-2,0),W_{j+2}=(-1,-1),W_{j+3}=(-3,0)$, that is an element of the set $\Walks_\Steps([j,j+3])$. This walk is not a tandem walk. We used the same notation as in \cref{fig:bip_orient_from_walk}.  \label{fig:bip_orient_from_not_tandem_walk}}
\end{figure}

\begin{proposition}[Theorem 1 of \cite{MR3945746}] \label{prop:inverse}
	For every finite interval $I$, the mapping $\Inverse: \Walks_A(I)\to \Maps(I)$ is a bijection. Moreover, for every $n\geq 1$, $\Theta(\mathcal W_n) = \mathcal O_n$ and $\bow^{-1}: \mathcal W_n \to \mathcal O_n$ coincides with $\Inverse$. 
	Finally, if $W \in \Walks_A(I)$ and $[j,k] \subset I$, then the map $\Inverse(W|_{[j,k]})$ is the submap obtained from $\Inverse(W)$ by keeping
	\begin{enumerate}
		\item edges with label $j,\ldots,k$ (explored edges);
		\item faces that have explored edges on both their left and right boundary (explored faces);
		\item other edges incident to explored faces (unexplored edges).
	\end{enumerate}
\end{proposition}

\subsection{Baxter permutations and bipolar orientations}
\label{sect:Baxt_bipol}

We first explain here why our mapping $\bobp$ given in \cref{defn:bobp}  is the same as the bijection $\Psi$ of Bonichon, Bousquet-Mélou and Fusy \cite[Section 3.2]{MR2734180}.

The definition of $\Psi$ can be rephrased in our setting as follows. Let $m \in \mathcal O_n$ be a bipolar orientation. We denote by $m^{-1}$ the symmetric image of $m$ along the vertical axis. Consider the tree $T(m^{-1})$, and set $\Psi(m)$ to be the only permutation $\pi \in \Perms_n$ such that for every $1\leq i \leq n$, the $i$-th edge to be visited in the exploration of $T(m)$ corresponds to the $\pi(i)$-th edge to be visited in the exploration of $T(m^{-1})$. It was observed in \cite[Remark 11]{MR2734180} that the exploration of $T(m^{-1})$ visits edges of $m$ in the same\footnote{Actually a stronger result holds: $T(m^{-1})$ and $T(m^{***})$ are related by a classic bijection of the set of finite trees, which is the counterpart of the \textit{Kreweras complement} for non-crossing partitions: the Lukasiewicz walk of $T(m^{***})$ is the reversal of the height function of $T(m^{-1})$. In particular they have the same scaling limit. This is similar to \cite[Lemma 2.4]{GHS}.} order as the exploration of $T(m^*)$, justifying that $\bobp(m) = \Psi(m)$.

We have the following additional properties of the mapping $\bobp$.

\begin{theorem}[\cite{MR2734180}, Theorems 2 and 3, and Propositions 1 and 4]\label{thm:rotation}
	One can draw $m$ on the diagram of $\bobp(m)$ in such a way that every edge of $m$ passes through the corresponding point of $\bobp(m)$. Moreover, we have the following symmetry properties: 
	\begin{enumerate}
		\item denoting by $\sigma^*$ the permutation obtained by rotating the diagram of $\sigma\in \mathfrak S_n$ clockwise by angle $\pi/2$, we have $\bobp(m^*) = \bobp(m)^*$;
		\item we have $\bobp(m^{-1}) = \bobp(m)^{-1}$.
	\end{enumerate}
\end{theorem}

\subsection{Discrete coalescent-walk processes}
\label{sec:discrete_coal}

This subsection is devoted to defining coalescent-walk processes and our specific model of coalescent-walk processes obtained from tandem walks by the mapping $\wcp$. We then define the permutation and forest naturally associated with a coalescent-walk process.

\begin{definition}
	\label{def:discrete_coal_process}
	Let $I$ be a (finite or infinite) interval of $\Z$. We call \emph{coalescent-walk process} on $I$ a family $\{(Z^{(t)}_s)_{s\geq t, s\in I}\}_{t\in I}$ of one-dimensional walks such that:
	\begin{itemize}
		\item for every $t\in I$, $Z^{(t)}_t=0$;
		\item for $t'\geq t \in I$, if $Z^{(t)}_k\geq Z^{(t')}_k$ (resp.\ $Z^{(t)}_k\leq Z^{(t')}_k$) then $Z^{(t)}_{k'}\geq Z^{(t')}_{k'}$ (resp.\ $Z^{(t)}_{k'}\leq Z^{(t')}_{k'}$) for every $k'\geq k$.
	\end{itemize}
\end{definition}

Note that, as a consequence, if there is a time $k$ such that $Z^{(t)}_k= Z^{(t')}_k$, then $Z^{(t)}_{k'}=Z^{(t')}_{k'}$ for every $k'\geq k$.
In this case, we say that $Z^{(t)}$ and $Z^{(t')}$ are \emph{coalescing} and we call \emph{coalescent point} of $Z^{(t)}$ and $Z^{(t')}$ the point $(\ell,Z^{(t)}_\ell)$ such that  $\ell=\min\{k\geq \max\{t,t'\}|Z^{(t)}_{k}=Z^{(t')}_{k}\}$.

We denote by $\Coals(I)$ the set of coalescent-walk processes on some interval $I$.

\subsubsection{The coalescent-walk process associated with a two-dimensional walk}
\label{sect: bij_walk_coal}
We introduce formally the coalescent-walk processes driven by some specific two-dimensional walks that include tandem walks. Let $I$ be a (finite or infinite) interval of $\Z$. 
Recall that $\Walks_\Steps(I)$ denotes the set of walks considered up to an additive constant, indexed by $I$, and that take their increments in $\Steps$, defined in \cref{eq:admis_steps} page~\pageref{eq:admis_steps}.
\begin{definition}\label{eq:distrib_incr_coal}
Let $W\in\Walks_\Steps(I)$ and denote by $W_t = (X_t,Y_t)$ for $t\in I$. The \emph{coalescent-walk process associated with} $W$
is the family of walks $\wcp(W) = \{Z^{(t)}\}_{t\in I}$, defined for $t\in I$ by $Z^{(t)}_t=0,$ and for all $\ell\geq t$ such that $\ell+1 \in I$,

	\begin{itemize}
		\item if $W_{\ell+1}-W_{\ell}=(1,-1)$ then $Z^{(t)}_{\ell+1}-Z^{(t)}_{\ell}=-1$;
		\item if $W_{\ell+1}-W_{\ell}=(-i,j)$, for some $i,j\geq 0$, then
		\begin{equation*}
		Z^{(t)}_{\ell+1}-Z^{(t)}_{\ell}=
		\begin{cases}
		j, &\quad\text{if}\quad Z^{(t)}_{\ell}\geq0,\\
		i,&\quad\text{if}\quad Z^{(t)}_{\ell}<0\text{ and }Z^{(t)}_{\ell}<-i,\\
		j-Z^{(t)}_{\ell},&\quad\text{if}\quad Z^{(t)}_{\ell}<0\text{ and }Z^{(t)}_{\ell}\geq -i.
		\end{cases}
		\end{equation*} 
\end{itemize}
\end{definition}
Note that this definition is invariant by addition of a constant to $W$. We check easily that $\wcp(W)$ is a coalescent-walk process meaning that $\wcp$ is a mapping $\Walks_\Steps(I) \to \Coals(I)$.
We also set 
$\mathcal{C}_n= \wcp(\mathcal W_n)$ and $\mathcal{C}=\cup_{n\in \Z_{\geq 0}}\mathcal{C}_n.$ For two examples, we refer the reader to the left-hand side of \cref{fig:Coal_process_exemp} (for the case of tandem walks) and to \cref{example_discrete_coal} (for the case of non tandem walks). We finally suggest to the reader to compare the formal \cref{eq:distrib_incr_coal} with the more intuitive definition given in \cref{sect:coal_walk_intro}.

\begin{figure}[htbp]
	\centering
		\includegraphics[scale=.77]{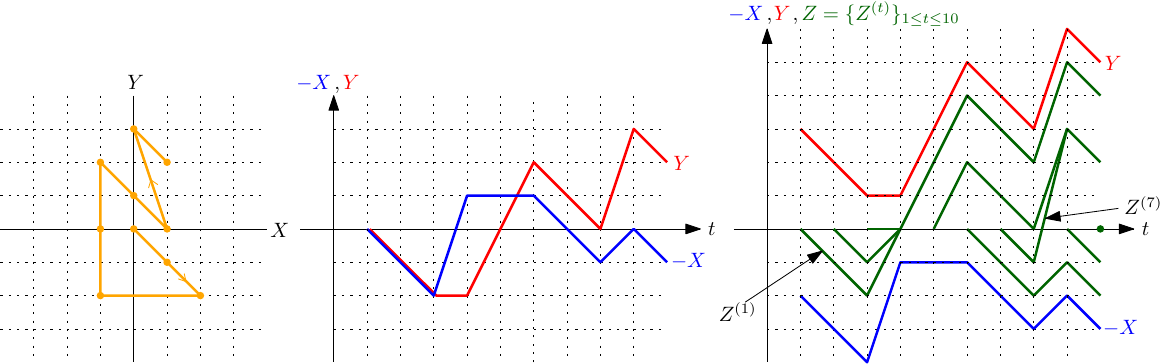}\\
		\caption{Construction of the coalescent-walk process associated with the orange walk $W=(W_t)_{1\leq t\leq 10}$ on the left-hand side. In the middle diagram the two walks $Y$ (in red) and $-X$ (in blue) are plotted. Finally, on the right-hand side the two walks are shifted (one towards the top and one to the bottom) and the ten walks of the coalescent-walk process are plotted in green. \label{example_discrete_coal}}
\end{figure}

\begin{observation}
	 The $y$-coordinates of the coalescent points of a coalescent-walk process obtained in this fashion are non-negative. 
\end{observation}

\subsubsection{The permutation associated with a coalescent-walk process}\label{sect:from_coal_to_perm}

Given a coalescent-walk process $Z = \{Z^{(t)}\}_{t\in I} \in \Coals(I)$ defined on a (finite or infinite) interval $I$, we can define a binary relation $\leq_Z$ on $I$ as follows:

\begin{equation}\label{eq:coal_to_perm}
\begin{cases}i\leq_Z i,\\
i\leq_Z j,&\text{ if }i<j\text{ and }\ Z^{(i)}_j<0,\\
j\leq_Z i,&\text{ if }i<j\text{ and }\ Z^{(i)}_j\geq0.\end{cases}
\end{equation}

\begin{proposition} 
	\label{prop:tot_ord}
	$\leq_Z$ is a total order on $I$. 
\end{proposition}
\begin{proof}
	Since every pair in $I$ is comparable by definition, we just have to check that $\leq_Z$ is an order. By construction it is antisymmetric and reflexive. For transitivity, take $i<j<k$.
	If $Z^{(i)}_k$ and $Z^{(j)}_k$ are both negative,
	then $i\leq_Z k$ and $j\leq_Z k$ and whatever the relative ordering between $i$ and $j$,
	transitivity holds on $\{i,j,k\}$.
	If they are both non-negative, the same reasoning holds. 
	If one of them is non-negative, and one of them is negative, say $Z^{(i)}_k< 0 \leq Z^{(j)}_k$ (the other case is similar), then $i\leq_Z k$ and $k\leq_Zj$. Now by definition of coalescent-walk process, it must be that $Z^{(i)}_j<Z^{(j)}_j = 0$, so that $i\leq_Z j$ and transitivity holds on $\{i,j,k\}$.
\end{proof}

This definition allows to associate a permutation with a coalescent-walk process on the interval $[n]$.

\begin{definition}Fix $n\in\Z_{\geq 0}$. Let $Z = \{Z^{(t)}\}_{i\in [n]} \in \Coals([n])$ be a coalescent-walk process on $[n]$. Denote $\cpbp(Z)$ the unique permutation $\sigma \in \Perms_n$ such that for all $1\leq i, j\leq n$, $$\sigma(i)\leq\sigma(j) \iff i\leq_Z j.$$
\end{definition}

We will furnish an example that clarifies the definition above in \cref{exmp:coal_tree} below.

We now need to introduce some more notation. If $x_1,\dots ,x_n$ is a sequence of distinct numbers, let $\std(x_1,\dots ,x_n)$ be the unique permutation $\pi$ in $\Perms_n$ that is in the same relative order as $x_1,\dots ,x_n,$ \emph{i.e.}, $\pi(i)<\pi(j)$ if and only if $x_i<x_j.$
Given a permutation $\sigma\in\Perms_n$ and a subset of indices $I\subseteq[n]$, let $\pat_I(\sigma)$ be the permutation induced by $(\sigma(i))_{i\in I},$ namely, $\pat_I(\sigma)\coloneqq\text{std}\big((\sigma(i))_{i\in I}\big).$
For example, if $\sigma=87532461$ and $I=\{2,4,7\}$ then $\pat_{\{2,4,7\}}(87532461)=\text{std}(736)=312$.

\begin{proposition}
	\label{prop:patterns}
	Let $\sigma$ be a permutation obtained from a coalescent-walk process $Z=\{Z^{(t)}\}_{1\leq t\leq n}$ via the mapping $\cpbp$. Let $I=\{i_1<\dots<i_k\}\subseteq[n]$. Then $\pat_I(\sigma)=\pi$ if and only if the following condition holds: for all $1\leq \ell< s \leq k,$
	$$Z^{(i_\ell)}_{i_s} \geq 0   \iff  \pi(s)<\pi(\ell).$$
\end{proposition}

\begin{proof}
	The statement immediately follows once one notes that
	\begin{equation*}
	Z^{(i_\ell)}_{i_s} \geq 0  \iff i_s\leq_Z i_\ell \iff \sigma(i_s)\leq\sigma(i_\ell)\iff  \pi(s)<\pi(\ell).
	\end{equation*}
	This is enough to conclude the proof.
\end{proof}

Note that thanks to \cref{prop:patterns}, pattern extraction in the permutation $\cpbp(Z)$ depends only on a finite number of trajectories, a key step towards proving permuton convergence for uniform Baxter permutations.

\subsubsection{The coalescent forest of a coalescent-walk process}\label{sect:from_coal_to_trees} 
Note that given a coalescent-walk process on $[n]$, the plane drawing of the family of the trajectories $\{Z^{(t)}\}_{t\in I}$ identifies a natural tree structure, more precisely, a \emph{$\Z$-planted plane forest}, as per the following definition.
\begin{definition}
	A \emph{$\Z$-planted plane  tree} is a rooted plane tree such that the root has an additional parent-less edge which is equipped with a number in $\Z$ called its (root-)index.
	
	A \emph{$\Z$-planted plane forest} is an ordered sequence of $\Z$-planted plane trees $(T_1,\ldots, T_\ell)$ such that the (root-)indices are weakly increasing along the sequence of trees. A $\Z$-planted plane forest admits an exploration process, which is the concatenation of the exploration processes of all the trees, following the order of the sequence.
\end{definition}

An example of a $\Z$-planted plane forest is given in the middle of \cref{coal_tree} (each tree is drawn with the root on the right; trees are ordered from bottom to top; the root-indices are indicated on the right of each tree).

We give here a formal definition of the $\Z$-planted plane forest corresponding to a coalescent-walk process. For a more informal description, we suggest to look at \cref{coal_tree} and at the description given in \cref{exmp:coal_tree}.

\begin{definition}\label{defn:corr_trees}
Let $Z$ be a coalescent-walk process on a finite interval $I$. Its \emph{labeled forest}, denoted $\fortree(Z)$ for ``labeled forest", is a $\Z$-planted plane forest with additional edge labels in $I$, defined as follows:
\begin{itemize}
	\item the edge-set is $I$, vertices are identified with their parent edge, and the edge $i\in I$ is understood as bearing the label $i$.
	\item For any pair of edges $(i,j)$ with $i<j$, $i$ is a descendant of $j$ if $(j,0)$ is the coalescent point of $Z^{(i)}$ and $Z^{(j)}$.
	\item Children of a given parent are ordered by $\leq_Z$.
	\item The different trees of the forests are ordered such that their root-edges are in increasing $\leq_Z$-order.
	\item The index of the tree whose root-edge has label $i$ is the value $Z^{(i)}_{\max I}$.
\end{itemize}
\end{definition}

\begin{figure}[htbp]
	\centering
		\includegraphics[scale=0.85]{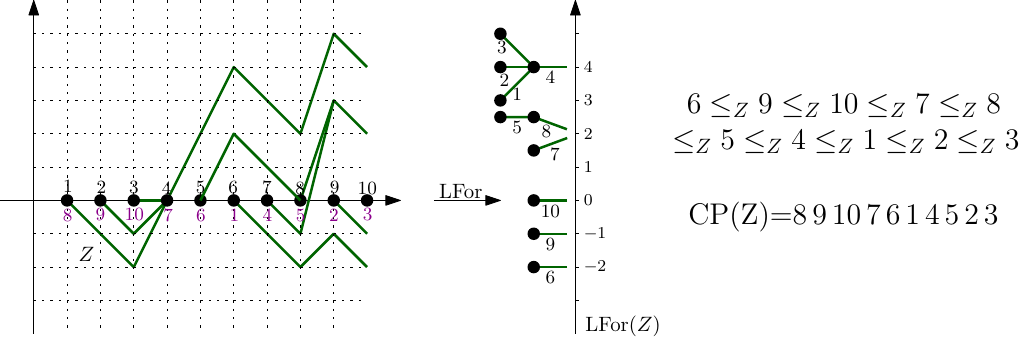}\\
		\caption{In the middle of the picture, the labeled forest $\fortree(Z)$ corresponding to the coalescent-walk process represented on the left that was obtained in \cref{example_discrete_coal}. How this labeled forest is constructed is explained in \cref{exmp:coal_tree}. On the right-hand side we also draw the associated total order $\leq_Z$ and the associated permutation $\cpbp(Z)$.  \label{coal_tree}}
\end{figure}

We have the following result, which is immediate from the properties of a coalescent-walk process.

\begin{proposition}\label{prop:fortree_cpbp}
	$\fortree(Z)$ is a $\Z$-planted plane forest, equipped with a labeling of its edges by the values of $I$.
	Moreover the total order $\leq_Z$ on $I$ coincides with the total order given by the exploration process of the labeled forest $\fortree(Z)$.
\end{proposition}

\begin{remark}\label{rk:cpbp_through_fortree}
	In the case where $I=[n]$ for some $n\in\Z_{\geq 0}$, the permutation $\pi = \cpbp(Z)$ is readily obtained from $\fortree(Z)$: for $1\leq i \leq n$, $\pi(i)$ is the position in the exploration of $\fortree(Z)$ of the edge with label $i$.
\end{remark}

\begin{exmp}\label{exmp:coal_tree}
	 \cref{coal_tree} shows the labeled forest of trees $\fortree(Z)$ corresponding to the coalescent-walk process $Z=\{Z^{(t)}\}_{t\in[10]}$ plotted on the left-hand side (where the labeled forest is plotted from bottom to top). It can be obtained by marking with ten dots the points $\{(t,Z^{(t)}_t=0)\}_{t\in[10]}$. The edge structure of the trees in $\fortree(Z)$ is given by the green lines starting at each dot, and interrupted at the next dot. The lines that go to the end uninterrupted (for example this is the case of the line starting at the fourth dot), correspond to the root-edges of the different trees. The indices of these root-edges are determined by the final heights of the corresponding interrupted lines.   
	The plane structure of $\fortree(Z)$ is inherited from the drawing of $Z$ in the plane. 
	
	We determine the order $\leq_Z$ by considering the exploration process of the labeled forest:  $6\leq_Z 9\leq_Z 10\leq_Z 7 \leq_Z 8 \leq_Z 5 \leq_Z 4 \leq_Z 1 \leq_Z 2 \leq_Z 3$. As a result, $\cpbp(Z) =8\,9\,10\,7\,6\,1\,4\,5\,2\,3$. 
	
	Equivalently, we can pull back (on the points $(t,Z^{(t)}_t=0)$ of the coalescent-walk process) the position of the edges in the exploration process (these positions are written in purple), and then $\cpbp(Z)$ is obtained by reading these numbers from left to right. 	
	
\end{exmp}

\subsection{From walks to Baxter permutations via coalescent-walk processes}
\label{sect:equiv_bij}
We are now in position to prove the main result of this section, that is \cref{thm:diagram_commutes}.

We are going to show that $\bobp = \cpbp\circ\wcp \circ \bow.$
The key ingredient is to show that the dual tree $T(m^*)$ of a bipolar orientation can be recovered from its encoding two-dimensional walk by building the associated coalescent-walk process $Z$ and looking at the labeled forest $\fortree(Z)$.
More precisely, let $W = (W_t)_{t\in [n]} = \bow (m)$ be the walk encoding a given bipolar orientation $m$, and $Z = \wcp(W)$ be the corresponding coalescent-walk process. Then the following result holds.

\begin{proposition}\label{prop:eq_trees} 
The following are equal:
	\begin{itemize}
	\item the dual tree $T(m^*)$ with edges labeled according to the order given by the exploration of $T(m)$;
	\item the tree obtained by attaching all the edge-labeled trees of $\fortree(Z)$ to a common root.
	\end{itemize}
\end{proposition}
\cref{thm:diagram_commutes} then follows immediately, by construction of $\bobp(m)$ from $T(m^*)$ and $T(m)$ (\cref{thm:bobp}) and of $\cpbp(Z)$ from $\fortree(Z)$ (\cref{rk:cpbp_through_fortree}). \cref{prop:eq_trees} is illustrated in an example in \cref{fig:bip_orient_with_coal}.

\begin{figure}[tbh]
	\centering
	\includegraphics[scale=0.85]{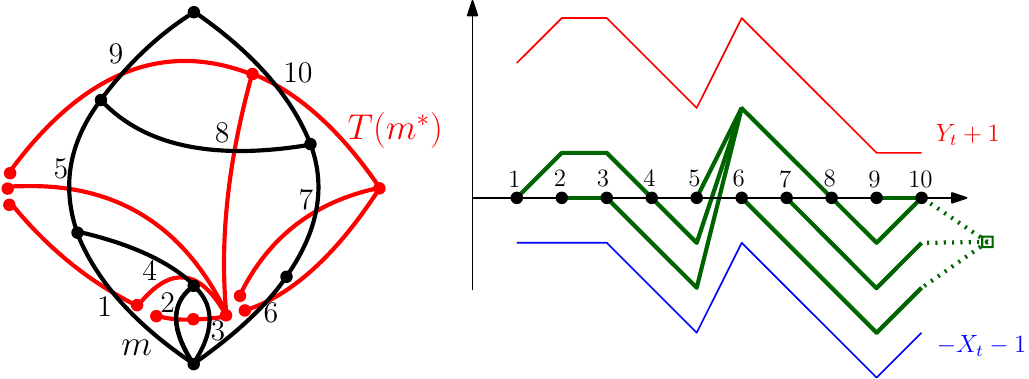}
	\caption{In the left-hand side the map $m$ from \cref{fig:bip_orient_and_perm} with the dual tree $T(m^*)$ in red with edges labeled according to the order given by the exploration of $T(m)$. In the right-hand side the associated coalescent-walk process $Z=\wcp\circ\bow(m)$. Note that the red tree (with its labeling) and the green tree (with its labeling) are equal.
		\label{fig:bip_orient_with_coal}}
\end{figure}

An interesting corollary of \cref{prop:eq_trees}, useful for later purposes (see for instance \cref{sect:anti-invo}), is the following result. Given a coalescent-walk process $Z$, we introduce the discrete local time process $L_Z = \left(L_Z^{(i)}(j)\right)$, $1\leq i \leq j \leq n$, defined by
\begin{equation}
\label{eq:local_time_process}
L_Z^{(i)}(j)=\#\left\{k\in [i,j]\middle|Z^{(i)}_k=0\right\}.
\end{equation}

\begin{corollary}\label{cor:local_time}
	Let $(m,W,Z,\sigma)$ be objects of size $n$ in $\mathcal{O}\times\mathcal{W}\times\mathcal{C}\times\mathcal{P}$ connected by the commutative diagram in \cref{eq:comm_diagram}. 
	Then the height process $(X^*_i)_{i\in [n]}$ of $T(m^*)$ is equal to
	$\left(L_Z^{(\sigma^{-1}(i))}(n)\right)_{i\in [n]}$. In other words, $$X^*_{\sigma(i)}=L_Z^{(i)}(n)-1,\quad{i\in [n]}.$$
\end{corollary}

\begin{proof}
	We start by recalling that labeling $T(m)$ and $T(m^*)$ according to their exploration processes, then $\sigma(i)$ is equal to the edge in $T(m^*)$ dual to the edge $i$ in $T(m)$.
	
	Now by definition, the height of the edge $\sigma(i)$ in $T(m^*)$ is equal to the number of edges in the unique paths from $\sigma(i)$ to the root of $T(m^*)$ (without counting the edge $\sigma(i)$). Using \cref{prop:eq_trees} and the remark at the beginning of the proof, the latter quantity is equal to $L_Z^{(i)}(n)-1$.
\end{proof}

We make the following observation useful for later purposes.

\begin{observation}\label{obs:ancestry line}
	Consider the tree $T(m^*)$ with edges labeled according to its exploration process. Let $P_i$ be the ancestry line in $T(m^*)$ of the edge $\sigma(i)$, i.e.\ the sequence of edges in the unique path in $T(m^*)$ from the edge $\sigma(i)$ to the root of $T(m^*)$.	
	Then, for $1\leq i \leq j\leq n$, $L_Z^{(i)}(j)$ is equal to the number of edges in $P_i$ with a $T(m)$-label weakly smaller than $j$.	
\end{observation}

The rest of this section is devoted to proving \cref{prop:eq_trees}. We will proceed by induction on the size of $m$. Because a restriction of a walk in $\mathcal W$ is not necessarily in $\mathcal W$, we will work on the larger space $\mathfrak W_\Steps$ and the corresponding space of maps $\Maps(I)$.

We start with a few definitions. Let $I$ be a finite interval and take $m\in \Maps(I)$, recalling the definition of $\Maps(I)$ from \cref{sect:KMSW}. In particular, $m$ is a bipolar orientation with explored edges labeled by $I$, and possibly unexplored edges at the top of its right boundary and at the bottom of its left boundary. The edge labeled by $\max I$ is called active. For $i\in I$, denote by $e_i$ the explored edge bearing label $i$ in $m$. Denote $e_i^*$ its dual edge in the map $m^*$.
%\gregory{explain better the nature of $m$}
\begin{definition}\label{def:F_k}
	The $\Z$-planted, edge-labeled, plane forest $\DualForest(m)$ is the restriction of $T(m^*)$ to edges $(e^*_i)_{i\in I}$. More precisely,
	\begin{enumerate}
		\item The edge-set of $\DualForest(m)$ is $(e^*_i)_{i\in I}$.
		\item An edge $e^*_i$ is on top of $e^*_j$ if it is the case in $T(m^*)$.
		\item Parent-less edges are planted in $\Z$ according to the following rule: 
		\begin{enumerate}
			\item if $e^*_i$ is dual to an explored edge on the right boundary of $m$, then it is indexed by its non-positive height relative to the active edge;
			\item otherwise, $e_i$ is at the left of an inner face $f$ of $m$, so that the parent of $e^*_i$ in $T(m^*)$ must be dual to an unexplored edge which is at the top-right edge of $f$. We index $e^*_i$ by the (positive) height of this edge relative to the active edge.
		\end{enumerate}
	\end{enumerate}
\end{definition}

This construction is illustrated in \cref{fig:bip_orient_m_k_F_k}.
  \begin{figure}[tbh]
  	\centering
  	\includegraphics[scale=0.9]{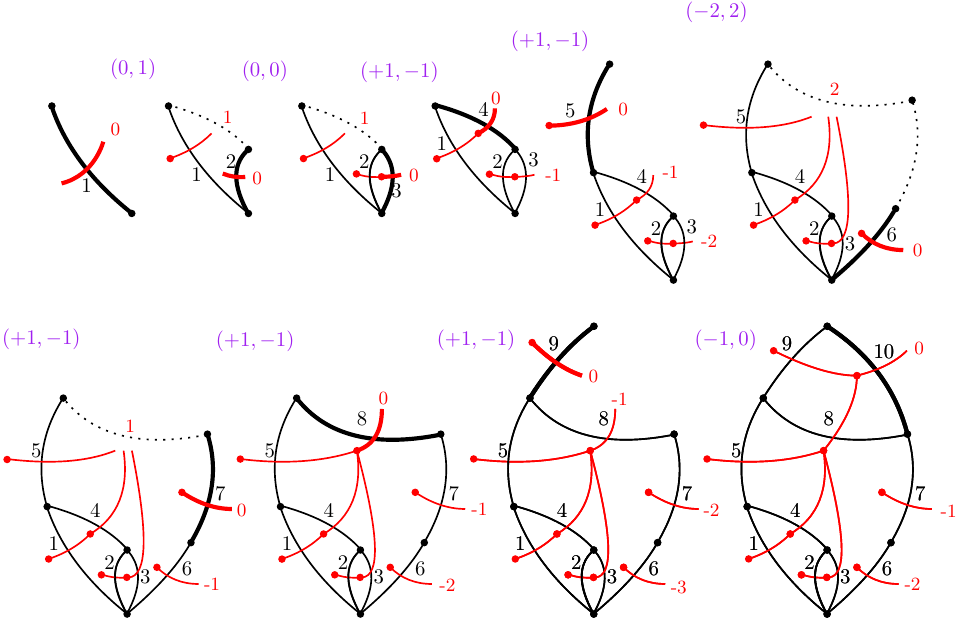}
  	\caption{The forests $\DualForest(m_1),\ldots,\DualForest(m_{10})$ drawn on top of the maps $m_1,\ldots,m_{10}$ of \cref{fig:bip_orient_from_walk}, with the increments of the walk $W$ in purple. We notice that for $m_{10}=m \in \mathcal O$, $\DualForest(m)$ is just $T(m^*)$ disconnected at its root.
  		\label{fig:bip_orient_m_k_F_k}}
  	\end{figure} 	
We remark that if $m\in \mathcal O$ (this is for instance the case of the last picture in \cref{fig:bip_orient_m_k_F_k}), then $\DualForest(m)$ is simply obtained by disconnecting $T(m^*)$ at its root, labeling all edges according to their position in the exploration of $T(m)$, and indexing root-edges by their (non-positive) height on the right boundary of $m$. Thus, the next result implies \cref{prop:eq_trees}.
  
  \begin{proposition}\label{prop:fortree}
	  	Let $I$ be a finite interval and $W \in \Walks_A(I)$. Set $Z = \wcp(W)$ and $m = \Inverse(W)$. Then
	  	\[\DualForest(m) = \fortree(Z).\]
  \end{proposition}
  \begin{proof}
  	 We will proceed by induction on the size of $I$. If $I = \{k\}$, both $\DualForest(m)$ and $\fortree(Z)$ consist of a single edge planted at $0$ with label $k$. For the induction step, assume $I = [j,k+1]$ with $j\leq k$, and let $m' = \Inverse(W|_{[j,k]})$. We will compare how $F = \DualForest(m)$ is obtained from $F' = \DualForest(m')$ and how $\fortree(Z)$ is obtained from $\fortree(Z|_{[j,k]})$, distinguishing two cases according to the increment $W_{k+1}- W_k$ (compare the following with \cref{fig:comparison_forest_evolution}). 
  	
  	\noindent \underline{\emph{First case:}} $W_{k+1} - W_k = (+1,-1)$. 
  	In this case, $m$ is obtained from $m'$ by moving the active edge up by one. So that $F$ has an additional edge (dual to $e_{k+1}$) compared to $F'$. This edge is parent-less, has index $0$, and is now the parent of all (previously parent-less) edges that used to have index $1$. All other parent-less edges remain parent-less and their indices decrease by $1$.
  	
  	On the other hand, the coalescent-walk process $Z$ is obtained from $Z|_{[j,k]}$ by adding an new walk starting at zero at time $k+1$ and adding to all the walks in $Z|_{[j,k]}$ a step equal to $-1$. This makes all the walks that were at time $k$ at height $1$ to coalesce with the new walk. 
  	
  	Using \cref{def:F_k}, the arguments above show that if $F'$ is equal to $\fortree(Z|_{[j,k]})$, then $F'$ is equal to $\fortree(Z)$.
  	
  	\noindent \underline{\emph{Second case:}} $W_{k+1} - W_k = (-i,j)$, for $i\geq 0, j\geq 0$. In this case, $m$ is obtained from $m'$ by gluing a face $f$ of left-degree $i+1$ and right-degree $j+1$ to the right boundary of $m'$. The previous active edge $e_k$ is at the top-left of $f$, and the new active edge $e_{k+1}$ at the bottom-right of $f$. Parent-less edges of positive index in $F'$ are still parent-less edges in $F$, and their index increases by $j$. Parent-less edges of index smaller than $-i$ are still parent-less edges, and their index increases by $i$.  Parent-less edges of $F'$ with index $0,-1,\ldots, - i$ are children in $T(m^*)$ of the edge dual to the top-right edge $\overline e$ of $f$. We now have two sub-cases.
  	
  	\begin{enumerate}
  		\item If $j=0$, then $\overline e = e_{k+1}$ is the new active edge, so parent-less edges of $F'$ with index $0,\ldots,-i$ are now children of the new parent-less edge $e_{k+1}$ at height $0$.
  		\item If $j>0$, then $\overline e$ is not an explored edge, so parent-less edges of $F'$ with index $0,\ldots,-i$ are still parent-less edges, but their index is now the height of $\overline e$ which is $j$. Finally, $e_{k+1}$ is a new parent-less edge of index $0$ with no children.
  	\end{enumerate}	
  
  	On the other hand, the coalescent-walk process $Z$ is obtained from $Z|_{[j,k]}$ by adding an new walk starting at zero at time $k+1$ and adding 
  	\begin{itemize}
  		\item to all the walks in $Z|_{[j,k]}$ that were strictly positive at time $k$ a step equal to $j$;
  		\item to all the walks in $Z|_{[j,k]}$ that were at height $\ell\in\{0,-1,\dots,-i\}$ at time $k$ a step equal to $j+\ell$;
  		\item to all the other walks in $Z|_{[j,k]}$ a step equal to $-i$;
  	\end{itemize}
	This has the following consequence: 
	\begin{enumerate}
		\item If $j=0$, all the walks that were at time $k$ at height $\ell\in\{0,-1,\dots,-i\}$ coalesce with the new walk at zero; 
		\item If $j>0$, all the walks that were at time $k$ at height $\ell\in\{0,-1,\dots,-i\}$ coalesce among themselves at height $j$.
	\end{enumerate}
	
  	Using \cref{def:F_k}, the arguments above show that if $F'$ is equal to $\fortree(Z|_{[j,k]})$, then $F'$ is equal to $\fortree(Z)$ (both if $j=0$ and $j>0$).
  \end{proof}
\begin{figure}[tbh]
	\centering
	\includegraphics[scale=0.73]{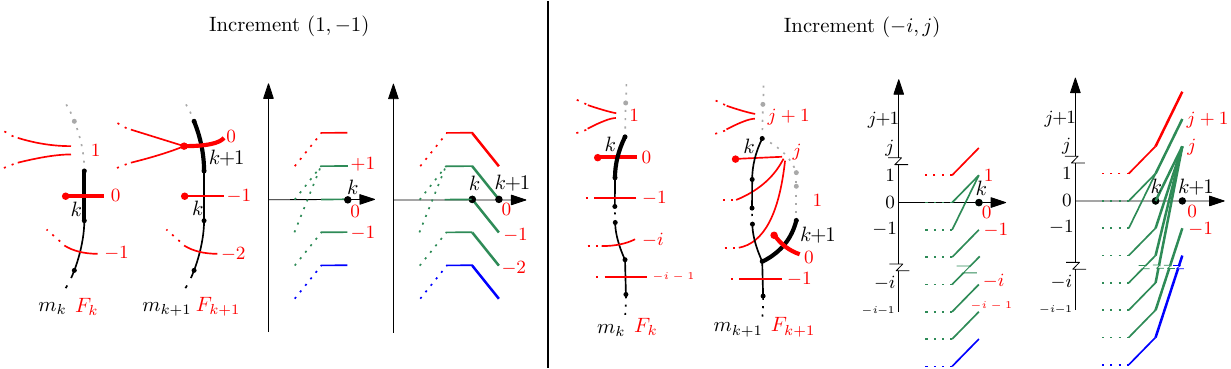}
	\caption{\label{fig:comparison_forest_evolution}
		The evolution between the forests $F=F_k$ and $F'=F_{k+1}$ in parallel with the increments of the coalescent-walk process $Z$ between times $k$ and $k+1$. On the left-hand side the first case of the proof of \cref{prop:fortree} and on the right-hand side the second one (subcase 2).}
\end{figure}

We collect for future use the following consequence.

\begin{corollary}\label{coroll:vertex degree}
	Let $I$ be a finite interval and $W \in \Walks_A(I)$. Set $Z = \wcp(W)$ and $m = \Inverse(W)$. Fix $i\in I$ and call $e$ the corresponding edge in $m$. Let $v$ the top vertex of $e$. 
	
	If $j>i$ is the smallest index greater than $i$ such that $Z^{(i)}_j=0$ and $Z^{(i)}_{j+1}<0$ (provided that such a $j$ exists), and $s>j+1$ is the smallest index greater than $j+1$ such that $Z^{(i)}_s\geq0$ (provided that such an $s$ exists) then
	the number of outgoing edges of $v$ in $m$ is bounded by $s-j-1$.
\end{corollary}

\begin{proof}
	It is enough to note that the only edges of $m$ that can be outgoing edges of $v$ are the ones corresponding to the indices $j+1,j+2,\dots, s-1$.
\end{proof}

\subsection{Anti-involutions for discrete coalescent-walk processes and the trees of bipolar orientations}
\label{sect:anti-invo}

We now want to investigate the relations between the four trees $T(m)$, $T(m^*)$, $T(m^{**})$, $T(m^{***})$ of a bipolar orientation $m$ (and its dual maps) and some corresponding coalescent-walk processes. The results presented in this section will be also useful for \cref{sec:final}.

\subsubsection{The reversed coalescent-walk process}

Let $m$ be a bipolar orientation,  $W=\bow(m)$ be the corresponding tandem walk, and $Z=\wcp\circ\bow(m)=\wcp(W)$ be the corresponding coalescent-walk process. We set

\begin{itemize}
	\item $LT(m^*)$ to be the tree $T(m^*)$ with edges labeled according to the order given by the exploration of $T(m)$.
	\item $\labtree(Z)$ to be the tree obtained by attaching all the edge-labeled trees of $\fortree(Z)$ to a common root.
\end{itemize}

We saw in \cref{prop:eq_trees} that $LT(m^*)=\labtree(Z)$.
We now want to recover from the walk $W$ (and a new associated coalescent-walk process) the tree $LT(m^{***})$, i.e.\ the tree $T(m^{***})$ with edges labeled according to the order given by the exploration of $T(m^{**})$. For that, we have to consider the following.

\begin{definition}
	Fix $n\in\Z_{>0}.$
	Given a one dimensional walk $X=(X_t)_{t\in [n]}$ we denote by $\cev{X}$ the time reversed walk $(X_{n+1-t})_{t\in [n]}$.
	Given a two-dimensional walk $W=(X,Y)=(X_t,Y_t)_{t\in [n]}$, we denote by $\cev{W}$ the time reversed and coordinates swapped walk $(\cev{Y},\cev{X})$.
\end{definition}

\begin{proposition}\label{prop:rev_coal_prop}
	Let $m$ be a bipolar orientation and  $W=\bow(m)$ be the corresponding walk. Consider the walk $\cev{W}$ and the corresponding coalescent-walk process $\cev{Z}\coloneqq\wcp(\cev{W})$.
	Then 
	$$\bow (m^{**})=\cev{W} \quad\text{and}\quad LT(m^{***})= \labtree(\cev{Z}).$$
\end{proposition}

An example of the coalescent-walk process $\cev{Z}=\wcp(\cev{W})$ is given on the left-bottom side of \cref{fig:anti-invo} in the case of the bipolar orientation $m$ considered in \cref{fig:bip_orient}, that is the map that we always used for our examples.

\begin{proof}[Proof of \cref{prop:rev_coal_prop}]
	Note that using \cref{prop:eq_trees} with the map $m^{**}$ (instead of $m$) and the walk $W^{**} = \bow (m^{**})$ (instead of $W = \bow (m)$) we obtain that 
	$LT(m^{***})=\labtree(Z^{**}),$
	where $Z^{**} = \wcp(W^{**})$. In order to conclude, it is enough to note that $W^{**}=\cev{W}$ and so $\labtree(Z^{**})=\labtree(\cev{Z})$.
\end{proof}

\subsubsection{Two anti-involution mappings}

We now know that given a bipolar orientation $m$ and the corresponding walk $W=\bow(m)$, we can read the trees $LT(m^*)$ and $LT(m^{***})$ in the coalescent-walk processes $Z=\wcp(W)$ and $\cev{Z}=\wcp(\cev{W})$ respectively.
Obviously, considering the bipolar map $m^*$ and the corresponding walk $W^*=\bow(m^*)$, we can read the trees $LT(m^{**})$ and $LT(m^{****})=LT(m)$ in the coalescent-walk processes $Z^*=\wcp(W^*)$ and $\cev{Z^*}=\wcp(\cev{W^*})$ respectively. 

Actually, we can determine the walk $W^*=\bow(m^*)$ directly from the coalescent-walk processes $Z$ and $\cev{Z}$, as explained in \cref{prop:anti-inv} below. 
We recall that the discrete local time process $L_Z = \left(L_Z^{(i)}(j)\right)$, $1\leq i \leq j \leq n$,  was defined by
$L_Z^{(i)}(j)=\#\left\{k\in [i,j]\middle|Z^{(i)}_k=0\right\}.$ We also recall (see \cref{thm:rotation}) that $\sigma^*$ denotes the permutation obtained by rotating the diagram of a permutation $\sigma$ clockwise by angle $\pi/2$.

\begin{proposition}\label{prop:anti-inv}
	Let $m$ be a bipolar orientation of size $n$. Set $W^*=(X^*,Y^*)=\bow(m^*)$, $\sigma=\bobp(m)$, $W=\bow(m)$, $Z=\wcp(W)$ and $\cev{Z}=\wcp(\cev{W})$. Then
	\begin{equation*}
	(X^*_{i})_{i\in[n]}=\left(L_Z^{(\sigma^{-1}(i))}(n)-1\right)_{i\in [n]}\quad\text{and}\quad
	(Y^*_{i})_{i\in[n]}=\left(L_{\cev{Z}}^{(\sigma^{*}(i))}(n)-1\right)_{i\in [n]}.
	\end{equation*} 
\end{proposition}

\begin{proof}
	The fact that $(X^*_{i})_{i\in[n]}=\left(L_Z^{(\sigma^{-1}(i))}(n)-1\right)_{i\in [n]}$ follows from \cref{cor:local_time}.
	Let $Z^{**}=\wcp\circ\bow(m^{**})$. Using \cref{cor:local_time} with the map $m^{**}$ (instead of $m$) and \cref{rem:height_process} we obtain that 
	$\left(Y^*_{n+1-i}\right)_{i\in[n]}=\left(L_{Z^{**}}^{OP(m^{**})^{-1}(i)}(n)-1\right)_{i\in [n]},$
	and so, since we know from \cref{prop:rev_coal_prop} that $Z^{**}=\cev{Z}$, we conclude that
	$(Y^*_{i})_{i\in[n]} = \left(L_{\cev{Z}}^{OP(m^{**})^{-1}(n+1-i)}(n)-1\right)_{i\in [n]}.$
	It remains to show that $OP(m^{**})^{-1}(n+1-i)=\sigma^{*}(i)$. From \cref{thm:rotation} we have that 
		$$OP(m^{**})^{-1}(n+1-i)=(OP(m)^{**})^{-1}(n+1-i)=OP(m)^{*}(i)=\sigma^{*}(i),$$
		where in the second equality we used that, for any permutation $\pi$ of size $n$, $\pi(n+1-i)=(\pi^{-1})^{***}(i)$. This ends the proof.
\end{proof}

From the above proposition it is meaningful to consider the following two mappings. 
We set $\wpc$ to be the mapping from the set of walks $\mathcal{W}$ to the set of pairs of discrete coalescent-walk processes $\mathcal{C}\times\mathcal{C}$, defined by
$$\wpc(W)=(\wcp(W),\wcp(\cev{W}))\quad\text{for all $W\in\mathcal{W}$}.$$
We also set $\pcw$ to be the mapping from the set of pairs of coalescent-walk processes in $\wpc(\mathcal{W})\subseteq \mathcal C \times \mathcal C$ to the set of walks $\mathcal{W}$, defined by
$$\pcw(Z,\cev{Z})=\left(L_Z^{(\sigma^{-1}(i))}(|Z|)-1,L_{\cev{Z}}^{(\sigma^{*}(i))}(|Z|)-1\right)_{i\in [|Z|]},\quad\text{for all $(Z,\cev{Z})\in \wpc(\mathcal{W})$},$$
where, if $(Z,\cev{Z})=\wpc(W)$, then $\sigma=\bobp\circ\bow^{-1}(W)=\cpbp\circ\wcp(W)$, i.e.\ $\sigma$ is the Baxter permutation associated with $W$.

\medskip

From our constructions and \cref{prop:rev_coal_prop,prop:anti-inv} we have the following.

\begin{theorem} \label{thm:discrete_invo}
	Fix $W^0\in\mathcal{W}$. Consider the following sequence
	$$W^0\xmapsto{ \wpc }(Z_1,\cev{Z}_1)\xmapsto{ \pcw }W^1\xmapsto{ \wpc }(Z_2,\cev{Z}_2)\xmapsto{ \pcw }W^2\xmapsto{ \wpc }\dots\xmapsto{ \wpc }(Z_4,\cev{Z}_4)\xmapsto{ \pcw }W^4.$$
	Then setting $m=\bow^{-1}(W^0)$, i.e.\ the bipolar map associated with $W^0$, we have that
	\begin{align*}
	&W^i=\bow\left(m^{*^{i}}\right),\quad\text{for}\quad i=0,1,2,3,4,\\
	\Big(\labtree\left(Z_i\right),&\labtree\left(\cev{Z}_i\right)\Big)=\left(LT\left(m^{*^i}\right),LT\left(m^{*^{i+2}}\right)\right),\quad\text{for}\quad i=1,2,3,4.
	\end{align*}
	Therefore (since $m=m^{****}$) $W^0=W^4$ and so $(\pcw\circ\wpc)^4=\Id=(\wpc\circ\pcw)^4$.
\end{theorem}

The coalescent-walk processes $Z_1$, $\cev{Z}_1$, $Z_2$ and $\cev{Z}_2$, and the corresponding edge-labeled trees $LT(m^{*})$, $LT(m^{***})$, $LT(m^{**})$ and $LT(m)$, in the specific case of our running example, are plotted in \cref{fig:anti-invo}. 

\begin{figure}
	\includegraphics[scale=.45]{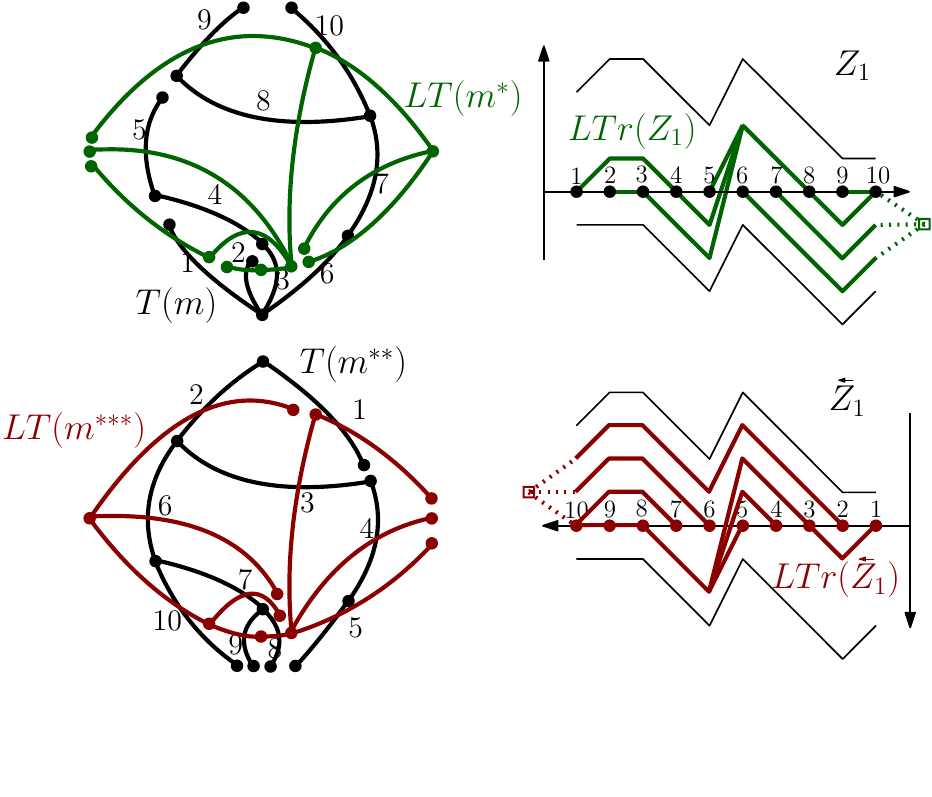}
	\includegraphics[scale=.45]{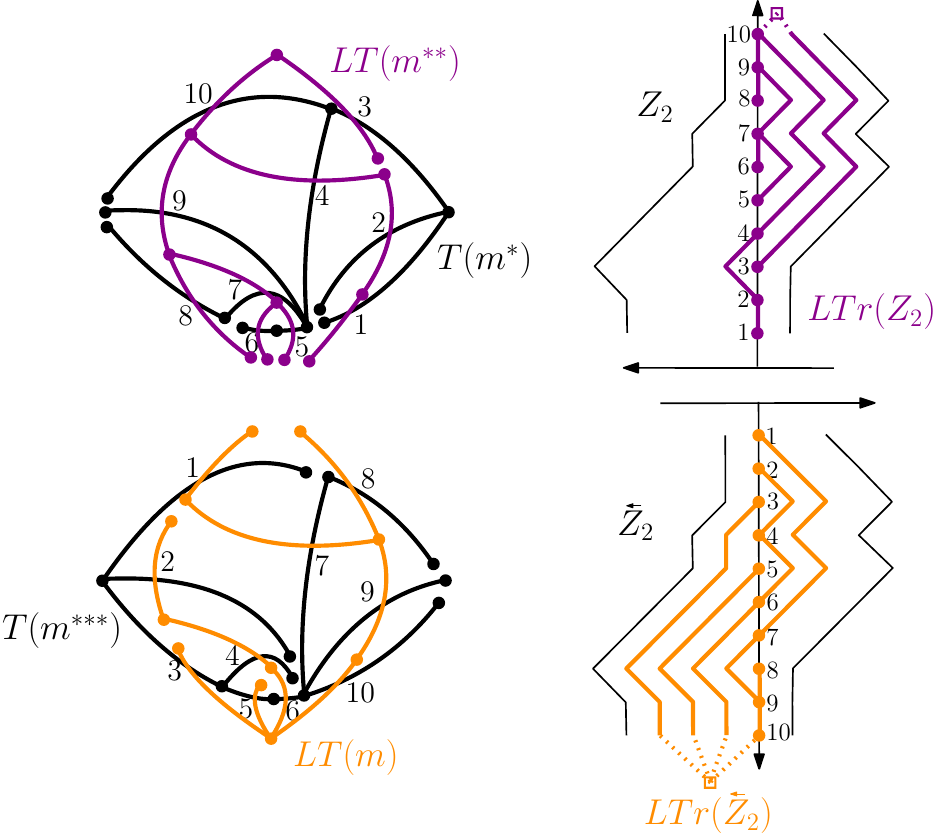}
	\caption{\label{fig:anti-invo} On the left-hand side the coalescent-walk processes $Z_1$ and $\cev{Z}_1$ and the corresponding edge-labeled trees $LT(m^{*})$ and $LT(m^{***})$. On the right-hand side the coalescent-walk processes $Z_2$ and $\cev{Z}_2$, and the corresponding edge-labeled trees $LT(m^{**})$ and $LT(m)$. We oriented the coalescent-walk processes in such a way that the comparison between trees is convenient.}
\end{figure}

\section{Local limit results}
\label{sec:local}

The goal of this section is to prove \cref{thm:local}, a joint quenched local convergence for the four families of objects put in correspondence by the mappings in the commutative diagram in \cref{eq:comm_diagram}.

\subsection{Mappings of the commutative diagram in the infinite-volume}\label{sect:infinite_bij}

Let $I$ be a (finite or infinite) interval of $\Z$. Recall that $\Walks_\Steps(I)$ is the subset of two-dimensional walks with time space $I$ and with increments in $\Steps$ (that was defined in \cref{eq:admis_steps} page~\pageref{eq:admis_steps}). Recall also that $\Coals(I)$ is the set of coalescent-walk processes on $I$, and that the mapping $\wcp : \Walks_\Steps(I) \to \Coals(I)$ sends walks to coalescent-walk processes, both in the finite and infinite-volume case.

In this section we extend the mappings $\cpbp$ and $\Inverse$ defined earlier to infinite-volume objects. Recall that the mapping $\Inverse$ was defined on a set of unconditioned walks in such a way that the restriction to $\mathcal W$ is the inverse of $\bow$ (see \cref{sect:KMSW}).

We start with the mapping $\cpbp$. We denote by $\Perms(I)$ the set of total orders on $I$, which we call permutations on $I$. This terminology makes sense because for $n\geq 1$, the set $\Perms_n$ of permutations of size $n$ is readily identified with $\Perms([n])$, by the mapping $\sigma \mapsto \preccurlyeq_\sigma$, where 
\begin{equation*}
i \preccurlyeq_\sigma j\qquad\text{if and only if}\qquad \sigma(i)\leq\sigma(j)\;.
\end{equation*} 
Then we can extend $\cpbp: \Coals(I) \to \Perms(I)$ to the case when $I$ is infinite, by setting $\cpbp(Z)$ to be the total order $\leq_Z$  on $I$ defined in \cref{sect:from_coal_to_perm}. This is consistent, through the stated identification, with our previous definition of $\cpbp$ when $I = [n]$.

Finally, we also extend $\Inverse$ to the infinite-volume case. We need first to clarify what is our definition of \emph{infinite planar maps}. From now on,  a planar map is a gluing along edges of a finite number of finite polygons.

\begin{definition}\label{defn:Inf_map}
An \emph{infinite oriented quasi-map} is an infinite collection of finite polygons with oriented edges, glued along some of their edges in such a way that the orientation is preserved. 
In the case the graph corresponding to an infinite oriented quasi-map is locally finite (i.e.\ every vertex has finite degree) then we say that it is an \emph{infinite oriented map}. 

We call \emph{boundary} of an infinite oriented map, the collection of edges of the finite polygons that are not glued with any other edge.
An infinite oriented map $m$ is 
\begin{itemize}
	\item \textit{simply connected} if for every finite submap $f\subset m$ there exists a finite submap $f'\subset m$ which is a planar map (i.e.\ is simply connected) with $f\subset f'\subset m$;
	\item \textit{boundaryless} if the boundary of $m$ is empty.
\end{itemize}
When these two conditions are verified, we say that $m$ is an \textit{infinite map of the plane}\footnote{We point out that when the map is locally finite, it is well defined as a topological manifold, and the combinatorial notions of simple connectivity and boundarylessness defined above are equivalent to the topological ones. By the classification of two-dimensional surfaces, $m$ is an infinite map of the plane if and only if it is homeomorphic to the plane.}.
\end{definition}

Let $I$ be an infinite interval and $w\in \Walks_\Steps(I)$. We recall that we can view a finite bipolar orientation as a finite collection of finite inner faces, together with an adjacency relation on the oriented edges of these faces. This allows us to construct $\Inverse(w)$ as a projective limit, as follows: from \cref{prop:inverse}, if $J$ and $J'$ are finite intervals and $J'\subset J \subset I$ then $\Inverse(w|_{J'})$ is a submap of $\Inverse(w|_J)$ defined in a unique way. This means that the face set of $\Inverse(w|_{J'})$ is included in that of $\Inverse(w|_J)$, and that two faces of $\Inverse(w|_{J'})$ are adjacent if and only if they are adjacent through the same edge in $\Inverse(w|_J)$.
Then, taking $J_n$ a growing sequence of finite intervals such that $I = \cup_n J_n$, we set $\Inverse(w) := \mathrlap {\  \uparrow} \bigcup_{n\geq 0} \Inverse(w|_{J_n})$ to be an infinite collection of finite polygons, together with a gluing relation between the edges of these faces, i.e.\ $\Inverse(w)$ is an infinite oriented quasi-map.

\subsection{Random infinite limiting objects}\label{sect:inf_loc_obj}
We define here what will turn out to be our local limiting objects. Recall that $\nu$ denotes the following probability distribution on $A$ (defined in \cref{eq:admis_steps} page~\pageref{eq:admis_steps}):
\begin{equation}\label{eq:walk_distrib}
\nu = \frac 12 \delta_{(+1,-1)} + \sum_{i,j\geq 0} 2^{-i-j-3}\delta_{(-i,j)},\quad\text{where $\delta$ denotes the Dirac measure,}
\end{equation}
and recall that $\overline{\bm W} = (\overline{\bm X},\overline{\bm Y}) = (\overline{\bm W}_t)_{t\in \Z}$ is a bi-directional two-dimensional random  walk with step distribution $\nu$, having value $(0,0)$ at time zero. The interest of introducing the probability measure $\nu$ resides in the following way of obtaining uniform elements of $\mathcal W_n$ as conditioned random walks. For all $n\geq 1$, let $\Walks_{A,\exc}^n \subset \Walks_A([0,n-1])$ be the set of two-dimensional walks $W=(W_t)_{0\leq t\leq n-1}$ in the non-negative quadrant, starting and ending at $(0,0)$, with increments in $A$. Notice that for $n\geq 1$, the mapping $\mathcal \Walks_{A,\exc}^{n+2} \to \mathcal W_{n}$ removing the first and the last step, i.e.\ $W \mapsto (W_{t})_{1\leq t \leq n}$,
is a bijection. A simple calculation then gives the following (obtained also in \cite[Remark 2]{MR3945746}):

\begin{proposition}\label{prop:unif_law}
	Conditioning on $\{(\overline{\bm W}_t)_{0\leq t\leq {n+1}}  \in \mathcal \Walks_{A,\exc}^{n+2}\}$, the law of  $(\overline{\bm W}_t)_{0\leq t \leq n+1}$ is the uniform distribution on $\Walks_{A,\exc}^{n+2}$, and the law of $(\overline{\bm W}_t)_{1\leq t \leq n}$ is the uniform distribution on $\mathcal W_n$.
\end{proposition} 

Let $\overline{\bm Z} = \wcp(\overline{\bm W})$ be the corresponding coalescent-walk process.

\begin{proposition}\label{prop:trajectories_are_rw}
	For every $t\in \Z$, $\overline{\bm Z}^{(t)}$ has the distribution of a random walk with the same step distribution as $\overline{\bm Y}$ (which is the same as that of $-\overline{\bm X}$).
\end{proposition}

\begin{remark}
	Recall that the increments of a walk of a coalescent-walk process are not always equal to the increments of the corresponding walk (see the last case considered in \cref{eq:distrib_incr_coal}). The statement of \cref{prop:trajectories_are_rw} is a sort of ``miracle'' of the geometric distribution.
\end{remark}

\begin{proof}[Proof of \cref{prop:trajectories_are_rw}]
	Let us fix $k\geq t$. By construction, $(\overline{\bm Z}^{(t)}_s)_{t\leq s\leq k}$ is a measurable functional of $(\overline{\bm W}_s)_{t\leq s\leq k}$. As a result, $\overline{\bm W}_{k+1}-\overline{\bm W}_k = (\overline{\bm X}_{k+1}-\overline{\bm X}_k, \overline{\bm Y}_{k+1}-\overline{\bm Y}_k)$ is independent of $(\overline{\bm Z}^{(t)}_s)_{t\leq s\leq k}$. We can rewrite $\overline{\bm W}_{k+1}-\overline{\bm W}_k = \bm B\times(1,-1) + (1-\bm B)\times(-\bm U,\bm V)$ where $\bm B$ is a Bernoulli random variable of parameter $1/2$, $\bm U,\bm V$ are geometric of parameter $1/2$, i.e.\ $\P(\bm U=\ell)=2^{-\ell-1}$, for all $\ell\geq 0$, altogether independent, and independent of $(\overline{\bm Z}^{(t)}_s)_{t\leq s\leq k}$.
	
	Now from the definition of $\overline{\bm Z}$ (\cref{eq:distrib_incr_coal}), we get
	\begin{equation*}
	\overline{\bm Z}^{(t)}_{k+1} -\overline{\bm Z}^{(t)}_{k}= - \bm B + (1-\bm B)\left[ \idf_{\overline{\bm Z}^{(t)}_{k}\geq 0} \bm V+ \idf_{\overline{\bm Z}^{(t)}_{k}< 0} \left(\bm U \idf_{\bm U < -\overline{\bm Z}^{(t)}_{k}} + (\bm V-\overline{\bm Z}^{(t)}_{k})\idf_{\bm U \geq -\overline{\bm Z}^{(t)}_{k}} \right)\right].
	\end{equation*}
	So the law of $\left(\overline{\bm Z}^{(t)}_{k+1} -\overline{\bm Z}^{(t)}_{k} \Big| (\overline{\bm Z}^{(t)}_s)_{t\leq s\leq k}\right)$ is equal to
	\begin{equation*}
	\begin{cases} 
	\mathcal L(-\bm B + (1-\bm B) \bm V), &\text{ if}\quad \overline{\bm Z}^{(t)}_{k}\geq 0,\\
	\mathcal L\left.\Big(-\bm B + (1-\bm B)\cdot\left(\bm U \idf_{\bm U < q} + (\bm V + q)\idf_{\bm U \geq q} \right)\Big)\right|_{q = -\overline{\bm Z}^{(t)}_{k}},&\text{ if} \quad\overline{\bm Z}^{(t)}_{k}< 0.
	\end{cases}
	\end{equation*}
	By the memoryless property of the geometric distribution, $\bm U \idf_{\bm U < q} + (\bm V + q)\idf_{\bm U \geq q}$ is distributed like $\bm V$ regardless of $q$. As a result, since $\overline{\bm Z}^{(t)}_{k+1} -\overline{\bm Z}^{(t)}_{k}$ is independent of $(\overline{\bm Z}^{(t)}_s)_{s\leq k}$, we have that
	$$\mathcal L\left(\overline{\bm Z}^{(t)}_{k+1} -\overline{\bm Z}^{(t)}_{k} \big| (\overline{\bm Z}^{(t)}_s)_{s\leq k}\right) = \mathcal L(-\bm B+\bm (1-\bm B)\bm V) = \mathcal L(\overline{\bm Y}_{k+1}-\overline{\bm Y}_k)= \mathcal L(-\overline{\bm X}_{k+1}+\overline{\bm X}_k).$$ 
	This implies the result in the statement of the proposition.	
\end{proof}

Let us now consider $\overline {\bm m} = \Inverse(\overline {\bm W})$. Recall that it is an infinite quasi-map, i.e.\ a countable union of finite polygons, glued along edges. 

\begin{proposition}\label{prop:inf_map_property}
	Almost surely, $\overline{\bm m}$ is an infinite map of the plane. In particular it is locally finite.
\end{proposition}

\begin{proof} 
	We show that a.s.\ every vertex of $\overline{\bm m}$ has finite degree, i.e.\ that $\overline{\bm m}$ is a.s.\ locally finite. 
	Let $i\in \Z$ be the index of some edge $e\in \overline{\bm m}$ and denote by $v$ its top vertex. Let $\bm j\geq i$ be the smallest index such that $\overline{\bm Z}^{(i)}_{\bm j} = 0$, $\overline{\bm Z}^{(i)}_{\bm j+1}<0$, and $\bm s\geq \bm j+1$ be the smallest index such that $\overline{\bm Z}^{(i)}_{\bm s}\geq 0$. The indices $\bm j$ and $\bm s$ exist a.s.\ thanks to \cref{prop:trajectories_are_rw}. These conditions, together with \cref{coroll:vertex degree}, imply that the number of outgoing edges of $v$ is bounded by $\bm s- \bm j-1$. 
	Repeating the same argument with the map ${\overline{\bm m}}^{**} \stackrel d= \overline{\bm m}$ (recall that $\overline{\bm m}^{**}$ denotes $\overline{\bm m}$ with the orientation of the edges reversed) we can prove the same result for the number of ingoing edges of $v$, proving that $v$ has a.s. finite degree.
	Since this argument can be done for all vertices of $\overline{\bm m}$ we can conclude that $\overline{\bm m}$ is a.s.\ locally finite.
	
	Let us now show that a.s.\ every edge of $\overline{\bm m}$ is adjacent to two faces, i.e.\ that $\overline{\bm m}$ is a.s.\ boundaryless. Let $i\in \Z$ be the index of some edge $e\in \overline{\bm m}$, and $\bm j\geq i$ be the smallest index after $i$ where $\overline{\bm Z}^{(i)}_{\bm j} = 0$, which exists a.s.\ thanks to \cref{prop:trajectories_are_rw}. Considering the finite map $\Inverse(\overline{\bm W}|_{[i,\bm j]})$, we see using \cref{prop:fortree} that $e$ has a.s.\ a face at its right and so this happens also in $\overline{\bm m}$. By countable intersection, this is a.s.\ true for every edge of $\overline{\bm m}$. The same property is also a.s.\ true at the left because $(\overline{\bm m})^{**}$ has the same distribution as $\overline{\bm m}$. So $\overline {\bm m}$ is a.s.\ boundaryless.

	The infinite map $\overline{\bm m}$ being a.s.\ simply connected is immediate: by definition of $\overline{\bm m}$, every finite submap $\bm f\subset \overline{\bm m}$ is included in one of the finite planar maps $\bm f_n=\Inverse(\overline{\bm W}|_{J_n})$.
\end{proof}

\subsection{Local topologies}\label{sect:local topologies}
We define a (finite or infinite) \emph{rooted walk} (resp.\ \emph{rooted coalescent-walk process}, \emph{rooted permutation}) as a walk (resp.\ coalescent-walk process, permutation) on a (finite or infinite) interval of $\Z$ containing $0$. More formally, we define the following sets (with the corresponding notions of size):
\begin{align*}
\widetilde \Walks_\bullet &\coloneqq \bigsqcup_{I \ni 0} \Walks(I),&&\text{where}\quad |W| \coloneqq |I|\text{ if }W\in \Walks(I),\\
\widetilde \Coals_\bullet &\coloneqq \bigsqcup_{I \ni 0} \Coals(I),&&\text{where}\quad |Z| \coloneqq |I|\text{ if }Z\in \Coals(I),\\
\widetilde \Perms_\bullet &\coloneqq \bigsqcup_{I \ni 0} \Perms(I),&&\text{where}\quad |\sigma| \coloneqq |I|\text{ if }\sigma\in \Perms(I).
\end{align*}
Of course, $0$ has to be understood as the root of any object in one of these classes. For $n\in \Z_{>0} \cup \{\infty\}$, $\Walks_\bullet^n$ is the subset of objects in $\widetilde \Walks_\bullet$ of size $n$, i.e.\ $\Walks_\bullet^n=\bigcup_{I \ni 0, |I|=n} \Walks(I)$, and $\Walks_\bullet$ denotes the set of finite-size objects. We also define analogs for $\widetilde \Coals_\bullet, \widetilde \Perms_\bullet$.

We justify the terminology. A rooted object of size $n$ can be also understood as an unrooted object of size $n$ together with an index in $[n]$ which identifies the root through the following identifications\footnote{Note that the natural identification for walks would be $(W,i) \longmapsto (W_{i+t} - W_i)_{t\in [-i+1, n-i]}$, but since we are considering walks up to an additive constant then the identification $(W,i) \longmapsto (W_{i+t})_{t\in [-i+1, n-i]}$ is equivalent.}: 
\begin{align*}
\Walks^n \times [n] \longrightarrow \Walks^n_\bullet,
&\quad(W,i) \longmapsto (W_{i+t})_{t\in [-i+1, n-i]}\in \Walks([-i+1, n-i]),\\
\Coals^n \times [n] \longrightarrow \Coals^n_\bullet,
 &\quad((Z^{(s)}_t)_{s,t\in [n]},i) \longmapsto (Z^{(i+s)}_{i+t})_{s,t\in [-i+1, n-i]}\in \Coals([-i+1, n-i]),\\
\Perms^n \times [n] \longrightarrow \Perms^n_\bullet,
 &\quad(\sigma,i) \longmapsto \leqsi\in \Perms([-i+1, n-i]),\quad 
\text {where}\quad 	\ell\leqsi j\iff \sigma_{\ell+i}\leq\sigma_{j+i}.
\end{align*}

We may now define \emph{restriction functions}: for $h\geq 1$, $I$ an interval of $\Z$ containing 0, and $\square\in  \Walks(I)$, $\Coals(I)$ or $\Perms(I)$, we define
\[r_h(\square) =\square |_{I \cap [-h,h]}.\]
So, for all $h\geq 1$, $r_h$ is a well defined function $\widetilde \Walks_\bullet \to \Walks_\bullet$, $\widetilde \Coals_\bullet \to \Coals_\bullet$, and $\widetilde \Perms_\bullet \to \Perms_\bullet$. Finally local distances on either $\widetilde \Walks_\bullet$, $\widetilde \Coals_\bullet$ or $\widetilde \Perms_\bullet$ are defined by a common formula:
\begin{equation}\label{eq:local_distance}
d\big(\square_1,\square_2\big)=2^{-\sup\big\{h\geq 1\;:\;r_h(\square_1)=r_h(\square_2)\big\}},
\end{equation}
with the conventions that $\sup\emptyset=0,$ $\sup\Z_{>0}=+\infty$ and $2^{-\infty}=0.$

The metric space $(\widetilde \Perms_{\bullet},d)$ is a compact space (see \cite[Theorem 2.16]{borga2018local}) and so complete and separable, i.e.\ it is a Polish space. On the other hand, $(\widetilde \Walks_{\bullet},d)$ and $(\widetilde \Coals_{\bullet},d)$ are Polish, but not compact (see for instance \cite[Section 1.2.1]{curienrandom}).

For the case of planar maps, the theory is slightly different. For a finite interval $I$, we denote by $\Maps(I)$ the set of planar oriented maps with edges labeled by the integers in the interval $I$. We point out that we do not put any restriction on the possible choices for the labeling.
We also define the set of finite rooted maps $\Maps_\bullet = \bigsqcup_{I\ni 0,|I|<\infty} \Maps(I)$, where the edge labeled by $0$ is called root.

A finite rooted map is obtained by rooting (i.e.\ distinguishing an edge) a finite unrooted map, by the identification $\Maps([n]) \times [n] \to \Maps_\bullet^n$ which sends $(m,i)$ to the map obtained from $m$ by shifting all labels by $-i$. The distance is defined similarly as in \cref{eq:local_distance}.
Let $B_h(m)$ be the ball of radius $h$ in $m$, that is the submap of $m$ consisting of the faces of $m$ that contain a vertex at distance less than $h$ from the tail of the root-edge. We set $d(m,m') =2^{-\sup\big\{h\geq 1\;:\;B_h(m) = B_h(m')\big\}},$ with the same conventions as before. We do not describe the set of possible limits, but simply take $\widetilde \Maps_{\bullet}$ to be the completion of the metric space $(\Maps_\bullet, d)$. In particular, it is easily seen that $\widetilde \Maps_{\bullet}$ contains all the infinite rooted planar maps.

\begin{proposition} 
	\label{prop:continuty_of_cpbp}
	The mappings $\wcp: \widetilde \Walks_\bullet \to	\widetilde \Coals_\bullet$ and $\cpbp: \widetilde \Coals_\bullet \to 
	\widetilde \Perms_\bullet$
	are $1$-Lipschitz.
\end{proposition}
\begin{proof}
	This is immediate since by definition, $r_h(\wcp(W)) = \wcp(r_h(W))$ and $r_h(\cpbp(Z)) = \cpbp(r_h(Z))$.
\end{proof}
\begin{proposition} 
	\label{prop:continuty_of_Inverse}
	The mapping $\Inverse: \widetilde \Walks_\bullet \to \widetilde \Maps_\bullet$ is almost surely continuous at $\overline{\bm W}$.
\end{proposition}

\begin{proof}
	Assume we have a realization $\overline{W}$ of $\overline{\bm W}$ such that $\overline{ m} = \Inverse(\overline{ W})$ is a rooted infinite oriented planar map, in particular is locally finite (this holds for almost all realizations thanks to \cref{prop:inf_map_property}). Let $h > 0$. The ball $B_h(\overline{m})$ is a finite subset of $\overline{m}$. Since $\overline{m} = \bigcup_n \Inverse(\overline{W}|_{[-n,n]})$, there must be $n$ such that $\Inverse(\overline{W}|_{[-n,n]}) \supset B_h(\overline{m})$. As a result, for all $W'\in\widetilde \Walks_\bullet$ such that $d(W',\overline{W})<2^{-n}$, then $B_h(\Inverse(W'))= B_h(\Inverse(\overline{ W}))$. This shows a.s.\ continuity.
\end{proof}

\subsection{Proofs of the local limit results}

We turn to the proof of \cref{thm:local}. We will prove local convergence for walks and then transfer to the other objects by continuity of the mappings $\wcp,\cpbp,\Inverse$. Let us recall that thanks to \cref{prop:unif_law}, $\bm W_n$ is distributed like $\overline{\bm W}|_{[n]}$ under a suitable conditioning. This will allow us to prove the following technical lemmas.

\begin{lemma}\label{lem:quenched_conv_walks1}
	Fix $h\in\Z_{>0}$ and $W \in \Walks([-h,h])\subset \Walks_\bullet^{2h+1}$.
	Fix $0<\varepsilon<1$. Then, uniformly for all $i$ such that $\lfloor n\varepsilon \rfloor +h < i< \lfloor (1-\varepsilon)n \rfloor-h$,
	\begin{equation*}\label{eq:statement_uenched_conv_walks1}
	\mathbb{P}\left(
	r_{h}({\bm W}_n,i)=W
	\right)
	\to\mathbb{P}\left(
	r_{h}(\overline{\bm W})=W \right).
	\end{equation*}
\end{lemma}

\begin{lemma}\label{lem:quenched_conv_walks2}
	Fix $h\in\Z_{>0}$ and $W \in \Walks([-h,h])\subset \Walks_\bullet^{2h+1}$.
	Fix $0<\varepsilon<1$. Then, uniformly for all $i,j$ such that $\lfloor n\varepsilon \rfloor +h < i,j< \lfloor (1-\varepsilon)n \rfloor-h$ and $|i-j|>2h$,
	\begin{equation}\label{eq:statement_uenched_conv_walks2}
	\mathbb{P}\left(
	r_{h}({\bm W}_n,i)=r_{h}({\bm W}_n,j)=W
	\right)
	\to\mathbb{P}\left(
	r_{h}(\overline{\bm W})=W \right)^2.
	\end{equation}
\end{lemma}

We just prove the second lemma, the proof of the first one is similar and simpler. Before doing that, we do the following observation, useful for the proof of \cref{lem:quenched_conv_walks2}. 
In what follows, if $W = (X,Y)$ is a two-dimensional walk, then $\inf W = (\inf X, \inf Y)$.

\begin{observation}\label{obs:infofwalks}
	Let $x=(x_i)_{i\in[0,n]}=(\sum_{j=1}^i y_j)_{i\in[n]}$ be a one-dimensional deterministic walk starting at zero of size $n$, i.e.\ $y_j\in\Z$ for all $j\in [n]$. Let $h\ll n$ and consider a second deterministic one-dimensional walk $x'=(x'_i)_{i\in[0,h]}=(\sum_{j=1}^i y'_j)_{i\in[0,h]}$. Fix also $k,\ell$ such that $0\leq k<\ell\leq n$ and consider the walk $x''=(x''_i)_{i\in[0,n+2h]}$ obtained by inserting two copies of the walk $x'$ in the walk $x$ at time $k$ and $\ell$. That is, for all $i\in[0,n+2h]$,
	\begin{equation*}
		x''_i= \sum_{j=1}^k y_j\cdot\mathds{1}_{j\leq i}+\sum_{j=1}^h y'_j\cdot\mathds{1}_{j+k\leq i}+\sum_{j=k+1}^\ell y_j\cdot\mathds{1}_{j+h\leq i}+\sum_{j=1}^h y'_j\cdot\mathds{1}_{j+\ell+h\leq i}+\sum_{j=\ell+1}^n y_j\cdot\mathds{1}_{j+2h\leq i}.
	\end{equation*}
	Then 
	$$\inf_{i\in[0,n+2h]}\{x''_i\}=\inf_{i\in[0,n]}\{x_i\}+\Delta,$$
	where $\Delta=\Delta(x,k,\ell,x')\in\R^2$ and it is bounded by twice the total variation of $x'$.
\end{observation}

\begin{proof}[Proof of \cref{lem:quenched_conv_walks2}]
	Set $E\coloneqq\left\{r_{h}(\overline{\bm W}|_{[n]},i)=r_{h}(\overline{\bm W}|_{[n]},j)=W\right\}$. By \cref{prop:unif_law}, the left-hand side of \cref{eq:statement_uenched_conv_walks2} can be rewritten as $\mathbb{P}\left(E \middle|(\overline{\bm W}_t)_{0\leq t\leq {n+1}}  \in \Walks_{A,\exc}^{n+2} \right)$. Using \cref{lem:AbsCont}, we have that
	\begin{multline}\label{eq:first_rewriting}
	\mathbb{P}\left(E \middle|(\overline{\bm W}_t)_{0\leq t\leq {n+1}}  \in \Walks_{A,\exc}^{n+2} \right)
	=\E\left[\mathds{1}_{\widetilde{E}}\cdot 
	\alpha_{n+2,\lfloor n\varepsilon \rfloor}^{0,0}\left(-\inf_{0\leq k \leq n+2-2\lfloor n\varepsilon \rfloor} \overline{\bm W}_k\;,\;\overline{\bm W}_{n+2-2\lfloor n\varepsilon \rfloor}\right)
	\right],
	\end{multline} 
	where  $\alpha_{n+2,\lfloor n\varepsilon \rfloor}^{0,0}(a,b)$ is a function defined in \cref{eq:alpha_tilting_function} page \pageref{eq:alpha_tilting_function} and 
	$$\widetilde{E}\coloneqq\left\{r_{h}(\overline{\bm W}|_{[n]},i-\lfloor n\varepsilon \rfloor)=r_{h}(\overline{\bm W}|_{[n]},j-\lfloor n\varepsilon \rfloor)=W\right\}.$$
	From \cref{obs:infofwalks}, conditioning on $\widetilde{E}$, we have that
	\begin{multline}\label{eq:equality_walks}
	\left(-\inf_{0\leq k \leq n+2-2\lfloor n\varepsilon \rfloor} \overline{\bm W}_k\;,\;
	\overline{\bm W}_{n+2-2\lfloor n\varepsilon \rfloor}
	\right)\\
	=\left(-\inf_{0\leq k \leq n+2-2\lfloor n\varepsilon \rfloor-2(2h+1)} \overline{\bm S}_k-  \bm\Delta\;,\;\overline{\bm S}_{n+2-2\lfloor n\varepsilon \rfloor-2(2h+1)}+2\delta\right),
	\end{multline}
	where $\overline{\bm S}=(\overline{\bm S}_t)_{t\in \Z}$ is the walk obtained from $(\overline{\bm W}_t)_{t\in \Z}$ removing the $2h+1$ steps around $i-\lfloor n\varepsilon \rfloor$ and $j-\lfloor n\varepsilon \rfloor$, $\delta = W_h - W_{-h}$ and $\bm \Delta$ is a deterministic function of $(\overline{\bm S}_t)_{t\in \Z}$, $i$, $j$ and $W$, bounded by twice the total variation of $W$.
	%In particular, $\bm \Delta$ is also independent of $(\overline{\bm W}_t)_{t\in \Z}$.
	Using the relation in \cref{eq:equality_walks} we can rewrite the right-hand side of \cref{eq:first_rewriting} as
	\begin{equation*}
	\mathbb{P}(\widetilde{E})
	\cdot
	\E\left[\alpha_{n+2,\lfloor n\varepsilon \rfloor}^{0,0}\left(-\inf_{0\leq k \leq n-2\lfloor n\varepsilon \rfloor-2(2h+1)} \overline{\bm S}_k- \bm \Delta\;,\;\overline{\bm S}_{n-2\lfloor n\varepsilon \rfloor-2(2h+1)}+2\delta\right)
	\right],
	\end{equation*}
	where we used the independence between $\widetilde{E}$ and the right-hand side of \cref{eq:equality_walks}. Note that $\P(\widetilde{E})=
	\mathbb{P}\left(r_{h}(\overline{\bm W})=W \right)^2$ since  $|i-j|>2h$ by assumption.
	We now show that 
	the second factor of the equation above converges to 1 uniformly for all $i$,$j$.
	Set for simplicity of notation 
	\begin{equation*}
	f(\overline{\bm S})=\left(-\inf_{0\leq k \leq n-2\lfloor n\varepsilon \rfloor-2(2h+1)} \overline{\bm S}_k - \bm \Delta\;,\;\overline{\bm S}_{n-2\lfloor n\varepsilon \rfloor-2(2h+1)}+2\delta\right).
	\end{equation*}
	By \cref{lem:LLT} we have
	$$\sup_{i,j}\Bigg|\E\left[\alpha_{n+2,\lfloor n\varepsilon \rfloor}^{0,0}\left(f(\overline{\bm S})\right)\right]-\E\left[\alpha_\eps\left(g(\overline{\bm S})\right)\right]\Bigg|\to0,$$
	where $\alpha_\eps(\cdot)$ is defined in \cref{eq:fuction_alpha}, and
	$$g(\overline{\bm S})=\left(-\frac 1{\sqrt { n}}\left(\inf_{0\leq k \leq n-2\lfloor n\varepsilon \rfloor-2(2h+1)} \overline{\bm S}_k\right) - \frac {\bm\Delta}{\sqrt{ n}}\;,\;\frac 1{\sqrt { n}}\overline{\bm S}_{n-2\lfloor n\varepsilon \rfloor-2(2h+1)} + \frac {2\delta}{\sqrt{ n}}\right).$$
	Therefore, in order to conclude, it is enough to show that $\E\left[\alpha_\eps\left(g(\overline{\bm S})\right)\right]\to1$. 
	
	We have that $\E\left[\alpha_\eps\left(g(\overline{\bm S})\right)\right]\to\E\left[\alpha_\eps
 	\left({-\inf_{0\leq t \leq 1-2\eps}\conti W(t),\conti W(1-2\eps)}\right)\right]$,
where
$\conti W = (\conti X,\conti Y)$ is a standard two-dimensional Brownian motion of correlation $-1/2$. This follows from the fact that $\bm \Delta$ is bounded, that $\alpha_\eps$ is a continuous and bounded function (see \cref{lem:LLT}), and that 
$$\left(\frac{1}{\sqrt n}\overline{\bm S}_{\lfloor nt\rfloor}\right)_{t\in[0,1]}\stackrel{d}{\longrightarrow}\left(\bm{\conti W}_t\right)_{t\in[0,1]}.$$
 The latter claim is a consequence of Donsker's theorem and the basic computation $\Var(\nu) = \begin{psmallmatrix}2 &-1 \\ -1 &2 \end{psmallmatrix}$. In addition, we have that $\E\left[\alpha_\eps
 \left({-\inf_{0\leq t \leq 1-2\eps}\conti W(t),\conti W(1-2\eps)}\right)\right]=1$ by \cref{prop:brown_ex} (used with $h=1$), and so we can conclude the proof. 
\end{proof}

We can now prove the main result of this section, i.e.\ the quenched local limit result presented in the introduction.

\begin{proof}[Proof of \cref{thm:local}]
	We start by proving that
	\begin{equation}\label{eq:firststepintheproof}
	\mathcal{L}\Big((\bm W_n, \bm i_n)\Big|\bm W_n\Big)\stackrel{P}{\longrightarrow}\mathcal{L}\left(\overline{\bm W}\right).
	\end{equation}
	For that it is enough (see for instance \cite[Corollary 2.38]{borga2018local} for a similar argument in the case of permutations) to show that, for any $h\geq 1$ and fixed finite rooted walk $W\in \Walks([-h,h]) \subset \Walks_\bullet$,
	\begin{equation}\label{eq:goaloftheproof}
	\mathbb{P}\left(r_{h}(\bm W_n,\bm i_n)=W
	\mid
	\bm W_n\right)
	\stackrel{P}{\longrightarrow}\mathbb{P}(r_{h}(\overline{\bm W})=W).
	\end{equation}
	Note that 
	\begin{align*}
	\mathbb{P}\left(r_{h}(\bm W_n,\bm i_n)=W
	\mid
	\bm W_n\right)
	&=\frac{\#\left\{j\in[n] : 
		r_h(\bm W_n,j)=W)\right\}}{n}
	=\frac 1 n \sum_{j\in [n]}\mathds{1}_{\left\{r_{h}(\bm W_n,j)=W)\right\}}.
	\end{align*}
	We use the second moment method to prove that this sum converges in probability to $\mathbb{P}(r_{h}(\overline{\bm W})=W)$.
	We first compute the first moment:
	\begin{equation*}
	\E\left[\frac 1 n \sum_{j\in [n]}\mathds{1}_{\left\{r_{h}({\bm W_n},j)=W)\right\}}\right]=\frac 1 n \sum_{j\in [n]}\mathbb{P}\left(r_{h}({\bm W_n},j)=W\right)
	\to\mathbb{P}\left(r_{h}(\overline{\bm W})=W\right),
	\end{equation*}
	where for the limit we used \cref{lem:quenched_conv_walks1}.
	We now compute the second moment:
	\begin{multline*}
	\E\left[\left(\frac 1 n \sum_{j\in [n]}\mathds{1}_{\left\{r_{h}(({\overline{\bm W}}_t)_{t\in [n]},j)=W)\right\}}\right)^2\;\right]\\
	=\frac{1}{n^2}\sum_{i,j\in [n], {|i-j|>2h}}\P\left(r_{h}({\bm W_n},i)=r_{h}({\bm W_n},j)=W\right) + O(1/n).
	\end{multline*}
	This converges to $\mathbb{P}(r_{h}(\overline{\bm W})=W)^2$ by \cref{lem:quenched_conv_walks2}.
	The computations of the first and second moment, together with Chebyshev's inequality lead to the proof of \cref{eq:goaloftheproof} and so to the quenched convergence of walks.
	
	Now to extend the result to other objects, we will use continuity of the mappings $\wcp,\cpbp,\Inverse$ (see \cref{prop:continuty_of_cpbp} and \cref{prop:continuty_of_Inverse}). 
	Using a combination of the results stated in \cite[Theorem 4.11, Lemma 4.12]{kallenberg2017random} we have that \cref{eq:firststepintheproof} implies the following convergence
	\begin{equation}\label{eq:fbvewuifviwefewoub}
	\mathcal{L}\Big(\big((\bm W_n, \bm i_n), (\wcp(\bm W_n), \bm i_n), (\cpbp\circ\wcp(\bm W_n), \bm i_n), (\Inverse(\bm W_n), \bm i_n)\big)\Big| \mathfrak{B}_n\Big)\stackrel{P}{\longrightarrow}\mathcal{L}\left(\overline{\bm W},\overline{\bm Z},\overline{\bm \sigma},\overline{\bm m}\right),
	\end{equation}
	and so \cref{eq:quenched_conv_walks} holds. The specific result that we used for \cref{eq:fbvewuifviwefewoub} is a generalization of the \emph{mapping theorem} for random measures: Let $(\bm{\mu}_n)_{n\in\Z_{>0}}$ be a sequence of random measures on a space $E$ that converges in distribution to a random measure $\bm{\mu}$ on $E.$ Let $F$ be a function from $E$ to a second space $H$ such that the set $D_F$ of discontinuity points of $F$ has measure $\bm{\mu}(D_F)=0$ a.s.. Then the sequence of pushforward random measures $(\bm{\mu}_n\circ F^{-1})_{n\in\Z_{>0}}$ converges in distribution to the pushforward random measure $\bm{\mu}\circ F^{-1}.$
\end{proof}

\section{Scaling limits of coalescent-walk processes}\label{sec:coalescent}

In this section we deal with scaling limits of coalescent-walk processes both in the finite and infinite-volume case. The main result in this section is \cref{thm:discret_coal_conv_to_continuous}. This will be the main ingredient in the proofs of the two main theorems in \cref{sec:final}, namely \cref{thm:permuton} and \cref{thm:joint_scaling_limits}. Nevertheless, we believe that our intermediate results, \cref{thm:coal_con_uncond} and \cref{prop:coal_con_cond}, are of independent interest.

All the spaces of continuous functions considered below are implicitly endowed with the topology of uniform convergence on every compact set.

\subsection{The continuous coalescent-walk process}

We start by defining the candidate continuous limiting object: it is formed by the solutions of the following family of stochastic differential equations (SDEs) indexed by $u\in \R$ and driven by a two-dimensional process $\conti W = (\conti X,\conti Y)$:
\begin{equation}\label{eq:flow_SDE}
\begin{cases}
d\conti Z^{(u)}(t) = \idf_{\{\conti Z^{(u)}(t)> 0\}} d\conti Y(t) - \idf_{\{\conti Z^{(u)}(t)\leq 0\}} d \conti X(t),& t\geq u,\\
\conti Z^{(u)}(t)=0,&  t\leq u.
\end{cases} 
\end{equation}
Existence and uniqueness of a solution of the SDE above (for $u\in \R$ fixed) were already studied in the literature in the case where the driving process $\conti W$ is a Brownian motion, in particular with the following result. 

We recall that a \emph{standard two-dimensional Brownian motion of correlation $\rho$} is a continuous two-dimensional Gaussian process such that the components $\conti X$ and $\conti Y $ are standard one-dimensional Brownian motions, and $\mathrm{Cov}(\conti X(t),\conti Y(s)) = \rho \cdot (t\wedge s)$.
\begin{theorem}[Theorem 2 of \cite{MR3098074}, Proposition 2.2 of \cite{MR3882190}]\label{thm:ext_and_uni}
	\label{thm:uniqueness}
	Fix $\rho \in (-1,1)$. Let $T\in (0,\infty]$ and let $\conti W = (\conti X,\conti Y)$ be a standard two-dimensional Brownian motion of correlation $\rho$ and time-interval $[0,T)$.
	We have pathwise uniqueness and existence of a strong solution for the SDE:
	\begin{equation} \label{eq:SDE}
	\begin{cases}
	d \conti Z(t) &= \idf_{\{\conti Z(t)> 0\}} d \conti Y(t) - \idf_{\{\conti Z(t)\leq 0\}} d \conti X(t), \quad  0\leq t < T,\\
	\conti Z(0)&=0.
	\end{cases} 
	\end{equation}
	Namely, letting $(\Omega, \mathcal F, (\mathcal F_t)_{0\leq t < T}, \Prob)$ be a filtered probability space satisfying the usual conditions, and assuming that $\conti W$ is an $(\mathcal F_t)_t$-Brownian motion, 
	\begin{enumerate}
		\item if $\conti Z,\widetilde{\conti Z}$ are two $(\mathcal F_t)_t$-adapted continuous processes that satisfy \cref{eq:SDE} almost surely, then $\conti Z=\widetilde{\conti Z}$ almost surely.
		\item There exists an  $(\mathcal F_t)_t$-adapted continuous process $\conti Z$ which satisfies \cref{eq:SDE} almost surely.
	\end{enumerate}
	In particular, there exists for every $t\in (0,T]$ a measurable mapping $\solution_t : \mathcal C([0,t)) \to \mathcal C([0,t))$, called the \emph{solution mapping}, such that 
	\begin{enumerate}[resume]
		\item setting $\conti Z = \solution_t(\conti W|_{[0,t)})$, then $\conti Z$ satisfies \cref{eq:SDE} almost surely on the interval $[0,t)$.
		\item For every $0\leq s \leq t \leq T$, then $F_t(\conti W|_{[0,t)})|_{[0,s)} = F_s (\conti W|_{[0,s)})$ almost surely.
\end{enumerate} 
\end{theorem}

Assume from now on that $\conti W = (\conti X, \conti Y)$ is a standard two-dimensional Brownian motion of correlation $-1/2$ and time-interval $\R$. Let $(\mathcal F_t)_{t\in \R}$ be the canonical filtration of $\conti W$. For every $u\in \R$ we define
\begin{equation}\label{eq:frhveiwervfoew}
\conti Z^{(u)}(t) = \begin{cases}
F_\infty\Big((\conti W(u + s) - \conti W(u))_{0\leq s <\infty}\Big)(t-u),& t\geq u,\\
0, &t<u.
\end{cases}
\end{equation}

It is clear that $\conti Z^{(u)}$ is $(\mathcal F_t)_t$-adapted. Because \cref{eq:flow_SDE} is invariant by addition of a constant to $\conti W$, and because $\conti W(u + s) - \conti W(u)$ is a Brownian motion with time-interval $\R_+$, we see that for every fixed $u\in\R$, $\conti Z^{(u)}$ satisfies \cref{eq:flow_SDE} almost surely.

\cref{eq:frhveiwervfoew} makes the mapping $(\omega,u)\mapsto \conti Z^{(u)}$ jointly measurable -- joint measurability follows by composition because $F_\infty$ is measurable and $(\omega, u)\mapsto (\conti W(u + s) - \conti W(u))_{0\leq s <\infty}$ is also measurable. Hence by Tonelli's theorem, for almost every $\omega$, $\conti Z^{(u)}$ is a solution for almost every $u$. 

\begin{remark}\label{rem:no_flow}
	Given $\omega$ (even restricted to a set of probability one), we cannot say that $\{\conti Z^{(u)}\}_{u\in \R}$ forms a whole field of solutions to \cref{eq:flow_SDE} driven by $\conti W$. Indeed, we cannot guarantee that the SDE holds for all $u\in\R$ simultaneously. In fact, we expect thanks to intuition coming from Liouville Quantum Gravity, that there exist exceptional times $u\in\R$ where uniqueness fails, with two or three distinct solutions. This phenomenon is also observed in another coalescing flow of an SDE \cite{MR2094439} and in the Brownian web \cite{MR3644280}.
\end{remark}

\begin{remark}\label{rem:solution_of_SDE_are_BM}
	By Lévy's characterization theorem  \cite[Theorem IV.3.6]{revuz2013continuous}, for every fixed $u\in\mathbb R$, the process $\conti Z^{(u)}$ is a standard one-dimensional Brownian motion on $[u,\infty)$ with $\conti Z^{(u)}(u) = 0$. Note however that the coupling of $\conti Z^{(u)}$ for different values of $u\in\R$ is highly nontrivial.
\end{remark}

\begin{definition}
We call \emph{continuous coalescent-walk process} (driven by $\conti W$) the collection of stochastic processes $\left\{\conti Z^{(u)}\right\}_{u\in \R}$. 
\end{definition}

Since for all $u\in\R$, $(\conti Z^{(u)}(t))_{t\geq u}$ is a Brownian motion, one can define almost surely its local time process at zero  $\conti L^{(u)}$ (see \cite[Chapter VI]{revuz2013continuous}): namely for $t\geq u$, $\conti L^{(u)}(t)$ is the limit in probability of
\[\frac 1 {2\eps} \Leb\left(\left\{s\in [u,t]: |\conti Z^{(u)}(s)| <\eps \right\}\right).\]
By convention we set $\conti L^{(u)}(t) = 0$ for $t<u$ so that  $\conti L^{(u)}$ is a continuous process on $\R$.

In the next section we show that the \emph{continuous coalescent-walk process} is the scaling limit of the discrete infinite-volume coalescent-walk processes defined in \cref{sect: bij_walk_coal}.

\subsection{The unconditioned scaling limit result}

Let $\overline{\bm W} = (\overline{\bm X},\overline{\bm Y}) =(\overline{\bm X}_k,\overline{\bm Y}_k)_{k\in \Z}$ be the two-dimensional random walk (with step distribution $\nu$) defined below \cref{eq:step_distribution_walk} page~\pageref{eq:step_distribution_walk}, and let $\overline{\bm Z} = \wcp(\overline{\bm W})$ be the corresponding discrete coalescent-walk process. Let also $(\overline {\bm L}^{(i)}(j))_{-\infty<i\leq j <\infty} = L_{\overline{\bm Z}}$ be the local time process of $\overline{\bm Z}$ as defined in \cref{eq:local_time_process} page~\pageref{eq:local_time_process}. By convention, from now on we extend trajectories of $\overline{\bm Z}$ and $\overline {\bm L}$ to the whole $\Z$ by setting $\overline{\bm Z}^{(j)}(i) = \overline {\bm L}^{(j)}(i) = 0$ for $i,j \in \Z$, $i<j$.
We define rescaled versions: for all $n\geq 1, u\in \R$, let $\overline {\conti W}_n:\R\to \R^2$, $\overline {\conti Z}^{(u)}_n:\R\to\R$ and $\overline {\conti L}^{(u)}_n:\R\to\R$ be the continuous functions defined by linearly interpolating the following points:
\begin{equation}\label{eq:rescaled_version}
\overline{\conti W}_n\left(\frac kn\right) = \frac 1 {\sqrt {2n}} \overline{\bm W}_{k}, \quad
\overline{\conti Z}^{(u)}_{n}\left(\frac kn\right) = \frac 1 {\sqrt {2n}} \overline{\bm Z}^{(\lceil nu\rceil)}_{k}, \quad
\overline{\conti L}^{(u)}_{n}\left(\frac kn\right) = \frac 1 {\sqrt {2n}} \overline{\bm L}^{(\lceil nu\rceil)}(k),\quad k\in \Z.
\end{equation}
Our first scaling limit result for infinite-volume coalescent-walk processes is the following result that deals with a single trajectory.
\begin{theorem}\label{thm:coal_con_uncond}
	Fix $u\in \R$. 
	We have the following joint convergence in $\mathcal C(\R,\R)^{4}$: 
	\begin{equation}
	\label{eq:coal_con_uncond}
	\left(\overline{\conti W}_n,\overline{\conti Z}^{( u)}_n,\overline{\conti L}^{( u)}_n\right) 
	\xrightarrow[n\to\infty]{d}
	\left(\conti W,\conti Z^{(u)},\conti L^{(u)}\right).
	\end{equation}
\end{theorem}

\begin{proof} The proof is spitted in two main parts.

\paragraph*{Convergence of $\overline{\conti W}_n$ and $(\overline{\conti Z}^{( u)}_n,\overline{\conti L}^{( u)}_n)$.}
The first step in the proof is to establish component-wise convergence in 
\cref{eq:coal_con_uncond}.
By Donsker's theorem, upon computing $\Var(\nu) = \begin{psmallmatrix}2 &-1 \\ -1 &2 \end{psmallmatrix}$,
we get that the rescaled random walk $\overline{\conti W}_n=(\overline{\conti X}_n,\overline{\conti Y}_n)$ converges to $\conti W=(\conti X, \conti Y)$ in distribution. Using \cref{prop:trajectories_are_rw}, we know that a single trajectory of the discrete coalescent-walk process has the distribution of a random walk, and applying again an invariance principle such as \cite[Theorem 1.1]{Borodin}, we get that $(\overline{\conti Z}^{(u)}_n({u+t}), \overline{\conti L}^{(u)}_n({u+t}))_{t\geq 0}$ converges to a one-dimensional Brownian motion and its local time, which is indeed distributed like $(\conti Z^{(u)},\conti L^{(u)})$ thanks to \cref{rem:solution_of_SDE_are_BM}.

\paragraph{Joint convergence.}
The second step in the proof is to establish joint convergence. By Prokhorov's theorem, both $\overline{\conti W}_n$ and $(\overline{\conti Z}^{( u)}_n,\overline{\conti L}^{( u)}_n)$ are tight sequences of random variables. Since the product of two compact sets is compact, then the left-hand side of \cref{eq:coal_con_uncond} forms a tight sequence. Therefore, again by Prokhorov's theorem, it is enough to identify the distribution of all joint subsequential limits in order to show the convergence in \cref{eq:coal_con_uncond}. Assume that along a subsequence, we have 
\begin{equation}\label{eq:joint_conve_to_prove}
\left(\overline{\conti W}_n,\overline{\conti Z}^{( u)}_n,\overline{\conti L}^{( u)}_n\right) 
\xrightarrow[n\to\infty]{d}
\left(\conti W,\widetilde {\conti Z},\widetilde{\conti L}\right),
\end{equation}
where $\widetilde {\conti Z}$ is a Brownian motion started at time $u$, and $\widetilde {\conti L}$ is its local time process at zero. The joint distribution of the right-hand side is unknown for now, but we will show that $\widetilde {\conti Z} = \conti Z^{(u)}$ almost surely, which will complete the proof. 
Using Skorokhod's theorem, we may define all involved variables on the same probability space and assume that the convergence in \cref{eq:joint_conve_to_prove} is in fact almost sure.

Let $(\mathcal G_t)_t$ be the smallest complete filtration to which $\conti W$ and $\widetilde{\conti Z}$ are adapted, i.e.\ $\mathcal G_t$ is the completion of $\sigma(\conti W(s),\widetilde{\conti Z}(s), s\leq t)$ by the negligible events of $\P$.
	We shall show that $\conti W$ is a $(\mathcal G_t)_t$-Brownian motion, that is for $t\in \R, s\in \R_{\geq 0}$, 
	$$\left(\conti W(t+s)-\conti W(t)\right) \indep \mathcal G_t.$$ 
	For fixed $n$, by definition of random walk, $\left(\overline{\conti W}_n({t+s}) - \overline{\conti W}_n(t)\right) \indep \sigma\left(\overline{\bm W}_k, k\leq \lfloor n t \rfloor\right)$. Therefore, 
	\begin{equation}
	\left(\overline{\conti W}_n({t+s})-\overline{\conti W}_n(t)\right) \ \indep\ \left(\overline{\conti W}_n(r),\overline{\conti Z}^{(u)}_n(r)\right)_{r\leq n^{-1}\lfloor nt \rfloor}.
	\end{equation}
	By convergence, we obtain that $\left(\conti W(t+s)-\conti W(t)\right)$ is independent of $\left(\conti W(r),\widetilde{\conti Z}(r)\right)_{r\leq t}$, completing the claim that $\conti W$ is a $(\mathcal G_t)_t$-Brownian motion.

	Now fix a rational $\eps>0$ and an open interval with rational endpoints $(\bm a,\bm b)$ on which $\widetilde{\conti Z}(t)>\eps$. Note that the interval $(\bm a,\bm b)$ depends on $\widetilde{\conti Z}(t)$. By almost sure convergence, there is $N_0$ such that for $n\geq N_0$, $\overline{\conti Z}^{(u)}_n>\eps/2$ on $(\bm a,\bm b)$. 
	On the interval $(\bm a+1/n,\bm b)$, the process $\overline{\conti Z}^{(u)}_n - \overline{\conti Y}_n$ is constant  by construction of the coalescent-walk process (see \cref{sect: bij_walk_coal}). 
	Hence the limit $\widetilde{\conti Z} - \conti Y$ is constant too on $(\bm a,\bm b)$ almost surely. We have shown that almost surely $\widetilde{\conti Z} - \conti Y$ is locally constant on $\{t>u:\widetilde{\conti Z}(t)>\eps\}$. This translates into the following almost sure equality:
	\[\int_{u}^t \idf_{\{\widetilde{\conti Z}(r)>\eps\}} d\widetilde{\conti Z}(r) = \int_u^t \idf_{\{\widetilde{\conti Z}(r)>\eps\}} d\conti Y(r), \quad t\geq u. \]
	The stochastic integrals are well defined: on the left-hand side by considering the canonical filtration of $\widetilde {\conti Z}$, on the right-hand side by considering $(\mathcal G_t)_t$.	
	The same can be done for negative values, leading to 
	\[\int_u^t \idf_{\{|\widetilde{\conti Z}(r)|>\eps\}} d\widetilde{\conti Z}(r)  =  \int_{u}^t \idf_{\{\widetilde{\conti Z}(r)>\eps\}} d\conti Y(r) -  \int_{u}^t \idf_{\{\widetilde{\conti Z}(r)<-\eps\}} d\conti X(r). \]
	By stochastic dominated convergence theorem, one can take the limit as $\eps\to 0$, \cite[Theorem IV.2.12]{revuz2013continuous}, and obtain 
	\[\int_u^t \idf_{\{\widetilde{\conti Z}(r)\neq0\}}d\widetilde{\conti Z}(r)  =  \int_{u}^t \idf_{\{\widetilde{\conti Z}(r)>0\}} d\conti Y(r) -  \int_{u}^t \idf_{\{\widetilde{\conti Z}(r)<0\}} d\conti X(r). \]
	Thanks to the fact that $\widetilde{\conti Z}$ is a Brownian motion, $\int_u^t \idf_{\{\widetilde{\conti Z}(r)=0\}}d\widetilde{\conti Z}(r) = 0$, so that the left-hand side of the equation above equals $\widetilde {\conti Z}(t)$.
	As a result $\widetilde {\conti Z}$ satisfies \cref{eq:flow_SDE} almost surely and we can apply pathwise uniqueness (\cref{thm:uniqueness}, item 1) to complete the proof that $\widetilde {\conti Z} = \conti Z^{(u)}$ almost surely.
\end{proof}

\subsection{The conditioned scaling limit result}
\label{sec:cond_conv}

In the previous section we saw a scaling limit result for infinite-volume coalescent-walk processes. We now deal with the finite-volume case, the one that we need for our results.

For all $n\geq 1$, let $\bm W_n$ be a uniform element of the space of tandem walks $\mathcal W_n$ and $\bm Z_n = \wcp(\bm W_n)$ be the corresponding uniform coalescent-walk process. Let also $\bm L_n = ( {\bm L}^{(i)}_{n}(j))_{1\leq i\leq j \leq n} = L_{{\bm Z_n}}$ be the local time process of $\bm Z_n$ as defined in \cref{eq:local_time_process} page~\pageref{eq:local_time_process}. For all $n\geq 1, u\in (0,1)$, let ${\conti W}_n:[0,1]\to \R^2$, ${\conti Z}^{(u)}_n:[0,1]\to\R$ and ${\conti L}^{(u)}_n:[0,1]\to\R$ be the continuous functions defined by linearly interpolating the following points defined for all $k\in [n]$,
\begin{equation}
{\conti W}_n\left(\frac kn\right) = \frac 1 {\sqrt {2n}} \bm W_{n}(k), \quad
{\conti Z}^{(u)}_{n}\left(\frac kn\right) = \frac 1 {\sqrt {2n}} \bm Z^{(\lceil nu\rceil)}_{n}(k), \quad
{\conti L}^{(u)}_{n}\left(\frac kn\right) = \frac 1 {\sqrt {2n}} \bm L^{(\lceil nu\rceil)}_{n}(k).
\end{equation}Our goal is to obtain a scaling limit result for these processes in the fashion of \cref{thm:coal_con_uncond}.

Let $\conti W_e$ denote a two-dimensional Brownian excursion of correlation $-1/2$ in the non-negative quadrant and denote by $(\Omega, \mathcal F, (\mathcal F_t)_{0\leq t\leq 1}, \Prob_\exc)$ the completed canonical probability
space of $\conti W_e$. From now on we work in this space. The law of the process $\conti W_e$ is characterized (for instance) by \cref{prop:brown_ex} in \cref{sec:appendix}. 
Using \cref{prop:unif_law} and \cref{prop:DW}, we have that $\conti W_e$ is the scaling limit of $\conti W_n$. Then, the scaling limit of $\conti Z_n$ should be the continuous coalescent-walk process driven by $\conti W_e$, i.e.\ the collection (indexed by $u\in[0,1]$) of solutions of \cref{eq:flow_SDE} driven by $\conti W_e$. 

Let us remark that since Brownian excursions are semimartingales \cite[Exercise XII.4.9]{revuz2013continuous}, it makes sense to consider stochastic integrals against such processes, so the SDE in \cref{eq:flow_SDE} driven by $\conti W_e$ is well defined. We can also transport existence and uniqueness of strong solutions from \cref{thm:ext_and_uni} using absolute continuity arguments as follows.

Denote by $\mathcal F^{(u)}_t$ the sigma-algebra generated by $\conti W_e(s) - \conti W_e(u)$  for $ u\leq s \leq t$ and completed by negligible events of $\Prob_\exc$.
\begin{theorem}\label{thm:ext_and_uni_excursion}
	For every $u\in(0,1)$, there exists a continuous $\mathcal F^{(u)}_t$-adapted stochastic process $\conti Z_e^{(u)}$ on $[u,1)$, such that
	\begin{enumerate}
	\item  the mapping $(\omega,u)\mapsto \conti Z_e^{(u)}$ is jointly measurable. 
	\item For every $0<u<r<1$, $\conti Z_e^{(u)}$ satisfies  \cref{eq:flow_SDE} with driving motion $\conti W_e$, restricted to the interval $[u,r]$, almost surely.
	\item If $0<u<r<1$ and $\widetilde {\conti Z}$ is an $\mathcal F^{(u)}_t$-adapted stochastic process that satisfies \cref{eq:flow_SDE}  with driving motion $\conti W_e$ on interval $[u,r]$, then $\widetilde  {\conti Z} =  \conti Z_e^{(u)}$ on $[u,r]$ almost surely. 
	\end{enumerate}
\end{theorem}

\begin{proof}
	Recall the solution mappings $\solution_t$ defined in \cref{thm:ext_and_uni}. For $0<u<r<1$, we define the process $\conti R_{u,r}\in \mathcal C([u,r])$ as follows:
	 $$\conti R_{u,r}(t) \coloneqq \solution_{r-u}\Big((\conti W_e({u+s}) - \conti W_e({u}))_{0\leq s \leq r-u}\Big)(t-u),\quad u\leq t \leq r.$$
	  By definition, $\conti R_{u,r}$ is measurable with regards to $\mathcal F_r^{(u)}$. By \cref{prop:brown_ex}, the distribution of $((\conti W_e({u+s}) - \conti W_e({u}))_{0\leq s \leq r-u}$ is absolutely continuous with regards to that of a Brownian motion with time-interval $[0,r-u]$. Hence thanks to items 3 and 4 of \cref{thm:ext_and_uni}, 
	\begin{enumerate}
		\item $\conti R_{u,r}$ almost surely satisfies \cref{eq:flow_SDE} driven by $\conti W_e$ on interval $[u,r]$;
		\item for $0<u<r<r'<1$, we have $\conti R_{u,r} = (\conti R_{u,r'})|_{[u,r]}$ almost surely.
	\end{enumerate}
	Moreover, $(u,r,\omega)\mapsto \conti R_{u,r}$ is measurable by construction.
	
	Finally, this holds simultaneously for all rational $r,r'$ such that $0<u<r<r'<1$, so that there almost surely exists $\conti Z^{(u)} \in \mathcal C([u,1))$ whose restriction coincides with $\conti R_{u,r}$ for every rational $r$. Hence it almost surely satisfies \cref{eq:flow_SDE} driven by $\conti W_e$. 	
	For fixed $r_0\in (u,1)$, $\conti Z^{(u)}|_{[u,r_0]} = \conti R_{u,r_0}$, which is  $\mathcal F^{(u)}_{r_0}$-measurable. Hence $\conti Z^{(u)}$ is $\mathcal F^{(u)}$-adapted. This proves existence of a strong solution.
	
	We now move to the uniqueness claim. Consider two $\mathcal F^{(u)}$-adapted solutions $\conti Z^{(u)},\widetilde{\conti Z}^{(u)}$ of the SDE in \cref{eq:flow_SDE} driven by $\conti W_e$. By assumption $r<1$. 
	There must be $H,\widetilde  H : \mathcal C([u,r]) \to \mathcal C([u,r])$ so that almost surely, 
	$$\conti Z^{(u)} = H(\conti W_e({s}) - \conti W_e({u}), u\leq s \leq r)\quad\text{and}\quad\widetilde {\conti Z}^{(u)} = \widetilde H(\conti W_e({s}) - \conti W_e({u}), u\leq s \leq r).$$
	By absolute continuity (\cref{prop:brown_ex}), given a two-dimensional Brownian motion $\conti B$, the processes $H(\conti B)$ and $\widetilde H(\conti B)$ are solutions of the SDE in \cref{eq:flow_SDE} driven by $\conti B$. By pathwise uniqueness (\cref{thm:uniqueness}, item 2), $H(\conti B)=\widetilde H(\conti B)$ almost surely so that by absolute continuity, $\conti Z^{(u)} = \widetilde {\conti Z}^{(u)}$ almost surely.	
\end{proof}

Now for $u\in (0,1)$ denote by $\conti Z_e^{(u)}$ the strong solution of \cref{eq:flow_SDE} driven by $\conti W_e$ provided by the previous theorem. Note that the process $\conti Z_e^{(u)}$ is a continuous process on the interval $[u,1)$. Since $\conti Z_e^{(u)}(u) = 0$, we extend continuously $\conti Z_e^{(u)}$ on $[0,1)$ by setting $\conti Z_e^{(u)}(t) = 0$ for $0\leq t\leq u$. It will turn out (see \cref{prop:coal_con_cond}) that $\conti Z_e^{(u)}$ can also be extended continuously at time $1$.

  A remark similar to \cref{rem:no_flow} holds for the family $\{\conti Z_e^{(u)}\}_{u\in(0,1)}$, that is, we can only guarantee that for almost every $\omega$, $\conti Z_e^{(u)}$ is a solution of \cref{eq:flow_SDE} for almost every $u\in(0,1)$. Denote by $(\conti L_e^{(u)})_{u\leq t < 1}$ the local time process at zero of the semimartingale $\conti Z_e^{(u)}$ on $[u,1)$. By convention, set $\conti L_e^{(u)}(t) = 0$ for $0\leq t <u$.

\begin{definition}\label{defn:cont_coal_proc}
We call \emph{continuous coalescent-walk process} (driven by $\conti W_e$) the collection of stochastic processes $\left\{\conti Z^{(u)}_e\right\}_{u\in(0,1)}$. 
\end{definition}

We can now prove a scaling limit result for finite-volume coalescent-walk processes. We first deal with the case of a single trajectory $\conti Z_n^{(u)}$. Then we consider a more general case in \cref{thm:discret_coal_conv_to_continuous}.

\begin{proposition}\label{prop:coal_con_cond}
	Fix $u\in (0,1)$.
	The stochastic process ${\conti Z}_e^{( u)}$ can be extended to a continuous function on $[0,1]$ by setting ${\conti Z}_e^{( u)}(1) = 0$, and we
	have the following joint convergence in the product space of continuous functions $\mathcal C([0,1],\R^2)\times \mathcal C([0,1],\R)\times \mathcal C([0,1),\R)$: 
	\begin{equation}
	\label{eq:coal_con_cond}
	\left({\conti W}_n,{\conti Z}^{( u)}_n,{\conti L}^{( u)}_n \right)
	\xrightarrow[n\to\infty]{d}
	\left(\conti W_e,\conti Z_e^{(u)},\conti L_e^{(u)}\right).
	\end{equation}
\end{proposition}

\begin{remark}
	Note that the convergence of local times does not go up to time 1. This will be corrected later in \cref{lem:local_time_does_not_disappear} for a uniformly random starting point, using a combinatorial argument.
\end{remark}
\begin{proof}
	The convergence in distribution ${\conti W}_n\xrightarrow[n\to\infty]{d}\conti W_e$ is \cref{prop:DW}.
	Now let $0<\eps<u\wedge(1-u)$. By construction, $\left(({\conti W}_n-{\conti W}_n(u))|_{[u,1-\eps]},{\conti Z}^{( u)}_n|_{[u,1-\eps]},{\conti L}^{( u)}_n|_{[u,1-\eps]} \right)$
	is a measurable functional of $(\bm W_{n}(k) - \bm W_{n}(\lfloor \eps n \rfloor))_{\lfloor \eps n \rfloor\leq k \leq \lfloor (1-\eps)n \rfloor}$. Using \cref{thm:coal_con_uncond} together with
	 \cref{lem:AbsCont,lem:LLT,prop:brown_ex}, we get that  (the arguments are similar to the ones used in the proof of \cref{prop:DW})
	\begin{multline*}
	\left(({\conti W}_n - \conti W_n(u))|_{[u,1-\eps]},{\conti Z}^{( u)}_n|_{[u,1-\eps]},{\conti L}^{( u)}_n|_{[u,1-\eps]} \right)\\
	\xrightarrow[n\to\infty]{d}
	\left((\conti W_e - \conti W_e(u))|_{[u,1-\eps]},\conti Z_e^{(u)}|_{[u,1-\eps]},\conti L_e^{(u)}|_{[u,1-\eps]}\right).
	\end{multline*}
As in the proof of \cref{thm:coal_con_uncond}, we use Prokorov's theorem twice, and obtain that the sequence
\begin{equation}\label{eq:cv_1}
\left( \conti W_n,\,\,\left(({\conti W}_n - \conti W_n(u))|_{[u,1-\eps]},{\conti Z}^{( u)}_n|_{[u,1-\eps]},{\bm L}^{( u)}_n|_{[u,1-\eps]} \right)_{\eps \in \mathbb Q \cap (0,u\wedge 1-u)}
\right)
\end{equation}
is tight. The only possible limit in distribution is
\begin{equation}\label{eq:cv_2}
\left( \conti W_e,\,\,\left((\conti W_e - \conti W_e(u))|_{[u,1-\eps]},\conti Z_e^{(u)}|_{[u,1-\eps]},\conti L_e^{(u)}|_{[u,1-\eps]}\right)_{\eps \in \mathbb Q \cap (0,u\wedge 1-u)}
\right)
\end{equation}
because the processes ${\bm Z}^{( u)}_n|_{[u,1-\eps]}$ and ${\bm L}^{( u)}_n|_{[u,1-\eps]} $ are measurable functionals of $(\bm W_{n}(k) - \bm W_{n}(\lfloor \eps n \rfloor))_{\lfloor \eps n \rfloor\leq k \leq \lfloor (1-\eps)n \rfloor}$ and the restriction mapping that sends $\conti W_n$ to $(\conti W_n - \conti W_n(u))|_{[u,1-\eps]}$ is continuous (and so this relation must carry over to the limit). Hence there is convergence in distribution of the sequence in \cref{eq:cv_1} to the limit in \cref{eq:cv_2}. We may now use Skorokhod's theorem to obtain a large probability space where almost surely, we have uniform convergence on $[0,1]$ of ${\conti W}_n$ to $\conti W_e$, uniform convergence on $[u,1-\eps]$ of ${\conti Z}^{( u)}_n$ to $\conti Z_e^{(u)}$ and ${\conti L}^{( u)}_n$ to ${\conti L}_e^{( u)}$ for every rational $\eps>0$.

We can now use the deterministic bound $-{\bm X}_n \leq  {\bm Z}^{(k)}_n \leq {\bm Y}_n$ (easily proven by induction) which implies that 
\begin{multline*}
	\sup_{k,\ell\geq n}\lVert {\conti Z}^{(u)}_k - \conti Z_{\ell}^{(u)} \rVert_{[0,1]}
	\leq
	2\sup_{k\geq n} \lVert {\conti Z}^{( u)}_{k} -
		 \conti Z_e^{(u)} \rVert_{[0,1-\eps]}+2\sup_{k\geq n} \lVert {\conti Z}^{( u)}_{k} \rVert_{[1-\eps,1]}\\
	 \leq 2\sup_{k\geq n} \lVert {\conti Z}^{( u)}_{k} -
		 \conti Z_e^{(u)} \rVert_{[0,1-\eps]} + 2\lVert \conti W_e \rVert_{[1-\eps,1]} + 2\sup_{k\geq n}\lVert {\conti W}_k -\conti W_e \rVert_{[1-\eps,1]}.
\end{multline*}
Taking $n$ to infinity yields $\limsup_{n\geq 1} \sup_{k,\ell\geq n}\lVert {\conti Z}^{(u)}_k - \conti Z_{\ell}^{(u)} \rVert_{[0,1]} \leq \lVert \conti W_e \rVert_{[1-\eps,1]}$. Since $\eps$ is arbitrary,  $( {\conti Z}_n^{(u)})_n$ is actually a Cauchy sequence in $\mathcal{C}([0,1],\R)$ and converges uniformly to a continuous function, which necessarily takes value zero at time $1$ and coincides with $\conti Z^{(u)}$ on $[0,1)$. 
\end{proof}
We finish by stating a version of the previous result, for countably many uniformly chosen starting points. This is the foundation upon which the next section is built.
\begin{theorem}
	\label{thm:discret_coal_conv_to_continuous}
	Let $(\bm u_i)_{i\in\Z_{>0}}$ be a sequence of i.i.d.\ uniform random variables on $[0,1]$, independent of all other variables. We have the following joint convergence in the product space of continuous functions $\mathcal C([0,1],\R^2)\times (\mathcal C([0,1],\R)\times \mathcal C([0,1),\R))^{\mathbb Z_{>0}}$:
	\begin{equation*}
	\left(
	{\conti W}_n,\left({\conti Z}^{(\bm u_i)}_n, {\conti L}^{(\bm u_i)}_n\right)_{i\in\Z_{>0}}
	\right)
	\xrightarrow[n\to\infty]{d}
	\left( \conti {W}_e,\left({\conti Z}^{(\bm u_i)}_e, {\conti L}^{(\bm u_i)}_e\right)_{i\in\Z_{>0}}
	\right).
	\end{equation*}
\end{theorem}

\begin{proof}
	Fix $u_1,\ldots,u_k\in (0,1)$. Joint tightness and the fact that $\conti Z_e^{(u)}$ and $\conti L_e^{(u)}$ are measurable functions of $\conti W_e$, imply that convergence in distribution in \cref{eq:coal_con_cond} holds jointly for $u\in \{u_1,\ldots,u_k\}$. This means that for every bounded continuous $\varphi: \mathcal C([0,1],\R^2)\times (\mathcal C([0,1],\R)\times \mathcal C([0,1),\R))^{\Z_{>0}}\to \R$, 
	\[\E\left[\varphi\left(
	{\conti W}_n,\left({\conti Z}^{( u_i)}_n, {\conti L}^{( u_i)}_n\right)_{1\leq i \leq k}
	\right)
	\right] \to 
	\E\left[\varphi
	\left( \conti {W}_e,\left({\conti Z}^{( u_i)}_e, {\conti L}^{( u_i)}_e\right)_{1\leq i \leq k}\right)	
	\right].\]
	With dominated convergence one can integrate this over $u_1,\ldots,u_k\in [0,1]$, which by Fubini--Tonelli's theorem gives
	\[\E\left[\varphi\left(
	{\conti W}_n,\left({\conti Z}^{(\bm u_i)}_n, {\conti L}^{(\bm u_i)}_n\right)_{1\leq i \leq k}
	\right)
	\right] \to 
	\E\left[\varphi
	\left( \conti {W}_e,\left({\conti Z}^{(\bm u_i)}_e, {\conti L}^{(\bm u_i)}_e\right)_{1\leq i \leq k}\right)	
	\right].\]
	As $k$ is arbitrary, this is the claim of convergence in distribution in the product topology.
\end{proof}

\section{Scaling limits of Baxter permutations and bipolar orientations}
\label{sec:final}

This section is split in two parts: in the first one, we construct the Baxter permuton (see \cref{defn:Baxter_perm}) from the continuous coalescent-walk process $\conti Z_e= \{\conti Z_e^{(u)}\}_{u\in [0,1]}$ introduced in \cref{defn:cont_coal_proc}, and we show that it is the limit of uniform Baxter permutations (see \cref{thm:permuton}). We also show that this convergence holds jointly with the one for the coalescent-walk process (proved in \cref{thm:discret_coal_conv_to_continuous}). Building on these results, in the second part, we prove a joint (scaling limit) convergence result for all the objects considered in this paper, i.e.\ tandem walks, Baxter permutations, bipolar orientations and coalescent-walk processes (see \cref{thm:joint_scaling_limits}). In both cases, a key ingredient is the convergence of the discrete coalescent-walk process to its continuous counterpart (\cref{thm:discret_coal_conv_to_continuous}).

\subsection{The permuton limit of Baxter permutations}

We recall some basic results on permuton limits that we need for this section. For a complete introduction we refer the reader to \cite[Section 2]{bassino2017universal} and references therein.

Firstly, the space of permutons $\mathcal M$, equipped with the topology of weak convergence of measures, is compact and metrizable by the metric $d_{\square}$ defined as follows: for every pair of permutons $(\mu,\mu'),$ 
$$d_{\square}(\mu,\mu')\coloneqq\sup_{R\in\mathcal{R}}|\mu(R)-\mu'(R)|,$$
where $\mathcal R$ denotes the set of rectangles contained in $[0,1]^2.$

We need now to define the permutation induced by $k$ points in the square $[0,1]^2$. Take a sequence of $k$ points $(X,Y)=((x_1,y_1),\dots, (x_k,y_k))$ in $[0,1]^2$ in general position, i.e.\ with distinct $x$ and $y$ coordinates. 
We denote by $\left((x_{(1)},y_{(1)}),\dots, (x_{(k)},y_{(k)})\right)$ the $x$-reordering of $(X,Y)$,
i.e.\ the unique reordering of the sequence $((x_1,y_1),\dots, (x_k,y_k))$ such that
$x_{(1)}<\cdots<x_{(k)}$.
Then the values $(y_{(1)},\ldots,y_{(k)})$ are in the same
relative order as the values of a unique permutation of size $k$, that we call the \emph{permutation induced by} $(X,Y)$.

Let $\bm\mu$ be a random permuton and $((\bm X_i,\bm Y_i))_{i\geq 1}$ be an i.i.d.\ sequence with distribution $\bm \mu$ conditionally\footnote{For a construction of a probability space where this is possible see \cite[Section 2.1]{bassino2017universal}.} on $\bm \mu$. On this space, we denote by $\Perm_k(\bm \mu)$ the permutation induced by $((\bm X_i,\bm Y_i))_{1\leq i \leq k}$. The following concentration result shows that $\bm \mu$ is close to $\Perm_k(\bm \mu)$ in probability when $k$ is large.

\begin{lemma}[Approximation of a random permuton by a random permutation {\cite[Lemma 2.3]{bassino2017universal}}]
	\label{lem:subpermapproxpermuton}
	There exists $k_0$ such that if $k>k_0$, 
	\[\Prob\left[
	d_{\square}(\mu_{\Perm_k(\bm{\nu})},\bm{\nu})
	\geq 16k^{-1/4}\right]
	\leq \frac12 e^{-\sqrt{k}}, \quad \text{for any random permuton $\bm{\nu}$.}
	\]
\end{lemma}

This lemma may be used to prove a nice characterization of permuton convergence in distribution: $\bm \mu_n$ converges to $\bm \mu$ in distribution if and only if $\Perm_k(\bm \mu_n)$ converges to $\Perm_k(\bm \mu)$ in distribution for every $k\geq 1$ \cite[Theorem 2.5]{bassino2017universal}. We will not use this result here but rather directly refer to the lemma above.

\medskip

We now introduce the candidate limiting permuton for Baxter permutations. Its definition is rather straightforward by analogy with the discrete case (see \cref{sect:from_coal_to_perm}). We consider the continuous coalescent-walk process ${\conti Z_e} = \{\conti Z_e^{(t)}\}_{t\in [0,1]}$. Actually $\conti Z_e^{(t)}$ was not defined for $t\in \{0,1\}$ (see \cref{defn:cont_coal_proc}). As what happens on a negligible subset of $[0,1]$ is irrelevant to the arguments to come, this causes no problems.

We first define a random binary relation $\leq_{\conti Z_e}$ on $[0,1]^2$ as follows (this is an analogue of the definition given in \cref{eq:coal_to_perm} page~\pageref{eq:coal_to_perm} in the discrete case):
\begin{equation}\label{eq:cont_coal_to_perm}
\begin{cases}t\leq_{\conti Z_e} t &\text{ for every }t\in [0,1],\\
t\leq_{\conti Z_e} s &\text{ for every }0\leq t<s \leq 1\text{ such that }\ \conti Z_e^{(t)}(s)<0,\\
s\leq_{\conti Z_e} t,&\text{ for every }0\leq t<s \leq 1\text{ such that }\ {\conti Z}_e^{(t)}(s)\geq0.\end{cases}
\end{equation}
Note that the map $(\omega, t,s)\mapsto \idf_{t\leq_{\conti Z_e} s}$ is measurable.

\begin{proposition} \label{prop:total_order}The relation $\leq_{\conti Z_e}$ is antisymmetric and reflexive. Moreover, there exists a random set $\bm A \subset [0,1]^2$ of a.s.\ zero Lebesgue measure, i.e.\ $\P(\Leb(\bm A)=0)=1$, such that the restriction of $\leq_{\conti Z_e}$ to $[0,1]^2\setminus \bm A$ is transitive almost surely.\end{proposition}
\begin{proof}
	Antisymmetry and reflexivity are immediate by definition. Therefore we just have to prove transitivity.
	
	Let us start by showing that, almost surely,  two distinct trajectories of the coalescent-walk process ${\conti Z_e} = \{\conti Z_e^{(t)}\}_{t\in [0,1]}$ do not cross. That is, if $0\leq r\leq s\leq t< 1$, and $\conti Z^{(r)}_e(s)\leq\conti Z^{(s)}_e(s)$, then $\conti Z^{(r)}_e(t)\leq\conti Z^{(s)}_e(t)$ almost surely. By contradiction, if $\conti Z^{(r)}_e(t)>\conti Z^{(s)}_e(t)$, then upon exchanging the trajectories when they first meet, one provides another solution of the SDE \eqref{eq:flow_SDE} page~\pageref{eq:flow_SDE} started at time $r$, in negation of the uniqueness claim (\cref{thm:ext_and_uni_excursion}, item 3). 
	
	By Fubini--Tonelli's theorem, there exists a random set $\bm A$ with $\P(\Leb(\bm A)=0)=1$, such that this non-crossing property holds on $[0,1]^2\setminus \bm A$ almost surely. From this result, the proof that $\leq_{\conti Z_e}$ is transitive on $[0,1]^2\setminus \bm A$ is the same as in the discrete case (see \cref{prop:tot_ord}).
\end{proof}

We now define a random function that encodes the total order $\leq_{\conti Z_e}$:
\begin{multline}\label{eq:level_function}
\varphi_{\conti Z_e}(t)\coloneqq\Leb\left( \big\{x\in[0,1]|x \leq_{\conti Z_e} t\big\}\right)\\
=\Leb\left( \big\{x\in[0,t)|\conti Z_e^{(x)}(t)<0\big\} \cup \big\{x\in[t,1]|\conti Z_e^{(t)}(x)\geq0\big\} \right),
\end{multline}
where here $\Leb(\cdot)$ denotes the one-dimensional Lebesgue measure. Note that since the mapping $(\omega, t,s)\mapsto \idf_{t\leq_{\conti Z_e} s}$ is measurable, the mapping $(\omega, t)\mapsto \varphi_{\conti Z_e}(t)$ is measurable too.

\begin{observation}
	Note that the function defined in \cref{eq:level_function} is inspired by the following: if $\sigma$ is the Baxter permutation associated with a coalescent-walk process $Z=\{Z^{(t)}\}_{t\in[n]}\in\mathcal{C}$, then
	$\sigma(i)=\#\{j\in[n]|j\leq_Z i\}$.	
\end{observation}

\begin{definition}\label{defn:Baxter_perm}
The \emph{Baxter permuton} $\bm \mu_B$ is the push-forward of the Lebesgue measure on $[0,1]$ via the mapping $(\Id,\varphi_{\conti Z_e})$, that is
\begin{equation}
\label{eq:def_permuton}
\bm \mu_B(\cdot)\coloneqq(\Id,\varphi_{\conti Z_e})_{*}\Leb (\cdot)= \Leb\left(\{t\in[0,1]|(t,\varphi_{\conti Z_e}(t))\in \cdot \,\}\right).
\end{equation} 
\end{definition}

The Baxter permuton $\bm \mu_B$ is a random measure on the unit square $[0,1]^2$ and the terminology is justified by the following lemma, that also states some results useful for the proof of \cref{thm:permuton}. The second item of the lemma is proved using similar ideas as for \cite[Proposition 3.1]{maazoun}.

\begin{lemma}\label{lem:boundary} The following claims hold:
\begin{enumerate}
\item For $0<t<s<1$, $\conti Z^{(t)}_e(s) \neq 0$ almost surely.
\item The random measure $\bm \mu_B$ is a.s.\ a permuton.
\item Almost surely, for almost every $t<s\in[0,1]$, 
	either
	$\conti Z_e^{(t)}(s)>0$ and $\varphi_{\conti Z_e}(s)<\varphi_{\conti Z_e}(t)$, or  $\conti Z_e^{(t)}(s)<0$ and $\varphi_{\conti Z_e}(s)>\varphi_{\conti Z_e}(t)$.
\end{enumerate}
\end{lemma}

\begin{proof}
We start by proving the first claim. Let $\eps>0$ be such that $0<t<s<1-\eps<1$. As we have seen in the proof of \cref{thm:ext_and_uni_excursion}, $(\conti Z^{(t)}_e(t+r))_{0\leq r \leq 1-t-\eps}$ is absolutely continuous with regards to a Brownian motion of lifetime $1-t-\eps$. As a result, $\conti Z^{(t)}_e(s) \neq 0$ almost surely, proving claim 1.

	For the second claim, by definition, the measure $\bm \mu_B$ is a probability measure on the unit square and its first marginal is almost surely uniform. As such, to prove claim 2 we simply have to check that 
	\begin{equation}
	\label{eq:goal_of_proof}
	(\varphi_{\conti Z_e})_{*}\Leb=\Leb \text{ a.s.}
	\end{equation}
	Let $(\bm U_i)_{i\in \Z_{>0}}$ be i.i.d.\ uniform random variables on $[0,1]$. Set for $k\geq 2$,
	$$\bm U_{1,k}\coloneqq\tfrac 1 {k-1}{\#\Big\{i\in [2,k]\Big|\bm U_i\leq_{\conti Z_e} \bm U_1 \Big\}}.$$
	The random variables $\left(\mathds{1}_{\{\bm U_i\leq_{\conti Z_e} \bm U_1\}}\right)_{i\geq 2}$ are i.i.d., conditionally on $(\conti W_e,\bm U_1)$. Thus by the law of large numbers $\bm U_{1,k}$ converges almost surely as $k$ tends to infinity to 
	$$\Prob[\bm U_{2} \leq_{\conti Z_e} \bm U_{1} \mid \conti W_e, \bm U_1] = \Leb\left( \big\{x\in[0,1]\mid x \leq_{\conti Z_e} \bm U_1\big\}\right)=\varphi_{\conti Z_e}(\bm U_1).$$
	On the other hand, by the exchangeability of the $\bm U_i,$ and using claim 1, the random variable $\bm U_{1,k}$ is uniform in $\big\{\frac{0}{k-1},\dots, \frac{k-1}{k-1}\big\},$ conditionally on $\conti W_e$. Therefore $\varphi_{\conti Z_e}(\bm U_1)$ is uniform on $[0,1]$ conditionally on $\conti W_e$. This proves \cref{eq:goal_of_proof} and claim 2.
	
	For the third claim, consider a pair of independent uniform random variables $\bm U$ and $\bm V$ independent of $\conti W_e$. It is immediate from \cref{prop:total_order} that if $\bm U \leq \bm V$ and $\conti Z_e^{(\bm U)}(\bm V)> 0$ then $\varphi_{\conti Z_e}(\bm U)\geq\varphi_{\conti Z_e}(\bm V)$ a.s., and if $\bm U \leq \bm V$ and $\conti Z_e^{(\bm U)}(\bm V)<0$ then $\varphi_{\conti Z_e}(\bm U)\leq\varphi_{\conti Z_e}(\bm V)$ a.s. The equality case $\conti Z_e^{(\bm U)}(\bm V) = 0$ is almost surely excluded by claim 1, and the equality case  $\varphi_{\conti Z_e}(\bm U) = \varphi_{\conti Z_e}(\bm V)$ is almost surely excluded by the fact that $\varphi_{\conti Z_e}(\bm U)$ and $\varphi_{\conti Z_e}(\bm V)$ are two independent uniform random variables thanks to \cref{eq:goal_of_proof}. This proves claim 3.
\end{proof}

We can now prove that Baxter permutations converge in distribution to the Baxter permuton. Since it will be useful in the next section, we also show that this convergence is joint with the convergence of the corresponding tandem walk and the corresponding coalescent-walk process.

We reuse the notation of \cref{sec:cond_conv}. In particular, 
${\bm W}_n$ is a uniform element of the space of tandem walks $\mathcal W_n$, ${\bm Z}_n=\wcp({\bm W}_n)$ is the associated uniform coalescent-walk process, and ${\bm \sigma}_n = \cpbp({\bm Z}_n)$ is the associated uniform Baxter permutation.

\begin{theorem}\label{thm:permuton}
	Jointly with the convergences in \cref{thm:discret_coal_conv_to_continuous}, we have that $\mu_{\bm \sigma_n} \xrightarrow{d} \bm \mu_B$.
\end{theorem}

\begin{proof}
	Using the notation of \cref{thm:discret_coal_conv_to_continuous}, we have to show that
	\begin{equation}\label{eq:coal_con_uncond2}
	\left(
	{\conti W}_n,\left({\conti Z}^{(\bm u_i)}_n, {\conti L}^{(\bm u_i)}_n\right)_{i\in\Z_{>0}},\mu_{\bm \sigma_n}
	\right)
	\xrightarrow[n\to\infty]{d}
	\left( \conti {W}_e,\left({\conti Z}^{(\bm u_i)}_e, {\conti L}^{(\bm u_i)}_e\right)_{i\in\Z_{>0}},\bm \mu_B
	\right).
	\end{equation}
	Since from \cref{thm:discret_coal_conv_to_continuous} we know that $\left({\conti W}_n,\left({\conti Z}^{(\bm u_i)}_n, {\conti L}^{(\bm u_i)}_n\right)_{i\in\Z_{>0}}\right)$ is a convergent sequence of random variables and the space of permutons $\mathcal M$ is compact, by Prokhorov's theorem, both $\left({\conti W}_n,\left({\conti Z}^{(\bm u_i)}_n, {\conti L}^{(\bm u_i)}_n\right)_{i\in\Z_{>0}}\right)$ and $\mu_{\bm \sigma_n}$ are tight sequences of random variables.
	
	Since the product of two compact sets is compact, then the left-hand side of \cref{eq:coal_con_uncond2} forms a tight sequence. Therefore, again by Prokhorov's theorem, it is enough to identify the distribution of all joint subsequential limits in order to show the convergence in \cref{eq:coal_con_uncond2}.
	
	Assume that along a subsequence, we have 
	\begin{equation}\label{eq:bwcibcwouec}
	\left(
	{\conti W}_n,\left({\conti Z}^{(\bm u_i)}_n, {\conti L}^{(\bm u_i)}_n\right)_{i\in\Z_{>0}},\mu_{\bm \sigma_n}
	\right)
	\xrightarrow[n\to\infty]{d}
	\left( \conti {W}_e,\left({\conti Z}^{(\bm u_i)}_e, {\conti L}^{(\bm u_i)}_e\right)_{i\in\Z_{>0}},\widetilde{\bm \mu}
	\right).
	\end{equation}
	The joint distribution of the pair 
	$$\Bigg(\left( \conti {W}_e,\left({\conti Z}^{(\bm u_i)}_e, {\conti L}^{(\bm u_i)}_e\right)_{i\in\Z_{>0}}\right), \widetilde{\bm \mu}\Bigg)$$ is unknown for now, but we will show that $\widetilde{\bm \mu}=\bm \mu_B$ almost surely, which will complete the proof.
	
	To simplify things, we assume that the subsequential convergence in \cref{eq:bwcibcwouec} is almost sure using Skorokhod's theorem. In particular, almost surely as $n\to\infty$, $\mu_{\bm \sigma_n} \to \widetilde{\bm \mu}$ in the space of permutons, and for every $i\geq 1$, $\conti Z_{n}^{(\bm u_i)} \to \conti Z_{e}^{(\bm u_i )}$ uniformly on $[0,1]$, where $(\bm u_i)_{i\geq 1}$ are i.i.d.\ uniform random variables on $[0,1]$.

	Fix $k\in\Z_{>0}$. We denote by $\bm \rho^k_n$ the pattern induced by $\bm \sigma_n$ on the indices $ \lceil n \bm u_1 \rceil,\ldots,\lceil n \bm u_k \rceil$ ($\bm \rho^k_n$ is undefined if two indices are equal). From the uniform convergence above, and recalling that ${\conti Z}^{(u)}_{n}\left(\frac kn\right) = \frac 1 {\sqrt {2n}} \bm Z^{(\lceil nu\rceil)}_{n}(k)$, we have for all $1\leq i<j\leq k$ that
	\begin{equation*}
	\sgn\left({\bm Z}_n^{( \lceil n \bm u_i \rceil \wedge \lceil n \bm u_j \rceil)}(\lceil n \bm u_i \rceil \vee \lceil n \bm u_j \rceil)\right)\xrightarrow[n\to\infty]{}\sgn(\conti Z_e^{(\bm u_i \wedge \bm u_j)}(\bm u_j \vee \bm u_i))\quad \text{ a.s.}
	\end{equation*}
	Note that the function $\sgn$ is not continuous, but by the third claim of \cref{lem:boundary}, the random variable $\conti Z_e^{(\bm u_i \wedge \bm u_j)}(\bm u_j \vee \bm u_i)$ is almost surely nonzero, hence a continuity point of $\sgn$.
	By \cref{prop:patterns}  and the third claim of \cref{lem:boundary}, this means that $\bm \rho_n^k \xrightarrow[n\to\infty]{} \bm \rho^k$, where $\bm \rho^k$ denotes the permutation $\Perm_k(\bm \mu_B)$ induced by $(\bm u_i, \varphi_{\conti Z_e}(\bm u_i))_{i \in [k]}$.
	Using \cref{lem:subpermapproxpermuton}, we have for $k$ large enough that
	\begin{equation}\label{eq:first_perm_bound}
	\Prob\left[
	d_{\square}(\mu_{\bm \rho_n^k},\mu_{\bm \sigma_n})
	> 16k^{-1/4}\right]
	\leq \frac12 e^{-\sqrt{k}} + O(n^{-1}),
	\end{equation}
	where the error term $O(n^{-1})$ comes from the fact that $\bm \rho^k_n$ might be undefined.
	Since $\bm \rho_n^k \xrightarrow[n\to\infty]{} \bm \rho^k$ and $\mu_{\bm \sigma_n} \to \widetilde{\bm \mu}$, then
	taking the limit as $n\to \infty$, we obtain that
	\[\Prob\left[
	d_{\square}(\mu_{\bm \rho^k},\widetilde {\bm \mu})
	> 16k^{-1/4}\right]
	\leq \frac12 e^{-\sqrt{k}},\]
	and so $\mu_{\bm \rho^k}\xrightarrow[k\to\infty]{P}\widetilde {\bm \mu}$. Another application of \cref{lem:subpermapproxpermuton} gives that $\mu_{\bm \rho^k} \xrightarrow[k\to\infty]{P}\bm \mu_B$. The last two limits yield $\widetilde {\bm \mu} = \bm \mu_B$ almost surely. This concludes the proof.
\end{proof}

\subsection{Joint convergence of the four trees of bipolar orientations}
Fix $n\geq 1$. Let $\bm m_n$ be a uniform bipolar orientation of size $n$, and consider its iterates $\bm m_n^*, \bm m_n^{**}, \bm m_n^{***}$ by the dual operation. Denote $\bigstar = \{\emptyset,*,**,{**}*\}$ the group of dual operations, that is isomorphic to $\Z/4\Z$. For $\theta \in \bigstar$, let $\bm W_n^\theta =(\bm X_n^\theta, \bm Y_n^\theta)$, $\bm Z_n^\theta$ and $\bm \sigma_n^\theta$ be the uniform objects corresponding to $\bm m_n^\theta$ via the commutative diagram in \cref{eq:comm_diagram} page~\pageref{eq:comm_diagram}. We also denote by $\bm L^\theta_n$  the discrete local time process of $\bm Z^\theta_n$ (see \cref{eq:local_time_process} page~\pageref{eq:local_time_process} for a definition).
We define rescaled versions as usual: for $u\in [0,1]$, let ${\conti W}_n^\theta:[0,1]\to \R^2$, ${\conti Z}^{\theta,(u)}_{n}:[0,1]\to\R$ and ${\conti L}^{\theta,(u)}_{n}:[0,1]\to\R$ be the continuous functions obtained by linearly interpolating the following families of points defined for all $k\in [n]$:
\begin{equation*}
{\conti W}_n^\theta\left(\frac kn\right) = \frac 1 {\sqrt {2n}} {\bm W}_{n}^\theta (k), 
\quad 
{\conti Z}^{\theta,(u)}_{n}\left(\frac kn\right) = \frac 1 {\sqrt {2n}} {\bm Z}^{\theta,(\lceil nu\rceil)}_{n}(k),
\quad 
{\conti L}^{\theta,(u)}_{n}\left(\frac kn\right) = \frac 1 {\sqrt {2n}} {\bm L}^{\theta,(\lceil nu\rceil)}_{n}(k).
\end{equation*}
 Finally, for each $n\in\Z_{>0}$, let $((\bm u_{n,i}, \bm u_{n,i}^{*}))_{i\geq 1}$ be an i.i.d.\ sequence of distribution $\mu_{\bm \sigma_n}$ conditionally on $\bm m_n$. Let also $\bm u_{n,i}^{**} = 1-\bm u_{n,i}^{}$ and $\bm u_{n,i}^{***} = 1-\bm u_{n,i}^{*}$ for $n,i\in\Z_{>0}$. The first and second marginals of a permuton are uniform irregardless of the permuton, which implies that for all $n\in\Z_{>0}$ and $\theta \in \bigstar$, $(\bm u_{n,i}^\theta)_{i\geq 1}$ is an i.i.d.\ sequence of uniform random variables on $[0,1]$ independent of $\bm m_n^\theta$ (but for every fixed $n\in\Z_{>0}$, the joint distribution of $\left((\bm u_{n,i}^\theta)_{i\geq 1}\right)_{\theta\in \bigstar}$ depends on $(\bm m^\theta_n)_{\theta\in \bigstar}$).
 
\medskip

We can now state one of the main theorems of this paper which is in some sense (made precise later) a joint scaling limit convergence result for all these objects. 
Recall the time-reversal and coordinate-swapping mapping $s:\mathcal C([0,1],\R^2) \to \mathcal C([0,1],\R^2)$ defined by $s(f,g) = (g(1-\cdot), f(1-\cdot))$.

\begin{theorem}\label{thm:joint_scaling_limits}
	Let $\conti W_e$ be a two-dimensional Brownian excursion of correlation $-1/2$ in the non-negative quadrant. Let $\conti Z_e$ be the associated continuous coalescent-walk process and $\conti L_e$ be its local-time process. Let $\bm u$ denote a uniform random variable in $[0,1]$ independent of $\conti W_e$. Then
\begin{enumerate}
\item almost surely, $\conti L_e^{(\bm u)} \in \mathcal C([0,1),\R)$ has a limit at $1$, and we still denote by $\conti L_e^{(\bm u)} \in \mathcal C([0,1],\R)$ its extension.
\item 
There exists a measurable mapping $r:\mathcal C([0,1],\R^2) \to \mathcal C([0,1],\R^2)$ such that almost surely, denoting $(\widetilde{\conti X}, \widetilde{\conti Y}) = r(\conti W_e)$,
\begin{equation}\label{eq:definition_map_r_1}
\widetilde{\conti X}(\varphi_{\conti Z_e}(\bm u)) =  \conti L_e^{(\bm u)}(1)\qquad\text{and}\qquad r(s(\conti W_e)) = s(r(\conti W_e)).
\end{equation}
These properties uniquely determine the mapping $r$ $\Prob_{\conti W_e}$-almost everywhere. Moreover,
\begin{equation}\label{eq:definition_map_r_2}
r(\conti W_e) \stackrel d= \conti W_e, \qquad r^2=s \qquad\text{and}\qquad r^4 = \Id, \qquad\Prob_{\conti W_e}-\text{a.e.}
\end{equation}
\item Let $(\bm u_{i})_{i\geq 1}$ be an auxiliary i.i.d.\ sequence of uniform random variables on $[0,1]$, independent of $\conti W_e$. For each $\theta \in \{\emptyset, *,**\}$, let $\conti W_e^{\theta*} = r(\conti W_e^{\theta})$ and  $\bm u_i^{\theta *} = \varphi_{\conti Z_e}(\bm u_i^\theta)$ for $i\geq 1$. Let also $\conti Z_e^\theta$ be the associated continuous coalescent-walk process, $\conti L_e^\theta$ be its local-time process and $\mu_{\conti Z^\theta_e}$ be the associated Baxter permuton. 
Then we have the joint convergence in distribution
	\begin{multline}\label{eq:joint_scaling_limits}
	\left(
	\conti W_n^\theta, 
	\left(\bm u_{n,i}^\theta,\conti Z^{\theta,(\bm u_{n,i}^\theta)}_n, \conti L^{\theta,(\bm u_{n,i}^\theta)}_n\right)_{i\in\Z_{>0}}, 
	\mu_{\bm \sigma_n^\theta}
	\right)_{\theta\in \bigstar}\\
	\xrightarrow[n\to\infty]{d}
	\left(
	\conti W_e^\theta, 
	\left(\bm u_i^\theta,\conti Z^{\theta,(\bm u_i^\theta)}_e, \conti L^{\theta,(\bm u_i^\theta)}_e\right)_{i\in\Z_{>0}}, 
	\mu_{\conti Z^\theta_e}
	\right)_{\theta\in \bigstar}
	\end{multline}
	in the space
	$$\left(\mathcal C([0,1], \R^2)\times ([0,1] \times \mathcal C([0,1], \R) \times \mathcal C([0,1], \R))^{\Z_{>0}} \times \mathcal M\right)^4.$$
	\item In this coupling, we almost surely have, for $\theta\in \bigstar$,
	\begin{equation}
	\varphi_{\conti Z_e^{\theta*}}\circ \varphi_{\conti Z_e^\theta} = 1 -\Id,\qquad\Prob_{\conti W_e}-\text{a.e.}
	\end{equation}
\end{enumerate}
\end{theorem}

\begin{remark}
	As in the discrete case, we point out that even though the joint distribution of $\left((\bm u^{\theta}_{i})_{i\geq 1}\right)_{\theta\in\bigstar}$ depends on $(\conti W_e^{\theta})_{\theta\in \bigstar}$, we have that $(\bm u^{\theta}_{i})_{i\geq 1}$ is independent of $\conti W_e^{\theta}$ for every fixed $\theta\in\bigstar$.
\end{remark}

\begin{remark}
	We highlight that the results presented in the theorem above (in particular in \cref{eq:definition_map_r_1,eq:definition_map_r_2}) are continuous analogs of the results obtained in \cref{sect:anti-invo} for discrete objects (see in particular \cref{thm:discrete_invo}). The specific connections between the results for continuous and discrete objects are made clear in the proof of the theorem.
\end{remark}

\begin{proof}[Proof of \cref{thm:joint_scaling_limits}]
We start by showing that the left-hand side of \cref{eq:joint_scaling_limits} is tight. \cref{thm:discret_coal_conv_to_continuous} and \cref{thm:permuton} give us  tightness of all involved random variables, with the caveat that $(\conti L_n^{\theta,(\bm u_i^\theta)})_n$ is a priori only tight in the space $\mathcal C([0,1),\R)$. Tightness in $\mathcal C([0,1],\R)$ follows from \cref{lem:local_time_does_not_disappear}, whose statement and proof are postponed to the end of this section, which proves in passing item 1.

We now consider a subsequence of 
\[\left(
\conti W_n^\theta, 
\left(\bm u_{n,i}^\theta,\conti Z^{\theta,(\bm u_{n,i}^\theta)}_n, \conti L^{\theta,(\bm u_{n,1}^\theta)}_n\right)_{i\geq 1}, 
\mu_{\bm \sigma_n^\theta}
\right)_{\theta\in \bigstar}
\]
converging in distribution. For fixed $\theta \in \bigstar$, we know the distribution of the limit thanks to \cref{thm:discret_coal_conv_to_continuous} and \cref{thm:permuton} (the limit of $\conti L^{\theta,(\bm u_i^\theta)}_n$, being a random continuous function on $[0,1]$, is determined by its restriction to $[0,1)$). Henceforth, it is legitimate to denote by
\begin{equation}\label{eq:limiting_vector}
\left(
\conti W_e^\theta, 
\left(\bm u_i^\theta,\conti Z^{\theta,(\bm u_i^\theta)}_e, \conti L^{\theta,(\bm u_i^\theta)}_e\right)_{i\geq 1}, 
\mu_{\conti Z^\theta_e}
\right)_{\theta\in \bigstar}
\end{equation}
the limit, keeping in mind that the coupling for varying $\theta$ is undetermined at the moment. We shall determine it to complete the proof of items 2 and 3. We start by proving the following identities (we use the convention $\text{\small****}=\emptyset$): 
\begin{align}
\conti W_e^{**} = s(\conti W_e),\quad \conti W_e^{***} = s(\conti W_e^*), \label{eq:main_reversal}\\
\bm u_i^{\theta *} = \varphi_{\conti Z^\theta_e}(\bm u_i^\theta),&\quad  i\geq 1, \theta \in {\bigstar}, \label{eq:main_phi}\\
\conti X_e^{\theta*}(\bm u_i^{\theta *}) = \conti L_e^{\theta,(\bm u_i^\theta)}(1),&\quad   i\geq 1,\theta \in {\bigstar}. \label{eq:main_L}
\end{align}

\medskip

The claim in \cref{eq:main_reversal} is the easiest. Thanks to \cref{prop:rev_coal_prop}, we have that $\conti W_n^{**} = s(\conti W_n)$ and $\conti W_n^{***} = s(\conti W_n^*)$, for every $n\in\Z_{>0}$.
Since $s$ is continuous on $\mathcal C([0,1],\R^2)$, the same result holds in the limit, proving \cref{eq:main_reversal}.

\medskip

To prove \cref{eq:main_phi}, we use the following lemma, whose proof is skipped. It follows rather directly from the definition of weak convergence of measures.
\begin{lemma}\label{lem:cv_iid_seq}
	Suppose that for $n\in \Z_{>0} \cup \{\infty\}$, $\bm \mu_n$ is a random measure on a Polish space and $(\bm X^n_i)_{i\geq 1}$ is an i.i.d.\ sequence of elements with distribution $\bm \mu_n$ conditionally on $\bm \mu_n$.
	Assume that $\bm \mu_n \to \bm \mu_\infty$ in distribution for the weak topology. Then we have the joint convergence in distribution 
	\[
	(\bm \mu_n,(\bm X^n_i)_{i\geq 1}) \xrightarrow[n\to\infty]{d} (\bm \mu_\infty,(\bm X^\infty_i)_{i\geq 1})\; .
	\]
\end{lemma}
In view of the construction of $\left(\mu_{\bm \sigma_n^\theta},(\bm u_{n,i}^\theta, \bm u_{n,i}^{\theta*})_{i\geq 1}\right)$,
it implies that the joint distribution of $\left(\mu_{\conti Z^\theta_e},(\bm u_i^\theta, \bm u_i^{\theta*})_{i\geq 1}\right)$
is that of $\mu_{\conti Z^\theta_e}$ together with an i.i.d.\ sequence of elements with distribution $\mu_{\conti Z^\theta_e}$ conditionally on $\mu_{\conti Z^\theta_e}$.
In particular, we must have $\bm u_i^{\theta *} = \varphi_{\conti Z_e^{\theta}}(\bm u_i^{\theta})$ almost surely.
This proves \cref{eq:main_phi}.

\medskip

Finally, we have the discrete identity $\conti X_n^{\theta*}(n^{-1}\lceil n\bm u_i^{\theta *}\rceil) = \conti L_n^{\theta,(\bm u_{n,i}^\theta)}(1) - \frac 1 {\sqrt{2n}}$ for every $n\geq 1$ thanks to \cref{cor:local_time}. By convergence in distribution, we obtain  \cref{eq:main_L}.

\medskip

The continuous stochastic process $\conti X_e^{\theta *}$ is almost surely determined by its values on the dense sequence $(\bm u_i^{\theta *})_{i\geq i_0}$ for all $i_0\in\Z_{>0}$. By \cref{eq:main_L} and the Kolmogorov's zero–one law, we have that $\conti X_e^{\theta *} \in \sigma(\conti W_e^\theta)$. This together with \cref{eq:main_reversal} implies that $\sigma(\conti Y_e^{\theta *}) = \sigma(\conti X_e^{\theta ***}) \subset \sigma(\conti W_e^{\theta**}) = \sigma(\conti W_e^{\theta})$. As a result $\conti W_e^{\theta*} \in \sigma(\conti W_e^{\theta})$ and so there exists a measurable mapping $r:\mathcal C([0,1],\R^2) \to \mathcal C([0,1],\R^2)$ such that 
\begin{equation}\label{eq:main_r}
r(\conti W_e^{\theta}) = \conti W_e^{\theta*}.
\end{equation} Then the claims in \cref{eq:definition_map_r_1,eq:definition_map_r_2} are an immediate consequence of \cref{eq:main_L,eq:main_reversal}. The fact that \cref{eq:definition_map_r_1} uniquely determines $r$ $\Prob_{\conti W_e}$-almost everywhere also results from the fact that a continuous function is uniquely determined by its values on a set of full Lebesgue measure. This completes the proof of item 2.

Additionally, \cref{eq:main_phi,eq:main_r} show that the coupling in \cref{eq:limiting_vector} is the one announced in the statement of item 3, and in particular is independent of the subsequence. Together with tightness, this proves item 3.

For item 4, we observe that $\bm u_{n,1}^{\theta**} = 1- \bm u_{n,1}^{\theta}$, so that taking the limit, $\bm u_{1}^{\theta**} = 1- \bm u_{1}^{\theta}$. Then item 4 follows from \cref{eq:main_phi}.
\end{proof}

\bigskip

We now move to the tightness lemma that was left aside. The proof relies heavily on the relation between the coalescent-walk process and the dual map presented in \cref{cor:local_time}. 

\begin{lemma} \label{lem:local_time_does_not_disappear}
	Let $\bm u$ be a uniform random variable on $[0,1]$, independent of ${\bm W}_n$. The sequence $(\conti L_n^ {(\bm u)} (1))_n$ is tight, and for every $\eps, \delta >0$, there exist $x\in (0,1)$ and $n_0\geq 1$ such that 
	\begin{equation}\label{eq:continuity_local_time_at_1}
	\Prob\Big(\conti L_n^ {(\bm u)} (1) -\conti L_n^ {(\bm u)} (1-x) \geq \delta\Big) \leq \eps,\quad \text{for all} \quad n\geq n_0.
	\end{equation}
	Therefore $(\conti L_n^{(\bm u)})_n$ is tight in the space $\mathcal C([0,1],\R)$.
\end{lemma}

\begin{proof}[Proof of \cref{lem:local_time_does_not_disappear}]
	Let us denote by $\bm U_n = \lceil n \bm u \rceil$, and $\bm V_n = \bm \sigma_n(\bm U_n)$, where we recall that $\bm \sigma_n=\cpbp\circ\wcp({\bm W}_n)$. Both $\bm U_n$ and $\bm V_n$ are separately independent of ${\bm W}_n$. Using \cref{cor:local_time}, we have $\conti L_n^ {(\bm u)}(1) = \frac 1 {\sqrt{2n}} \bm ({\bm X}_n^*(\bm V_n) + 1)$, from which tightness for $(\conti L_n^ {(\bm u)} (1))_n$ follows.
	We turn to the analysis of 
	\begin{equation}
	\conti L_n^ {(\bm u)} (1) -\conti L_n^ {(\bm u)} (1-x)  
	\leq \frac{1}{\sqrt{2n}}\Big(\bm L_n ^{(\bm U_n)}(n) - \bm L_n ^{(\bm U_n)}(\lfloor (1-x)n \rfloor)\Big).
	\label{eq:tight_1}
	\end{equation}
	We now consider the tree $T(\bm m^*_n)$ with edges labeled by its exploration process. From \cref{obs:ancestry line}, the quantity $\left(\bm L_n ^{(\bm U_n)}(n) - \bm L_n ^{(\bm U_n)}(\lfloor (1-x)n \rfloor)\right)$ counts the number of edges on the ancestry line of the edge $\bm V_n$ in $T(\bm m^*_n)$ with a $T(\bm m_n)$-label strictly greater than $\lfloor (1-x)n \rfloor$.
	The idea of the proof is to show the existence of an edge on this ancestry line of height less than $\delta \sqrt{2n}$  and of $T(\bm m_n)$-label less than $\lfloor (1-x)n \rfloor$ with high probability. As $T(\bm m_n)$-labels decrease going up an ancestry line of $T(\bm m_n^*)$, this will be enough (see \cref{eq:tight_2} below).
	
	Let us make this more precise. Let $\bm \Delta_n$ a random quantity to be determined later (see \cref{eq:defn_rndm_quant} below), but that will turn out to be smaller than $\delta\sqrt{2n}$ with probability bounded below.	 
	Set
	\begin{equation}\label{eq:defn_techn_quant}
	\bm B_n = \sup\{k\leq  \bm V_n : \bm X^*_n(k) = \bm \Delta_n\}\quad\text{and}\quad \bm \tau_n = \inf\{k> \bm V_n : \bm X^*_n(k) \leq \bm \Delta_n\}.
	\end{equation}
	On the event $\mathcal E_n = \{\bm X_n^*(\bm V_n) \geq \bm \Delta_n\} = \{\bm \tau_n \neq  \infty \} = \{\bm B_n \neq  -\infty \}$, the edge $\bm B_n$ is an ancestor of the edge $\bm V_n$ in $T(\bm m^*_n)$, of height $\bm X_n^*(\bm B_n) = \bm \Delta_n$. We denote by $\bm A_n = \bm \sigma_n^{-1}(\bm B_n)$ the $T(\bm m_n)$-label of this edge (see \cref{fig:schema_for_tight_lemma} for a schema of the notation).
	\begin{figure}[tbh]
		\centering
		\includegraphics[scale=0.8]{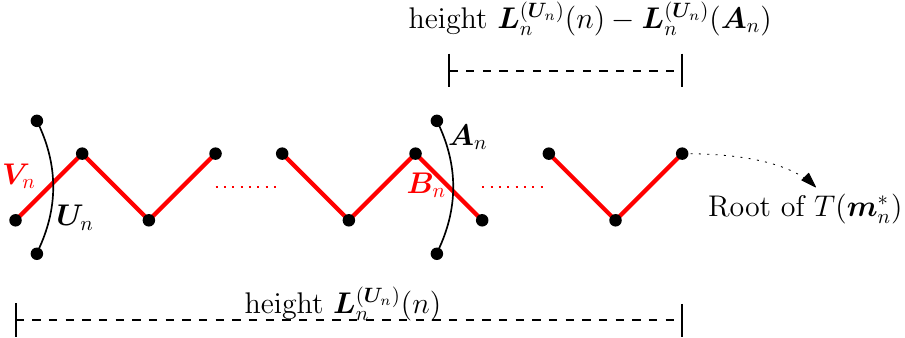}
		\caption{The ancestry line of $\bm V_n$ in $T(\bm m^*_n)$ is (partially) plotted in red together with the different quantities involved in the proof. Recall that the $T(\bm m_n)$-labels along this ancestry line increase from left to right.
			\label{fig:schema_for_tight_lemma}}
	\end{figure}

 By \cref{obs:ancestry line}, $\bm X_n^*(\bm B_n) = \bm \Delta_n = \left(\bm L_n ^{(\bm U_n)}(n) - \bm L_n ^{(\bm U_n)}(\bm A_n)\right)$. Since $\bm L_{n}^{(\bm U_n)}(\cdot)$ is increasing and non-negative, it is clear that $\bm L_n ^{(\bm U_n)}(n) - \bm L_n ^{(\bm U_n)}(\lfloor (1-x)n \rfloor)$ is bounded by $\bm \Delta_n$ unless the event $\mathcal E_n$ is realized and $\bm A_n \geq \lceil n(1-x) \rceil$.
 Translating into probabilities, 
	\begin{equation}\label{eq:tight_2}
	\Prob\Big(\conti L_n^ {(\bm u)} (1) -\conti L_n^ {(\bm u)} (1-x) \geq \delta\Big)
	\leq \Prob\Big(\mathcal E_n, \bm A_n \geq \lceil n(1-x) \rceil\Big) 
	+  \Prob\Big(\bm \Delta_n \geq \delta\sqrt{2n}\Big).
	\end{equation}
We focus on the first term in the right-hand side of the equation above. The quantity $\bm A_n$ is the index of the walk $\bm W_n$ corresponding to an edge of $\bm m_n$ whose definition is clearer from the walk $\bm W_n^*$. Hence it would be more tractable to rewrite the condition $ \bm A_n \geq \lceil n(1-x) \rceil$ in terms of the walk $\bm W_n^*$. To that end, we introduce $\eta>0$ and assume that $\max_{[\lceil n(1-x) \rceil,n]} \bm Y_n \leq \eta \sqrt{2n}$. If $\bm A_n \geq \lceil n(1-x) \rceil$ then necessarily $\bm Y_n(\bm A_n) \leq \eta \sqrt{2n}$. This trick is very useful to our purposes, as $\bm  Y_n(\bm A_n) = \bm L_{n}^{*,(\bm B_n)}(n) - 1$ (this follows once again by \cref{cor:local_time} together with \cref{thm:discrete_invo}). Finally,  we have the following inequality:
	\begin{equation}\label{eq:tight_3}
	 \Prob\Big(\mathcal E_n, \bm A_n \geq \lceil n(1-x) \rceil\Big) \leq \Prob\Big(\mathcal E_n, \bm L_{n}^{*,(\bm B_n)}(n) \leq \eta\sqrt{2n}+1\Big) + \Prob\Big(\max_{[\lceil n(1-x) \rceil,n]} \bm Y_n \geq \eta\sqrt{2n}\Big).
	\end{equation}
The first term in the right-hand side is, as desired, solely about the walk $\bm W_n^*$ and its corresponding coalescent-walk process, and we now focus on controlling it. By definition of $\bm B_n$ and $\bm \tau_n$ (see \cref{eq:defn_techn_quant}), the walk $\bm X_n^*-\bm \Delta_n$ takes a positive excursion between times ${\bm B}_n$ and ${\bm \tau}_n$
so that by construction of the coalescent-walk process, the walk ${\bm Z}^{({\bm B}_n)}_n$ takes a negative excursion between these times and weakly crosses zero upwards between times ${\bm \tau}_n - 1$ and ${\bm \tau}_n$. Hence, denoting $\bm G_n \coloneqq {\bm Z}^{({\bm B}_n)}_n(\bm \tau_n)$ and  $\bm R_n(k) \coloneqq {\bm Z}^{({\bm B}_n)}_n({\bm \tau}_n + k) - \bm G_n,\,\,k \geq 0$, we have
\begin{align*}
&\bm G_n =  {\bm Y_n}({\bm \tau}_n)-{\bm Y_n}({\bm \tau_n} - 1), \\	
&\bm L_{n}^{*(\bm B_n)}(\bm \tau_n+ k) = 1+\#\{i\in [0,k], \bm R_n(i) = -\bm G_n\}.
\end{align*}
As a result, $\bm L_{n}^{*,(\bm B_n)}(n) = \#\{i\in [0,n-\bm \tau_n], \bm R_n(i) = -\bm G_n\} +1$, and
\begin{multline}\label{eq:tight_4}
\Prob\Big(\mathcal E_n, \bm L_{n}^{*,(\bm B_n)}(n) \leq \eta\sqrt{2n}+1\Big) \leq\\ 
 \Prob\Big(\mathcal E_n, \inf_{j\in [0,n^{1/4}]}\#\{i\in [0,n-\bm \tau_n], \bm R_n(i) = -j\}\leq \eta\sqrt{2n}\Big) 
+ \Prob\Big(\max_{0\leq k \leq n}|{\bm Y}_n(k)-{\bm Y}_n(k - 1)|\geq n^{1/4}\Big)\\
\leq \Prob\Big(\mathcal E_n, \inf_{j\in [0,n^{1/4}]}\#\{i\in [0,n-\bm \tau_n], \bm R_n(i) = -j\}\leq \eta\sqrt{2n}\Big) 
+ o(1).
\end{multline}
The second term in the second line was easily treated by removing the excursion conditioning using \cref{eq:LLTexcursion} of \cref{lem:techn_estimates} and then using a union bound, yielding $$\Prob(\max_{0\leq k \leq n}|{\bm Y}_n(k)-{\bm Y}_n(k - 1)|\geq n^{1/4}) \leq Cn^4\, n \, 2^{-n^{1/4}}.$$

We turn to the first term in the right-hand side of \cref{eq:tight_4}. We use the idea that $\bm R_n$ is close in distribution to a random walk, which implies that its local time near zero in a time interval of order $n$ is indeed of order $\sqrt{n}$, and so the first term can be made small by taking $\eta$ small.
Actually, thanks to \cref{prop:trajectories_are_rw}, $\bm R_n$ would exactly be a random walk if there were no excursion conditioning on $\bm W^*_n$ and if $\bm \Delta_n$ were defined so that $\bm \tau_n$ is a stopping time of $\bm W^*_n$.
We shall use an absolute continuity argument to compare our current situation to this ideal one. 
Let $0<2u<y$ and set 
\begin{equation}\label{eq:defn_rndm_quant}
\bm \Delta_n \coloneqq \bm X^*_n(\lfloor yn \rfloor).
\end{equation}
Then
\begin{multline}\label{eq:tight_5}
\Prob(\mathcal E_n, \inf_{j\in [0,n^{1/4}]} \#\{i\in [0,n-\bm \tau_n], \bm R_n(i) = -j\}\leq \eta\sqrt{2n}) \\
\leq \Prob\Big(
\bm V_n\geq yn, 
\bm \tau_n \leq (1-2u)n,\,\,   \inf_{j\in [0,n^{1/4}]}\#\{i\in [0,un], \bm R_n(i) = -j\}\leq \eta \sqrt{2n}
\Big)\\
+ y + \Prob(\bm \tau_n \geq (1-2u)n).
\end{multline}
Our choice of definition for $\bm \Delta_n$ makes the event in the first term in the right-hand side of the equation above measurable with respect to $\sigma((\bm W^*_{n,\lfloor nu \rfloor+k} - \bm W^*_{n,\lfloor nu \rfloor})_{0\leq k\leq n - 2 \lfloor nu \rfloor},\bm V_n)$. 	By \cref{lem:AbsCont,prop:unif_law}, its probability is bounded independently of $n$ by a constant $C_u$ times the same probability under the unconditioned law (for which $\bm W_n^*$ is a random walk of step distribution $\nu$). 
Under the unconditioned law, $\bm \tau_n$ is a stopping time. Applying the strong Markov property and using \cref{prop:trajectories_are_rw}, we have that
$\bm R_n$ is a random walk of step distribution $\nu$.
So using an invariance principle for random walk local times \cite[Theorem 1.1]{Borodin}, the quantity
$\frac 1 {\sqrt n}\inf_{j\in [0,n^{1/4}]}\#\{i\in [0,un], \bm R_n(i) = -j\}$
converges in distribution to the local time at zero of a standard Brownian motion $\conti B$ during the interval $[0,u]$, which is distributed like $|\conti B_u|$. Hence the first term in the right-hand side of \cref{eq:tight_5} is bounded by $C_u(\P(|\conti B_u|\leq\eta) +o_{u,\eta}(1))$.

Combining this with the estimates in \cref{eq:tight_1,eq:tight_2,eq:tight_3,eq:tight_4,eq:tight_5}, we obtain:
\begin{multline*}
\Prob\Big(\conti L_n^ {(\bm u)} (1) -\conti L_n^ {(\bm u)} (1-x) \geq \delta\Big) \\
\leq \Prob\Big(\bm \Delta_n \geq \delta\sqrt{2n}\Big)  
+\Prob\Big(\max_{[\lceil n(1-x) \rceil,n]} \bm Y_n \geq \eta\sqrt{2n}\Big)
+ o(1)\\
+ y
+\Prob\Big(\bm \tau_n \geq (1-2u)n\Big) 
+C_u\P\Big(|\conti B_u|\leq\eta\Big) 
+o_{u,\eta}(1).
\end{multline*}
The probability of each term is readily bounded as follows 
\begin{multline*}
\Prob\Big(\conti L_n^ {(\bm u)} (1) -\conti L_n^ {(\bm u)} (1-x) \geq \delta\Big) \\
\leq \Prob\Big(\max_{[0,y]} \conti X_n^*\geq \delta\Big) + \Prob\Big(\max_{[1-x,1]} \conti Y_n \geq \eta\Big)
+ o(1)\\
+y 
+\Prob\Big( \min_{[\bm V_n/n,1-2u]}\conti X^*_n \geq \conti X^*_n(y)\Big)
+C_u\P\Big(|\conti B_u|\leq\eta\Big) +o_{u,\eta}(1).
\end{multline*}
As both $\conti Y_n$ and $\conti X^*_n$ converge to Brownian excursions, we can make this estimate arbitrarily small for large $n$ upon choosing $y$ small enough, then $u$ small enough, then $\eta$ small enough, and then $x$ small enough. This proves the lemma.
\end{proof}

\appendix

\section{Walks in the two-dimensional non-negative quadrant}
\label{sec:appendix}

 \subsection{Statements of the technical results}
Let $\bm W=(\bm W_k)_{k\in \Z_{\geq 0}}$ be a two-dimensional random walk with step distribution $\nu$ (defined in \cref{eq:step_distribution_walk} page~\pageref{eq:step_distribution_walk}), started at a point $x\in\Z^2$. We denote this measure by $\Prob_x$. Let $\conti W = (\conti X,\conti Y)$ be a standard two-dimensional Brownian motion of correlation $-1/2$.
After the simple computation $\Var(\nu) = \begin{psmallmatrix}2 &-1 \\ -1 &2 \end{psmallmatrix}$, the classical Donsker's theorem implies that the process $\left(\frac 1 {\sqrt {2n}} \bm W_{\lfloor nt \rfloor}\right)_{t\in[0,1]}$ converges in distribution to the process $(\conti W_t)_{t\in[0,1]}$.
In this section we are interested in the behavior of $\bm W$ under the conditioning of starting and ending close to the origin, and staying in the non-negative quadrant $Q = \Z_{\geq 0}^2$.
This has been treated in much wider generality in \cite{MR3342657} and \cite{duraj2015invariance}, and specialized in \cite{bousquet2019plane} to families of walks with steps in $\Steps$ (defined in \cref{eq:admis_steps} page~\pageref{eq:admis_steps}).
The following convergence in distribution can be found in \cite{MR3945746}, as an immediate consequence of \cite[Theorem 4]{duraj2015invariance}.
\begin{proposition}\label{prop:DW}
	Let $x,y\in Q$. 
	Then
	\[\Prob_x\left(\left(\tfrac 1 {\sqrt{2 n}} \bm W_{\lfloor nt \rfloor}\right)_{0\leq t \leq 1} \in \cdot \;
	\middle| \; \bm W_{[0,n]}\subset Q, \bm W_n = y\right)
	\xrightarrow[n\to\infty]{} \Prob(\conti W_e \in \cdot),\]
	where $\conti W_e$ is some process that we call the two-dimensional Brownian excursion of correlation $-1/2$ in the non-negative quadrant.
\end{proposition}

We will now go through the initial steps of a slightly different proof of this result, one that highlights an absolute continuity phenomenon between a conditioned walk away of its starting and ending points and an unconditioned one. The two lemmas that we prove here (absolute continuity of the walk and local limit estimate of the density factor) are needed in this paper to show convergence of a coalescent-walk process driven by a conditioned random walk. 

In what follows, we recall that if $W = (X,Y)$ is a two-dimensional walk, then $\inf W = (\inf X, \inf Y)$.
We also use the \emph{hat} to denote reversal of coordinates, so that $\widehat{(i,j)} = (j,i)$.

 \begin{lemma}\label{lem:AbsCont}
 	Let $h:(\Z^2)^{n-2m+1}\to\R$ be a bounded measurable function. Let $x,y\in Q$ and $1\leq m<n/2$. Then
 	\begin{multline*}
 	\E_x[h((\bm W_{i+m}-\bm W_m)_{0\leq i \leq n-2m}) \mid \bm W_{[0,n]}\subset Q, \bm W_n = y]\\
 	= \E_0\left[
 	h(\bm W_i)_{0\leq i \leq n-2m}\cdot
 	\alpha_{n,m}^{x,y}\left(-\inf_{0\leq i \leq n-2m} \bm W_i\; ,\;\bm W_{n-2m}\right)
 	\right],
 	\end{multline*}
 	where
 	\begin{equation}\label{eq:alpha_tilting_function}
 	\alpha_{n,m}^{x,y}(a,b) = \sum_{z\in Q \colon z-a \in Q} \frac{
 		\Prob_x(\bm W_m = z, \bm{W}_{[0,m]}\subset Q)
 		\Prob_{\widehat y}(\bm W_m = \widehat z+ \widehat b, \bm{W}_{[0,m]}\subset Q)
 	}{\Prob_x(\bm W_n = y, \bm{W}_{[0,n]}\subset Q)}.
 	\end{equation}
 \end{lemma}
 
 \begin{lemma}\label{lem:LLT}
 	Fix $x,y\in Q$. For all $1/2>\eps>0$,
 	\begin{equation*}
 	\lim_{n\to\infty} \sup_{a\in\Z^2_{\geq 0} ,b\in \Z^2} \left\lvert\alpha^{x,y}_{n,\lfloor n\eps\rfloor}(a,b) - \alpha_{\eps}\left(\tfrac a{\sqrt {  n}},\tfrac b{\sqrt { n}}\right)\right\rvert = 0,
 	\end{equation*}
 	where $\alpha_\eps$ is a bounded continuous function on $(\R_+)^2\times \R^2$  defined by
 	\begin{equation}\label{eq:fuction_alpha}
 	\alpha_\eps(a,b) =\frac{1}{\eps^5}\frac {\sqrt 3}{8} \int_{x: x-a \in \R^2_+} g\left(\frac{x}{\sqrt{2\eps}}\right)g\left(\frac{x+b}{\sqrt{2\eps}}\right)dx
 	\end{equation}
 	and 
 	\begin{equation}\label{eq:fuction_g}
 	g(x_1, x_2)=\frac{1}{\sqrt{3 \pi}} x_1 x_2(x_1+x_2) \exp \left(-\frac{1}{3}\left(x_1^{2}+x_2^{2}+x_1 x_2\right)\right).
 	\end{equation}
 \end{lemma}
 
 A byproduct of this approach is a different characterization of the law of $\conti W_e$, which is immediate from \cref{prop:DW,lem:AbsCont,lem:LLT}.
 \begin{proposition}\label{prop:brown_ex}
 	For every $\eps>0$, the distribution of $(\conti W_e (\eps+t) - \conti W_e(\eps))_{0\leq t\leq 1-2\eps}$ is absolutely continuous with regards to the distribution of $\conti W_{|[0,1-2\eps]}$. 
 	%The density function is the map
 	%\begin{equation*}
 	%\mathcal{C}([0,1-2\eps],\R^2)\to \R,\qquad
 	%f\mapsto
 	%\frac{1}{\eps^5}\alpha\big(-\inf_{[0,1-2\eps]}f\; ,\; f(1-2\eps)\big),
 	%\end{equation*}
 	%provided that $-\inf_{[0,1-2\eps]}f\in\R^2_{+}.$
 	In particular, for every $\eps>0$ and for every integrable function  $h:\mathcal{C}([0,1-2\eps],\R^2)\to \R$,
 	\begin{equation*}
 	\E\Big[h\big((\conti W_e (\eps+t) - \conti W_e(\eps))_{0\leq t\leq 1-2\eps}\big)\Big]
 	= \E\Big[
 	h\big(\conti W_{|[0,1-2\eps]}\big)
 	\alpha_{\eps}\big(-\inf_{[0,1-2\eps]} \conti W\; ,\; \conti W(1-2\eps)\big)\Big].
 	\end{equation*}
 \end{proposition}

 \subsection{Proof of the technical results}
  \begin{proof}[Proof of \cref{lem:AbsCont}]
 	We write
 	\begin{multline*}	 	
	 	\idf\{\bm W_{[0,n]}\subset Q, \bm W_n = y\} = \idf\{\bm W_m + \inf_{m\leq i \leq n-m} (\bm W_i-\bm W_m) \geq (0,0) \} 
	 	  \cdot\idf\{\inf_{0\leq i \leq m} \bm W_i \geq (0,0)\} \\
	 	  \cdot\idf\{\inf_{0\leq i \leq m} (\bm W_{n-i}- \bm W_n + y) \geq (0,0)\}\cdot \idf\{(\bm W_{n-m}- \bm W_n + y) = \bm W_m + (\bm W_{n-m}-\bm W_m)\}.	
	\end{multline*}
	We introduce a decomposition over the values of $\bm W_m$, yielding
	\begin{multline*}	 	
	\idf\{\bm W_{[0,n]}\subset Q, \bm W_n = y\} = \sum_{z:z +  \inf_{m\leq i \leq n-m} (\bm W_i-\bm W_m) \geq (0,0)}
	\idf\{\bm W_m = z\} \cdot \idf\{\inf_{0\leq i \leq m} \bm W_i \geq (0,0)\} \\
	 \cdot  \idf\{\inf_{0\leq i \leq m} (\bm W_{n-i}- \bm W_n + y) \geq (0,0)\}\cdot  \idf\{(\bm W_{n-m}- \bm W_n + y) = z + (\bm W_{n-m}-\bm W_m)\} .	
	\end{multline*}
	Using the independence of increments of the random walk, along with the fact that $\bm W_{n-i} - \bm W_n$ is a random walk of step distribution $(x,y) \mapsto \nu(-x,-y) = \nu(y,x)$, we obtain 
	\begin{multline*}
	\Prob_x(\bm W_{[0,n]}\subset Q, \bm W_n = y\mid (\bm W_{i+m}-\bm W_m)_{0\leq i \leq n-2m})\\
	=\sum_{z\in Q \colon z+\inf_{0\leq i \leq n-2m} (\bm W_i) \in Q}
		\Prob_x(\bm W_m = z, \bm{W}_{[0,m]}\subset Q)
		\Prob_{\widehat y}(\bm W_m = \widehat z+ \widehat{\bm W_{n-2m}}, \bm{W}_{[0,m]}\subset Q|\bm W_{n-2m}).
	\end{multline*}
	From that we can conclude using \cref{eq:alpha_tilting_function} that
	\begin{multline*}
	\E_x[h((\bm W_{i+m}-\bm W_m)_{0\leq i \leq n-2m}) \mid \bm W_{[0,n]}\subset Q, \bm W_n = y]=\\
	\E_x\left[h((\bm W_{i+m}-\bm W_m)_{0\leq i \leq n-2m}) \frac{\Prob_x(\bm W_{[0,n]}\subset Q, \bm W_n = y\mid (\bm W_i - \bm W_m)_{m\leq i \leq n-m})}{\P_x\left(\bm W_{[0,n]}\subset Q, \bm W_n = y\right)}\right]\\
	= \E_0\left[
	h(\bm W_i)_{0\leq i \leq n-2m}\cdot
	\alpha_{n,m}^{x,y}\left(-\inf_{0\leq i \leq n-2m} \bm W_i\; ,\;\bm W_{n-2m}\right)
	\right].
	\end{multline*}
	This concludes the proof.
 \end{proof}

Finally we prove the estimate given in \cref{lem:LLT} for the density factor $\alpha_{n,m}^{x,y}(a,b)$. It relies on local limit results one can find in \cite[Propositions 8.2-3-6]{bousquet2019plane}. These are specializations of the results of \cite{MR3342657}.  We collect those estimates in the following lemma.
 \begin{lemma}\label{lem:techn_estimates}
 	Fix $x\in Q$. There exists a positive function $V$ on $Q$ such that as $n\to\infty$ the following asymptotics hold
 	\begin{gather}
 	\mathbb{P}_x\left(\bm W_{[0,n]} \subset Q\right) \sim \frac{1}{4 \sqrt{\pi}} V(x) n^{-3 / 2}\text{ as }n \rightarrow \infty, \label{eq:LLTconditioning}\\
 	\delta_1(x,n) \coloneqq \sup _{y \in Q}\left|n^{5 / 2} \cdot \mathbb{P}_x\left(\bm W_n=y, \bm W_{[0,n]} \subset Q\right)-\frac{ V(x)}{8 \sqrt{\pi}} g\left(\frac{y}{\sqrt{2n}}\right)\right| \rightarrow 0, 	\label{eq:LLTmeander}\\
 	\mathbb{P}_x\left(\bm W_n=y, \bm W_{[0,n]} \subset Q\right) \sim \frac{1}{8 \sqrt{3} \pi}\cdot \frac{V(x) V(\widehat y)}{n^{4}}, \label{eq:LLTexcursion}
 	\end{gather}
 	where $g$ was defined in \cref{eq:fuction_g} above.	
 \end{lemma} 
 
 From \cref{lem:techn_estimates}, the proof of \cref{lem:LLT} is similar to the proof of \cref{eq:LLTexcursion} from \cref{eq:LLTconditioning,eq:LLTmeander} in \cite[Proposition 8.3]{bousquet2019plane} or \cite[Theorem 5]{MR3342657}.

\begin{proof}[Proof of \cref{lem:LLT}]
	In what follows, $m = \lfloor n \eps \rfloor$ for some $\eps>0$. Let us consider $\alpha_{n,m}^{x,y}(a,b)$ defined in \cref{eq:alpha_tilting_function}.
	By \cref{eq:LLTexcursion}, the denominator (which is independent of $a,b$) 
	is of order $n^{-4}$. This is the scale to which we need to estimate the numerator.
	
We first deal with the infiniteness of the sum by cutting it off at $t\sqrt{n}$, for some $t>0$, and bound the remainder. Using \cref{eq:LLTmeander} for one factor (note that $g$ is bounded) and \cref{eq:LLTconditioning} for the other, there is a constant $C$ depending only on $x,y$ such that
	\begin{align*}
	R_n:= &\sum_{|z|>t\sqrt n} 
\Prob_x(\bm W_m = z, \bm W_{[0,m]} \subset Q)
\Prob_{\widehat y}(\bm W_m = \widehat z+ \widehat b, \bm W_{[0,m]}\subset Q)\\
	&\leq C n^{-3/2}n^{-5/2}\sum_{|z|>t\sqrt n} \Prob_x(\bm W_m = z \mid \bm W_{[0,m]} \subset Q)\\
	&= C n^{-4}\Prob_x(|\bm W_{ \lfloor n \eps \rfloor}| > t\sqrt n \mid \bm W_{[0, \lfloor n \eps \rfloor]} \subset Q).
	\end{align*}
	Thanks to the central limit theorem for $\bm W$ under the meander conditioning \cite[Proposition 8.5]{bousquet2019plane}, we can find a function $\delta_2(x,y,\eps,n,t)$ independent of $a,b$ such that
	\begin{equation} n^{4}R_n \leq \delta_2(x,y,\eps,n,t)\qquad \text{and}\qquad 
	\lim_{t\to\infty} \limsup_{n\to\infty} \delta_2(x,y,\eps,n,t) = 0 .
	\label{eq:remainder_proof_lemma}
	\end{equation}
	Now set
	\begin{align*}
	B_n \coloneqq \sum_{z:z-a\in Q ,|z|\leq t\sqrt n} 
	\Prob_x(\bm W_m = z, \bm W_{[0,m]} \subset Q)
	\cdot \Prob_{\widehat y}(\bm W_m = \widehat z+ \widehat b, \bm W_{[0,m]} \subset Q).
	\end{align*}
	Using \cref{eq:LLTmeander} and symmetry of $g$, we have for fixed $x$ and $y$ that
	\begin{multline*}
	B_n = m^{-5}\cdot \frac{ V(x)V(\widehat y)}{(8 \sqrt{\pi})^2}\sum_{z:z-a\in Q ,|z|\leq t\sqrt n}  g\left(\frac{z}{\sqrt{2\eps n}}\right)g\left(\frac{z+b}{\sqrt{2\eps n}}\right) \\
	+ O(1) (t\sqrt n)^2 m^{-5}(\delta_1(x,m) + \delta_1(\widehat y,m)).
	\end{multline*}	
	Collecting this estimate of the numerator with the estimate in \cref{eq:LLTexcursion} of the denominator, both uniform in $(a,b)$, we have
	\begin{align*}
	\alpha^{x,y}_{n,\lfloor \eps n \rfloor}(a,b) &= O(1)n^{4}R_n + o(1) + \frac {\sqrt 3+o(1)}{8\eps^5}\times \frac 1 {(\sqrt n)^2} \sum_{z \geq a,|z|\leq t\sqrt n}  g\left(\frac{z}{\sqrt{2\eps n}}\right)g\left(\frac{z+b}{\sqrt{2 \eps n}}\right)\\
	&= O(1)n^{4}R_n + o(1) + \frac {\sqrt 3+o(1)}{8\eps^5}\int_{w\geq \frac a {\sqrt n}, |w|\leq t} g\left(\frac{w}{\sqrt{2\eps }}\right)g\left(\frac{w+b/\sqrt n}{\sqrt{2\eps }}\right)dw + o(1), 
	\end{align*}
	where the last $o(1)$ corresponds to the uniform modulus of continuity of $g$ at the scale $n^{-1/2}$ resulting from Riemann summation. All error terms are uniform in $(a,b)$.	
	Finally, 
	\begin{align*}
	\left\lvert\alpha_{n,\lfloor n\eps\rfloor}(a,b) - \alpha_\eps\left(\tfrac 1{\sqrt n}a,\tfrac 1{\sqrt n}b\right)\right\rvert 
	&= O(1)n^{4}R_n + o(1) + O(1)\int_{|w|> t} g(w/\sqrt{2 \eps}).
	\end{align*}
	By integrability of $g$ and \cref{eq:remainder_proof_lemma}, this last term can be made negligible by taking $n\to\infty$ and then $t\to\infty$.
\end{proof}

\section{Generalizations}
\label{sec:general}

\subsection{A coalescent-walk process for separable permutations}\label{sect:sep_perm}

We return here to the class of separable permutations, a well-known subclass of the Baxter permutations, defined by avoidance of the two classical patterns $2413$ and $3142$. As we already mentioned in the introduction, the scaling limit of this class of permutations, called the \textit{Brownian separable permuton}, was introduced in \cite{bassino2018brownian}. We also point out that the mapping $\bobp$ puts separable permutations in bijection with \textit{rooted series-parallel non-separable maps} \cite[Prop. 6]{MR2734180}. 

In this section, we explain an encoding of separable permutations by a discrete coalescent-walk process, different from the one given by $\wcp \circ \bow \circ \bobp^{-1}$, but more suitable for our purposes. We will also present what we believe to be the scaling limit of this coalescent-walk process, and relate it to the construction of the Brownian separable permuton given in \cite{maazoun}.

We first recall another definition of separable permutations more suited to our goals. A \emph{signed tree} $t$ is a rooted plane tree whose internal vertices are decorated with signs in $\{\oplus, \ominus\}$ (see the first picture of \cref{fig:separable} for an example).
We label its leaves with the integers from $1$ to $k$ according to the exploration process of $t$.
The signs can be interpreted as coding a different ordering of the rooted tree $t$: we call $\tilde t$ the tree obtained from $t$ by reversing the order of the children of each vertex with a minus sign (see the second picture of \cref{fig:separable}).
The order of the leaves is changed by this procedure, and we set $\sigma(i)$ to be the position in $\tilde t$ (w.r.t.\ its exploration process) of the leaf $i$.
We call $\mathrm{perm}(t)$ this permutation $\sigma \in \Perms_k$ (see the third picture of \cref{fig:separable}). Separable permutations are exactly the ones  obtained from a signed tree through this procedure.

We now introduce the discrete coalescent-walk process associated with a separable permutation. Let $t$ be a signed tree with $k$ leaves and $e$ edges, and let $C = (C_0,\ldots,C_{2e})$ be its contour function. For every $j\in[1,2e-1]$ which is a local minimum of $C$, we denote $s_j$ the sign of the internal vertex of $t$ which is visited by $C$ at time $j$. For every $i$ which is a local maximum of $C$ (a visit-time of a leaf of $t$), we construct a walk $Z^{(i)}$ starting at time $i$ at $0$, i.e.\ $Z^{(i)}_i = 0$, and that stays equal to zero until time $\ell_i$, where $\ell_i$ is the first local minimum of $C$ after time $i$. The walk is then defined inductively by the following: for all $\ell_i \leq j\leq 2e-1$,
	\begin{equation*}
	Z^{(i)}_{j+1} -Z^{(i)}_{j} \coloneqq 
	\begin{cases}
	(C_{j+1}- C_{j}), &\text{ if }\quad Z^{(i)}_j>0,\\
	- (C_{j+1}- C_{j}), &\text{ if }\quad Z^{(i)}_j<0,\\
	-1,&\text{ if }\quad Z^{(i)}_j=0 \text{ and }s_j = \oplus,\\
	1,&\text{ if }\quad Z^{(i)}_j=0 \text{ and }s_j = \ominus,\\
	0,&\text{ if }\quad Z^{(i)}_j=0 \text{ and }j\text{ is not a local minimum of }C.
	\end{cases}
	\end{equation*}
	
	\begin{figure}
	\centering
		\includegraphics[scale=1.2]{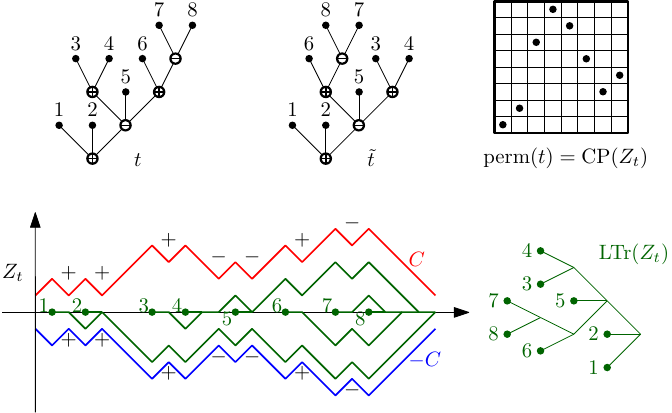}
		\caption{An example of a coalescent-walk process driven by the signed excursion of a signed tree. \label{fig:separable}}
	\end{figure}

We set $Z_t = \{Z^{(i)}, i\text{ local maximum of }C\}$, that is the coalescent-walk process associated with the separable permutation $\mathrm{perm}(t)$ (see the fourth picture of \cref{fig:separable} for an example). We observe that $Z_t$ is a coalescent-walk process on $[0,2e]$ in the sense of \cref{def:discrete_coal_process}, except that the trajectories do not start at every point of the interval, which is irrelevant to the rest of the discussion.

We leave to the reader the following observation (similar to \cref{prop:eq_trees}) that justifies the construction of $Z_t$. We denote by $\labtree(Z_t)$ the labeled tree induced by the trajectories of the coalescent-walk process $Z_t$ (see the fifth picture of \cref{fig:separable} for an example).
\begin{observation}
	The tree $\labtree(Z_t)$ is the same as the tree $\widetilde t$ (forgetting the signs). Consequently, using \cref{prop:fortree_cpbp}, we have that $\cpbp(Z_t) = \mathrm{perm}(t)$.
\end{observation}

Hence we see that separable permutations may be constructed using a coalescent-walk process driven by the discrete contour function (and its reflection) of a signed tree. 

\bigskip
An \emph{alternating Schröder tree} is a signed tree with no vertices of outdegree one, and the additional property that signs alternate along ancestry lines. A uniform separable permutation of size $n$ corresponds to a uniform alternating signed tree with $n$ leaves  \cite[Proposition 2.13]{bassino2018brownian}. Upon rescaling, the contour function of the uniform alternating Schröder tree converges to a Brownian excursion \cite[Proposition 2.23]{bassino2018brownian}. A small leap of faith leads us to believe that the scaling limit of the discrete coalescent-walk process should be the continuous coalescent-walk process given by the following family of SDEs defined for all $u\in[0,1]$,

\begin{equation} \label{eq:tanaka}
\begin{cases}
d\conti Z^{(u)}(t) = \sgn(\conti Z^{(u)}(t)) \:d \bm e(t), &0<u\leq t
\leq 1,\\
\conti Z^{(u)}(u) = 0,
\end{cases}
\end{equation}
where $\bm e$ is a Brownian excursion on $[0,1]$. This is a variation of
the well-known \textit{Tanaka SDE}
\begin{equation} \label{eq:standard_tanaka}
d\conti Z(t) = \sgn(\conti Z(t)) \:d \conti B(t).
\end{equation}
where $\conti B$ is a standard Brownian motion.
The characteristic feature of this equation is the absence of pathwise
uniqueness: solutions cannot be measurable functions of the driving
process $\conti B$ and must also incorporate additional randomness. For
instance, taking for simplicity $\conti Z(0) = 0$ (see exercise IX.1.19
in \cite{revuz2013continuous}) one solves by taking $\conti Z$ to be a
standard Brownian motion and $\conti B (t) = \int_0^t \sgn(\conti Z(t))
d \conti Z(t)$. Then $\sigma(\conti B) \subseteq \sigma(|\conti Z|)$
hence $\sigma(\conti Z) \nsubseteq \sigma(\conti B)$.
This absence of pathwise uniqueness raises many questions when one
wishes to couple the solutions of \cref{eq:standard_tanaka} for
different starting times and points. An elegant way of doing so is with
the notion of \textit{stochastic flow of maps} developed by Le Jan and
Raimond in \cite{MR2060298}. The same authors show in \cite{MR2235172}
that there exists a unique flow of maps solving Tanaka's SDE
(\cref{eq:standard_tanaka}), that this flow is \textit{coalescent}, and
explicitly construct it (see in particular the proof of Theorem 2.1).

Transposed in our setting,\footnote{Here the driving motion is a
	Brownian excursion $\bm e$ and not a Brownian motion $\conti B$. Results
	can be transposed from one setting to the other by absolute continuity,
	as in \cref{sec:cond_conv}.}
their construction gives the following solution of \cref{eq:tanaka}.
Conditional on $\bm e$,  choose a uniform element $\bm s(\ell)
\in\{+1,-1\}$ independently for every $\ell\in (0,1)$ that is a local
minimum\footnote{
	For the technicalities involved in indexing an i.i.d.\ sequence by
	this random countable set, see \cite{maazoun}.}
of $\bm e$.
For $0\leq u \leq t \leq 1$, set $\bm m(u,t) \coloneqq \inf_{[u,t]} \bm
e$, and $\bm \mu(u,t) =\sup\{s\in [u,t]: \bm e(s) = \bm m(u,t)\}$. Then define
\[
\conti Z^{(u)}(t) \coloneqq (\bm e(t) - \bm m(u,t))\bm s( \bm \mu(u,t)).
\]
This construction is reminiscent of the discrete setting. Moreover we
leave to the reader the following:
\begin{observation}\label{obs:separable_permuton}
	Let $\mu_{\conti Z}$ be the permuton built (as in \cref{sec:final})
	from the continuous coalescent-walk process $\conti Z$ defined by
	\cref{eq:tanaka} . Then $\mu_{\conti Z}$ coincides with the permuton
	constructed from $(\bm e,\bm s)$ in \cite{maazoun}. In particular, it
	has the distribution of the Brownian separable permuton.
\end{observation}
This shows that the Brownian separable permuton falls in the framework
of continuous coalescent-walk processes. We believe an approach similar
to the one of this paper would be doable to rigorously prove the
convergence from discrete to continuous coalescent-walk processes.

\medskip

As separable permutations form a subset of the Baxter permutations, another route would be to specialize the bijections given for Baxter permutations in \cref{thm:diagram_commutes}. Let $\bm \sigma$ be a uniform separable permutation of size $n$, $(\bm X,\bm Y) = \bow \circ \bobp^{-1}(\bm \sigma)$ and $\bm Z = \wcp(\bm X,\bm Y)$. Simulations lead us to believe that when the size of $\bm \sigma$ is large, both $\bm X$ and $\bm Y$ concentrate around the contour function of the alternating signed tree coding $\bm \sigma$. We also believe that the discrete coalescent-walk process $\bm Z$ should converge, in the limit, to the continuous coalescent-walk process $\conti Z$ defined by the Tanaka's SDEs in \cref{eq:tanaka}.

\subsection{Liouville quantum gravity and mating-of-trees}
\label{sec:lqg_sde}

As mentioned in the introduction, uniform infinite bipolar triangulations were studied in \cite{GHS}.
Their key lemma is \cite[Proposition 4.1]{GHS}, which is similar to our \cref{thm:coal_con_uncond}. To state it using our notation, let $(\tilde{\conti X}_n, \tilde{\conti Y}_n)$ be the continuous rescaling (at scale $n$) of the bi-infinite two-dimensional random walk defining (through $\Inverse$) a uniform infinite-volume bipolar triangulation, and $\tilde{\conti Z}_n$ the continuous rescaling of the corresponding coalescent-walk process\footnote{The uniform infinite-volume bipolar triangulation is properly defined in \cite[Section 1.3.3]{GHS}. The corresponding bi-infinite two-dimensional random walk is denoted by $\mathcal{Z}=(\mathcal{L}_n,\mathcal{R}_n)_{n\in\Z}$ in their article. The trajectory of the coalescent-walk process starting at time $0$ is denoted by $\mathcal X$.}. Then \cite[Proposition 4.1]{GHS} states that\footnote{The process $(\tilde{\conti X},\tilde{\conti Y},\tilde{\conti Z}^{(0)})$ is denoted $(L,R,X)$ in \cite{GHS}.}
\begin{equation}\label{eq:GHS}
(\tilde{\conti X}_n,\tilde{\conti Y}_n, \tilde{\conti Z}^{(0)}_n) \xrightarrow[n\to\infty]{d} \left(\tilde{\conti X},\tilde{\conti Y},\tilde{\conti Z}^{(0)}\right).
\end{equation}  
The process $\left(\tilde{\conti X},\tilde{\conti Y},\tilde{\conti Z}^{(0)}\right)$ on the right-hand side of the equation above is a special case (for  $\kappa' = 12$ and $\theta = \pi/2$) of the general process $\left(\tilde{\conti X}_{\kappa'},\tilde{\conti Y}_{\kappa'},\tilde{\conti Z}_{\kappa',\theta}^{(0)}\right)$ defined in terms of two parameters $\kappa' \in (4,\infty)$ and $\theta\in [0,2\pi)$ that they construct using Liouville quantum gravity, imaginary geometry and mating of trees. We can describe it roughly as follows: 
let $\bm \mu$ be a $\sqrt{16/\kappa'}$-LQG quantum plane, 
 $\bm h$ be a Gaussian free field independent of $\bm \mu$, and $\bm \eta$ be the space-filling SLE${}_{\kappa'}$ curve of angle zero generated (in the sense of imaginary geometry) by $\bm h$.
\begin{itemize}
	\item The process $(\tilde{\conti X}_{\kappa'},\tilde{\conti Y}_{\kappa'})$ is a standard two-dimensional Brownian motion of correlation $\rho=-\cos(4\pi/\kappa')$ given by the \emph{mating-of-tree} encoding of $(\bm \mu,\bm \eta)$.
	\item The process $\tilde{\conti Z}_{\kappa',\theta}^{(0)}$ tracks, in some sense, the interaction between $\bm \eta$ and another SLE${}_{16/\kappa'}$ curve of angle $\theta$ also generated by $\bm h$.
\end{itemize}
Gwynne, Holden and Sun prove that there exists a constant\footnote{The explicit expression of $p(\kappa',\theta)$ is not known.} $p = p(\kappa',\theta)$, with $p(\kappa',\pi/2) \equiv 1/2$, so that $\tilde{\conti Z}_{\kappa',\theta}^{(0)}$ is a skew Brownian motion of parameter $p$ (see \cite{MR2280299}) and they describe the conditional distribution of $(\tilde{\conti X}_{\kappa'},\tilde{\conti Y}_{\kappa'})$ given $\tilde{\conti Z}^{(0)}_{\kappa',\theta}$ (see \cite[Proposition 3.2]{GHS}). 
They also note that $\tilde{\conti Z}_{\kappa',\theta}^{(0)}$ is a measurable functional of $(\bm \mu,\bm h)$, which turns to be completely determined by $(\tilde{\conti X}_{\kappa'},\tilde{\conti Y}_{\kappa'})$. Nevertheless, they do not make explicit the measurable mapping $(\tilde{\conti X}_{\kappa'},\tilde{\conti Y}_{\kappa'})\mapsto \tilde{\conti Z}_{\kappa',\theta}^{(0)}$.

\medskip

In the case $\kappa' = 12$ and $\theta = \pi/2$, after comparing\footnote{More precisely one would need to show that \cref{thm:coal_con_uncond} works also in the case of bipolar triangulations. This is a special case of the generality conjecture of \cref{sec:perspectives}, about which we are very confident.} \cref{eq:GHS} and \cref{eq:coal_con_uncond} in \cref{thm:coal_con_uncond}, one sees that our approach provides the explicit mapping $(\tilde{\conti X},\tilde{\conti Y})\mapsto \tilde{\conti Z}^{(0)}$ through solving the SDE \eqref{eq:flow_SDE_intro} driven by $(\tilde{\conti X},\tilde{\conti Y})$ for $u=0$.

\medskip

For general $\kappa' \in (4,\infty)$ and $\theta\in [0,2\pi)$, we believe that a related SDE  provides the explicit mapping $(\tilde{\conti X}_{\kappa'},\tilde{\conti Y}_{\kappa'})\mapsto \tilde{\conti Z}_{\kappa',\theta}^{(0)}$.
More precisely, let $\conti Z^*_{\kappa',\theta}=\{\conti Z_{\kappa',\theta}^{*(u)}\}_{u\in \R}$ be defined by the solutions of the following SDEs for $u\in\R$,
\begin{equation}\label{eq:generalized}
\begin{cases}
d\conti Z_{\kappa',\theta}^{*(u)}(t) = \idf_{\{\conti Z_{\kappa',\theta}^{*(u)}(t)> 0\}} d\conti Y^*_{\kappa'}(t) - \idf_{\{\conti Z_{\kappa'}^{*(u)}(t)\leq 0\}} d\conti X_{\kappa'}^*(t) +(2p-1)d\conti L_{\kappa',\theta}^{*(u)}(t), \quad &t>u,\\
\conti Z_{\kappa',\theta}^{*(u)}(t) = 0, \quad &t\leq u,
\end{cases}
\end{equation}
where  $p = p(\kappa',\theta)$ is the constant mentioned above, $(\conti X_{\kappa'}^*,\conti Y_{\kappa'}^*)$ is a standard two-dimensional Brownian motion of correlation $\rho$ with $\rho = - \cos(4\pi/\kappa')$, and $\conti L_{\kappa',\theta}^{*(u)}(t)$ is the local time at zero of $\conti Z_{\kappa',\theta}^{*(u)}$ accumulated during the time interval $[u,t]$.

\begin{conjecture}
	The SDE in \cref{eq:generalized} admits existence and pathwise uniqueness of the strong solution.
\end{conjecture}

\begin{conjecture}
	 For all $\kappa' \in (4,\infty)$ and $\theta\in [0,2\pi)$, the following equality in distribution holds 
	 $$\left(\conti X_{\kappa'}^*,\conti Y_{\kappa'}^*,\conti Z_{\kappa',\theta}^{*(0)}\right)\stackrel{d}{=}\left(\tilde{\conti X}_{\kappa'},\tilde{\conti Y}_{\kappa'}, \tilde{\conti Z}_{\kappa',\theta}^{(0)}\right).$$
	 In particular the SDE \eqref{eq:generalized} for $u=0$ explicitly defines the mapping $(\tilde{\conti X}_{\kappa'},\tilde{\conti Y}_{\kappa'})\mapsto \tilde{\conti Z}_{\kappa',\theta}^{(0)}$.
\end{conjecture}
In support of this conjecture, we point out that a local time term appears in the analysis of the case $\kappa' = 16$ and $\theta = \pi/3$ in \cite{LSW}. Moreover, the Lévy's characterization theorem guarantees that $\conti Z_{\kappa',\theta}^{*(u)} - (2p-1)\conti L_{\kappa',\theta}^{*(u)}$ is a Brownian motion, so that $\conti Z_{\kappa',\theta}^{*(u)}$ is indeed a skew Brownian motion of parameter $p$. 

In the case $p=1/2$, we recover the perturbed Tanaka SDE \eqref{eq:SDE} of \cite{MR3098074,MR3882190}, which has pathwise unique solutions for $\rho \in (-1,1)$. The edge case $\rho=1$ (i.e.\ $\kappa' = 4$, the underlying geometry being the critical $2$-LQG) corresponds to the Tanaka SDE \eqref{eq:tanaka} which does not have pathwise uniqueness.

When $\rho=1$ (i.e.\ $\kappa' = 4$) but $p\neq 1/2$, going back to the finite-volume case and denoting by $\bm e_*$ a one-dimensional Brownian excursion on $[0,1]$, we obtain the following SDEs defined for all $u\in\R$,
\begin{equation}\label{eq:skew_tanaka}
d\conti Z_{4,\theta}^{*(u)}(t) = \sgn(\conti Z_{4,\theta}^{*(u)}(t)) d \bm e_*(t) + (2p-1)d\conti L_{4,\theta}^{*(u)}(t),\quad  t\geq u,
\end{equation}
which we believe to give rise (using the same procedure described above \cref{defn:Baxter_perm}) to the \textit{biased Brownian separable permuton of parameter $1-p$}, in the terminology of \cite{bassino2017universal,maazoun}.
The opposite edge case $\rho = -1$ (i.e.\ $\kappa' = \infty$, the underlying geometry being $0$-LQG, that is Euclidean geometry) is Harrison and Shepp's equation defining skew Brownian motion, whose solutions are pathwise unique \cite{MR2280299}, and whose coalescing flow was studied by Burdzy and his coauthors (see \cite{MR2094439} and the references therein). 

Although the cases $p\neq 1/2$ and $\rho \neq -1$ are not present in the literature, we expect pathwise uniqueness of \cref{eq:generalized} to hold for every $p\in [0,1]$ and $\rho \in [-1,1)$. 

\section{Simulations of large Baxter permutations}\label{sect:simulations}

The simulations for Baxter permutations presented in the first page of this paper have been obtained in the following way:
\begin{enumerate}
	\item first, we have sampled a uniform random walk of size $n+2$ in the non-negative quadrant starting at $(0,0)$ and ending at $(0,0)$ with increments in the set $A$ (defined in \cref{eq:admis_steps} page~\pageref{eq:admis_steps}). This has been done using a ``rejection algorithm": it is enough to sample a walk $W$ starting at $(0,0)$ with increments distribution given by \cref{eq:walk_distrib}, up to the first time it leaves the non-negative quadrant. Then one has to check if the last step inside the non-negative quadrant is at the origin $(0,0)$. When this is the case (otherwise we resample a new walk), the part of the walk $W$ inside the non-negative quadrant, denoted $\widetilde W$, is a uniform walk conditionally to its size in the non-negative quadrant starting at $(0,0)$ and ending at $(0,0)$ with increments in the set $A$.
	\item Removing the first and the last step of $\widetilde W$, thanks to \cref{prop:unif_law}, we obtained a uniform random walk in $\mathcal W_n$.
	\item Finally, applying the mapping $\cpbp\circ\wcp$ to the walk given by the previous step, we obtained a uniform Baxter permutation of size $n$ (thanks to \cref{thm:diagram_commutes}).
\end{enumerate}
Note that our algorithm gives a random Baxter permutation which, conditioned on its size to be equal to $n$, is uniformly distributed among all Baxter permutations of size $n$.
\bibliographystyle{alpha}
\bibliography{bibli}

\end{document}